\documentclass[a4paper]{amsart}
\usepackage[utf8]{inputenc}
\usepackage{amsmath,amssymb}
\usepackage{color}
\usepackage{esint}
\usepackage{bbm}
\usepackage{mathrsfs}
\usepackage{tikz}
\usetikzlibrary{shapes,positioning}
\usepackage{enumerate}
\usepackage[pdfborder={0 0 0}]{hyperref}

\newtheorem{theorem}{Theorem}[section]
\newtheorem{lemma}[theorem]{Lemma}
\newtheorem{proposition}[theorem]{Proposition}
\newtheorem{corollary}[theorem]{Corollary}

\theoremstyle{definition}
\newtheorem{definition}{Definition}[section]

\theoremstyle{remark}
\newtheorem{remark}[definition]{Remark}

\numberwithin{equation}{section}

\newcommand\set[1]{\left\{\,#1\,\right\}}		% set
\newcommand\abs[1]{\left|#1\right|}				% modulus
\newcommand\norm[1]{\left\Vert#1\right\Vert}	% norm
\newcommand{\lamax}{\lambda_{\text{max}}}       % lambda_max
\newcommand{\lamin}{\lambda_{\text{min}}}       % lambda_min
\newcommand{\R}{\mathbb{R}}

\newcommand{\N}{\mathbb{N}}

\newcommand{\cA}{{\mathcal A}}
\newcommand{\cB}{{\mathcal B}}
\newcommand{\cC}{{\mathcal C}}
\newcommand{\cD}{{\mathcal D}}
\newcommand{\cE}{{\mathcal E}}
\newcommand{\cF}{{\mathcal F}}

\newcommand{\cL}{{\mathcal L}}

\newcommand{\cS}{{\mathcal S}}

\newcommand{\cV}{{\mathcal V}}

\newcommand{\fF}{{\mathfrak F}}

\DeclareMathOperator{\dist}{dist}				% distance
\DeclareMathOperator{\sign}{sign}				% signum
\DeclareMathOperator{\id}{id}					% identity
\DeclareMathOperator{\supp}{supp}				% support
\DeclareMathOperator{\tr}{tr}					% trace
\DeclareMathOperator{\divv}{div}				% div
\DeclareMathOperator{\osc}{osc}				% div

\DeclareMathOperator*{\esslim}{ess\,lim}

\title[On a degenerate elliptic problem]{On a degenerate elliptic problem arising in the least action principle for Rayleigh-Taylor subsolutions}
\author{bj\"orn gebhard, jonas hirsch, j\'ozsef j. kolumb\'an}
\begin{document}
\maketitle
\begin{abstract}
    We address a degenerate elliptic variational problem arising in the application of the least action principle to averaged solutions of the inhomogeneous Euler equations in Boussinesq approximation emanating from the horizontally flat Rayleigh-Taylor configuration. We give a detailed derivation of the functional starting from the differential inclusion associated with the Euler equations, i.e. the notion of an averaged solution is the one of a subsolution in the context of convex integration, and illustrate how it is linked to the generalized least action principle introduced by Brenier in \cite{Brenier89,Brenier18}. Concerning the investigation of the functional itself, we use a regular approximation in order to show the existence of a minimzer enjoying partial regularity, as well as other properties important for the construction of actual Euler solutions induced by the minimizer. Furthermore, we discuss to what extent such an application of the least action principle to subsolutions can serve as a selection criterion.
\end{abstract}

\section{Introduction}\label{sec_introduction}
Since its introduction in 2009 by De Lellis and Sz\'ekelyhidi \cite{DeL-Sz-Annals} to the context of fluid dynamics, the method of convex integration has been a powerful tool to show ill-posedness of initial value problems and to provide counterexamples to the conservation of physical quantities in a low enough regularity regime, we refer to \cite{Buckmaster_Vicol_survey,DeL-Sz-Survey} for recent surveys. Besides being an engine for counterexamples, due to their highly oscillatory nature, the solutions obtained by convex integration have shortly after also been utilized to describe turbulent behavior in situations where a regular solution simply can not exist due to irregular initial data. Examples include vortex sheets in the homogeneous two-dimensional Euler equations \cite{Mengual_Sz_vortex_sheet,Sz-KH}, as well as the Muskat problem for the incompressible porous media equation \cite{Castro_Cordoba_Faraco,Castro_Faraco_Mengual,Castro_Faraco_Mengual_2,Cordoba,Foerster-Sz,Hitruhin-Lindberg,Mengual,Noisette-Sz,Sz-Muskat}, and the horizontally flat Rayleigh-Taylor instability in the inhomogeneous Euler equations \cite{GKBou,GKSz}. 

The existence of solutions emanating in the stated situations relies on a general convex integration theorem for the corresponding system, saying that a subsolution, which can be seen as an averaged solution, induces infinitely many turbulent solutions that are close in a weak sense to the subsolution, see Theorem \ref{thm:CI} below for example. Therefore, having such a theorem at hand, it remains to construct a suitable subsolution. However, typically  there are plenty of admissible subsolutions emanating from the same initial data, i.e. also on the level of averaged motion a vast amount of evolutions differing for instance in the quantitative size of the induced turbulence of the solutions is possible. One therefore has to choose a particular subsolution based on a meaningful selection criterion. Up to now essentially two strategies have been used: reduction after some ansatzes to a hyperbolic conservation law and selection of the unique entropy solution \cite{GKSz,Sz-KH,Sz-Muskat}, and in the case of the Euler equations  short time selection by means of maximal initial energy dissipation \cite{GKBou,Mengual_Sz_vortex_sheet}. 

In the case of the incompressible porous media equation with flat initial configuration the former strategy has been applied by Sz\'ekelyhidi \cite{Sz-Muskat}, yielding a family of subsolutions out of which the one with ``maximal mixing'' coincides with the unique solution of a different relaxation given by Otto \cite{otto-ipm} based on gradient flows. Regarding non-flat initial interfaces local in time subsolutions of different types have been constructed in \cite{Castro_Cordoba_Faraco,Foerster-Sz,Noisette-Sz}, also in the only partially unstable case \cite{Castro_Faraco_Mengual_2}. Properties of the subsolutions selected in the flat case by either of the above mentioned relaxations have been incorporated in these constructions, see for instance \cite[Remark 4.2]{Castro_Cordoba_Faraco}.

The latter criterion, i.e. maximal initial energy dissipation in the case of the Euler equations, was motivated by the entropy rate admissibility criterion of Dafermos \cite{Dafermos_entropy_rate} and before also discussed in the context of solutions obtained by convex integration for the compressible Euler equations, \cite{Chiodaroli_Kreml_energy_dissipation,Feireisl}.  

Motivated by the search for a global in time selection criterion and the well-known fact that the Euler equations can formally be derived from the least action principle, see Section \ref{sec_action_functional} for more details, the present article originates from the question of what happens if one imposes the least action principle on the level of subsolutions.

We follow this question in the case of the Euler equations in Boussinesq approximation
\begin{align}\label{eq:bou}
\begin{split}
\partial_t v +\divv (v\otimes v) +\nabla p&=-\rho gA e_n,\\
\divv v&=0,\\
\partial_t \rho + \divv (\rho v)&=0,
\end{split}
\end{align}
stated on $(0,T)\times\mathcal D $ with $T>0$, $\cD:=(0,1)^{n-1}\times(-L,L)\subset\R^n$ and with initial data
\begin{align}\label{eq:initial_data}
\rho(0,x)=\sign(x_n),\quad v(0,x)=0,\quad x\in\cD.
\end{align}
On the boundary of $\cD$ we set
the usual no-penetration boundary condition
\begin{equation}\label{eq:boundary_condition}
v\cdot \nu =0 \quad \text{on }\partial\cD\times[0,T),
\end{equation}
in order to complement the incompressibility condition. Here $\nu$ denotes the exterior unit normal of $\partial\cD$.

System \eqref{eq:bou},\eqref{eq:initial_data},\eqref{eq:boundary_condition} models two incompressible ideal fluids with homogeneous densities $0<\rho_-<\rho_+$, with $\rho_+-\rho_-$ small, initially at rest and separated by a horizontally flat interface under the influence of gravity, i.e. in one of the most classical occurrences of the Rayleigh-Taylor instability. The unknowns are the normalized fluid density $\rho:[0,T)\times\mathcal D \rightarrow\R$, i.e. $\rho\in\{\pm 1\}$ a.e., the velocity field $v:[0,T)\times\mathcal D\rightarrow\R^n$ and the pressure of the fluid $p:[0,T)\times\mathcal D\rightarrow\R$. Furthermore, $e_n\in\R^n$ is used to denote the $n$th coordinate vector, $g>0$ the gravitational constant and 
$
A:=\frac{\rho_+-\rho_-}{\rho_++\rho_-}
$ is the Atwood number.

The article splits into two essentially independent parts. The first part consisting of Section \ref{sec_action_functional} and complemented by appendices \ref{sec_brenier}, \ref{sec_convex_integration} addresses the application of the least action principle.
In Section \ref{sec_action_functional} we first of all recall the relaxation of \eqref{eq:bou} seen as a differential inclusion and then illustrate how the least action principle imposed on a suitable class of one-dimensional subsolutions gives rise to a variational problem. In appendix \ref{sec_brenier} we show how this problem relates to the relaxation of the least action principle by Brenier \cite{Brenier89,Brenier18} in terms of generalized flows. Appendix \ref{sec_convex_integration} contains details regarding an adaptation to our particular setting (in terms of integrability) of the usual convex integration result, stated in Theorem \ref{thm:CI}. 

The variational problem derived in Section \ref{sec_action_functional} is of the type
\begin{align}\label{eq:variational_problem_in_introduction}
    \text{minimize }\cA(u)\text{ for } u\in X,
\end{align}
where the functional $\cA$ is given by
\begin{align}\label{eq:functional_in_introduction}
    \cA(u)=\int_\Omega F(\nabla u)-V(x_2,u)\:dx
\end{align}
with $\Omega:=(0,T)\times(-L,L)\subset\R^2$, $F:\R^2\rightarrow [0,+\infty]$,
\begin{align}\label{eq:definition_of_F_in_introduction}
F(p):=\begin{cases}
0,&\text{if }p_1=0,\\
+\infty,&\text{if }p_1\neq 0,~\abs{p_2}\geq 1,\\
\frac{p_1^2}{2(1-p_2^2)},&\text{else},
\end{cases}
\end{align}
and $V:[-L,L]\times \R\rightarrow \R$, $(x_2,z)\mapsto V(x_2,z)$ being a suitable nonlinear potential. To give an example the reader may think of 
\begin{align}\label{eq:example_for_V_in_introduction}
    V(x_2,z)=-gAz+\frac{3gA}{4L}(z-(\abs{x_2}-L))^2.
\end{align}
The (affine) space of functions under consideration reads
\begin{align}\label{eq:definition_of_X_in_introduction}
    X:=\set{u\in H^1(\Omega):u(x_1,\pm L)=0,~u(0,x_2)=-u(T,x_2)=\abs{x_2}-L}.
\end{align}

The functional $\cA$ is elliptic but degenerate in the sense that the minimal eigenvalue of $D^2F(p)$ is vanishing for $p_1=0$, while the maximal eigenvalue of $D^2F(p)$ becomes infinite as $\abs{p_2}$ approaches $1$ from below. The functional will be further introduced and investigated in Sections \ref{sec_main_result}-\ref{sec_further_properties}, which render the second part of our paper. A reader only interested in the variational problem itself might directly jump to these sections, which are readable without the background given in Section \ref{sec_action_functional}. We like to point out though that some aspects of our investigation in particular in Section \ref{sec_further_properties} stem their motivation from the usability of the minimizer as a subsolution for the Boussinesq system \eqref{eq:bou},\eqref{eq:initial_data},\eqref{eq:boundary_condition}.

The two parts of our investigation are brought together again in the final Section \ref{sec_summary_questions} where we summarize our results and discuss some open problems. Here a rough version of our main theorem can be written as follows.
\begin{theorem}\label{thm:introduction}
Under suitable conditions on the potential $V$, problem \eqref{eq:variational_problem_in_introduction} has a minimizer $u\in X\cap\cC^0(\overline{\Omega})$ enjoying $\cC^2$ regularity on a nonempty open set $\Omega'\subset \Omega$ of which every connected component is simply connected. On $\Omega'$ there holds $\partial_{x_1}u>0$, $\abs{\partial_{x_2} u}<1$, while outside $\Omega'$ we have $\partial_{x_1}u=0$, $\abs{\partial_{x_2}u}\leq 1$ a.e..
The gradient of $u$ can be used to define a one-dimensional subsolution of \eqref{eq:bou},\eqref{eq:initial_data},\eqref{eq:boundary_condition} and thus induces infinitely many solutions via convex integration.
\end{theorem}
As said, a more detailed version including statements on energy dissipation and attainment of initial and boundary data can be found in Section \ref{sec_summary_questions}. There we also summarize the ansatzes we make and reflect upon our initial question regarding a selection criterion for subsolutions emanating from Rayleigh-Taylor initial data. 

%The short answer at this point is that the search for a truly satisfactory criterion remains open. 

%\comm{underselling it a bit?}

\section{The action functional for subsolutions}\label{sec_action_functional}

The main point of this section is the derivation of the variational problem \eqref{eq:variational_problem_in_introduction}. We begin with a short review of related previous work for equation \eqref{eq:bou} before recalling the notion of a subsolution.

\subsection{The Boussinesq system}
For sufficiently regular initial data, system \eqref{eq:bou}, \eqref{eq:boundary_condition} is locally well-posed, see \cite{Chae-Nam,Danchin,Elgindi}, where in \cite{Elgindi} it also has been shown that finite time singularity formation occurs for smooth data.
We note however that the initial data of our interest \eqref{eq:initial_data}, with $\rho_0$ being only essentially bounded, does not fall into the regularity classes considered in \cite{Chae-Nam,Danchin,Elgindi}.

Contrary to local well-posedness for lower regularity classes the existence of infinitely many weak solutions can be shown by means of convex integration.  Relying on the convex integration method for the homogeneous Euler equations from \cite{DeL-Sz-Annals} the existence of infinitely many solutions with compact space-time support for \eqref{eq:bou} without the influence of gravity, i.e. $g=0$, has been shown in \cite{Bronzi}. Moreover, the paper 
\cite{Chiodaroli_Michalek} addresses system \eqref{eq:bou} under the additional influence of the Coriolis force in the momentum balance and dissipation in the continuity equation. For this dissipative Boussinesq system it is shown that for a given initial density $\rho_0\in \cC^2\cap L^\infty$ there exists an irregular initial velocity field $v_0$ inducing infinitely many solutions which are admissible in the sense that,  for almost every positive time, the total energy of the solutions does not exceed the total initial energy. This is an important property both from a physical point of view and for the mathematical weak-strong uniqueness property, see \cite{Wiedemann} for an overview and \cite{cheng} for the particular case of system \eqref{eq:bou}.

In order to obtain existence of turbulently mixing solutions emanating from the actual Rayleigh-Taylor interface \eqref{eq:initial_data}, Sz\'ekelyhidi, the first and third author have established in \cite{GKSz} the full relaxation of the inhomogeneous incompressible Euler equations (without Boussinesq approximation) allowing them to construct admissible solutions to the corresponding initial value problem under the condition that the quotient of the two fluid densities satisfies
\[
\frac{\rho_+}{\rho_-}\geq \left(\frac{4+2\sqrt{10}}{3}\right)^2.
\]
This translates to an Atwood number $A\geq 0.845$, i.e. in the so called ``ultra high'' regime.

In view of this Atwood number condition the first and third author thereafter have addressed in \cite{GKBou} the Euler equations in Boussinesq approximation, which is applicable for low Atwood number. In the latter paper the full relaxation could also be explicitly given and, in contrast to \cite{GKSz}, admissible turbulent solutions for \eqref{eq:bou},\eqref{eq:initial_data},\eqref{eq:boundary_condition} be constructed, respectively selected, without restrictions on the size of $A>0$. Although of course the equations can not be seen as a reasonable physical system for larger $A$. The selection of these subsolutions, cf. Definition \ref{def:subs} below, is based on imposing maximal initial energy dissipation in the class of one-dimensional self-similar subsolutions. Imposing these requirements leads to a variational problem for the self-similar density profile and the initial speed of the opening of the mixing zone. The problem could be solved explicitly, giving a possible small time selection of subsolutions within the stated class. In the following we will derive a global in time variational problem based on the relaxation given in \cite{GKBou}.

\subsection{Relaxation as a differential inclusion}\label{sec_relaxation_boussinesq}
We first of all rephrase \eqref{eq:bou} as a differential inclusion.
Let as before $n\geq 2$, $L>0$, $T>0$, $\mathcal D =(0,1)^{n-1}\times (-L,L)$ and set $Z:=\R\times\R^n\times\R^n\times\cS^{n\times n}_0\times \R$, where $\cS^{n\times n}_0$ denotes the set of symmetric trace free matrices. 

Consider the linear system
\begin{align}\label{eq:lin_syst}
    \begin{split}
        \partial_t v + \divv \sigma + \nabla p &=-\rho gA e_n,\\
        \divv v&=0,\\
        \partial_t \rho + \divv m&=0,
    \end{split}
\end{align}
in $(0,T)\times\mathcal D$, with boundary conditions
\begin{align}\label{eq:lin_syst_boundary_condition}
    m\cdot\nu =0,\quad v\cdot \nu=0
\end{align}
 on $(0,T)\times\partial\mathcal D$, as well as for given functions $e_0,e_1:(0,T)\times\mathcal D\to \R$ with $e_0\pm e_1\geq 0$ the family of sets $K_{(t,x)}\subset Z$, $(t,x)\in (0,T)\times \cD$ defined by $(\rho,v,m,\sigma,p)\in K_{(t,x)}$ if and only if
 \begin{align}\label{eq:definition_of_K}
     |\rho|=1,\quad m=\rho v,\quad v\otimes v-\sigma=(e_0(t,x)+\rho e_1(t,x))\id.
 \end{align}

It is easy to see that if a tuple $z:=(\rho,v,m,\sigma,p)$ of $L^1_{loc}$ functions satisfies system \eqref{eq:lin_syst},\eqref{eq:lin_syst_boundary_condition} distributionally and if there holds
\begin{align}\label{eq:Kconstr}
    z(t,x)\in K_{(t,x)}
\end{align}
for almost every $(t,x)\in (0,T)\times\mathcal D$, then $(\rho,v)$ is a solution of the inviscid Boussinesq system with local energy density function
\begin{align}\label{eq:local_energy_density}
\begin{split}
    \mathcal E(t,x)&=\frac{1}{2}|v(t,x)|^2+\rho(t,x) gA x_n\\
    &=\frac{n}{2}(e_0(t,x)+\rho(t,x) e_1(t,x))+\rho(t,x) gA x_n.
\end{split}
\end{align}
Similarly in the other direction, if $(\rho,v)\in L^\infty((0,T)\times\cD)\times L^2((0,T)\times \cD)$ solves \eqref{eq:bou}, \eqref{eq:boundary_condition} and $\rho\in\set{-1,1}$ a.e., then one can set $m=\rho v$, $\sigma=(v\otimes v)^\circ$ to see that there exists $e_0,e_1$ and a pressure, such that one obtains a solution to the differential inclusion \eqref{eq:lin_syst},\eqref{eq:lin_syst_boundary_condition},\eqref{eq:Kconstr} with $K_{(t,x)}$ defined with respect to the functions $e_0,e_1$. 

Note that at this point there is no unique or canonical choice for $e_0$ and $e_1$. As can be seen from \eqref{eq:local_energy_density} they dictate the form of the local kinetic energy and should be continuous for the application of convex integration, cf. Theorem \ref{thm:CI} below. The introduction of two such continuous functions in \cite{GKBou} instead of only one in the case of the homogeneous Euler equations \cite{DeL-Sz-Adm} allows the kinetic energy of the solutions to oscillate along the oscillations of the density $\rho$. This additional flexibility turned out to be advantageous in setting up variational problems on the level of subsolutions and will also be exploited in Sections \ref{sec_1D_subsolutions}, \ref{sec:applying_least_action} below.

For the relaxation of the Boussinesq system in the sense of differential inclusions condition \eqref{eq:Kconstr} is replaced by requiring the tuple $z(t,x)$ to take values in the interior of the convex hull $K^{co}_{(t,x)}$, or more generally of the $\Lambda$-convex hull $K^\Lambda_{(t,x)}$, instead. In the case of \eqref{eq:lin_syst}, \eqref{eq:Kconstr} the two notions of convex hulls coincide and its interior is given by 
\begin{multline}\label{eq:Uconstr}
      U_{(t,x)}:=\left\{z=(\rho,v,m,\sigma,p):\ |\rho|<1,\
      \frac{|m+v|^2}{n(\rho+1)^2}<e_0(t,x)+e_1(t,x),\right.\\
      \frac{|m-v|^2}{n(\rho-1)^2}<e_0(t,x)-e_1(t,x),\\ 
      \left.  \lamax\left( \frac{(m-\rho v)\otimes (m-\rho v)}{1-\rho^2}+v\otimes v-\sigma\right)<e_0(t,x)+\rho e_1(t,x)
      \right\},
\end{multline}
see \cite[Proposition 3.6]{GKBou}.
\begin{definition}\label{def:subs}
We say that $z=(\rho,v,m,\sigma,p)$ is a subsolution with respect to a pair of measurable functions $e_0$, $e_1$ iff. $\rho\in L^\infty((0,T)\times\mathcal D)$, $v,m\in L^2((0,T)\times\mathcal D)$, $\sigma\in L^1((0,T)\times\mathcal D)$, $p$ is a distribution, 
$z$ satisfies the linear system \eqref{eq:lin_syst} with boundary and initial data \eqref{eq:lin_syst_boundary_condition}, \eqref{eq:initial_data} in the sense of distributions, there holds
\begin{gather}\label{eq:technical_conditions_for_convex_integration_1}
    e_0+\rho e_1\in L^1((0,T)\times\mathcal D),\quad (1-\rho^2)e_1\in L^1((0,T)\times\cD),\\
    \rho e_1\leq 0 \text{ a.e. on }(0,T)\times \cD \label{eq:technical_conditions_for_convex_integration_2}
\end{gather}
and in addition there exists an open set $\mathscr U\subset  (0,T)\times\mathcal D$ such that
\begin{itemize}
    \item[(a)] the functions $e_0,e_1,\rho,v,m,\sigma$ are continuous on $\mathscr U$, and $z(t,x)\in U_{(t,x)}$ for every  $(t,x)\in \mathscr U$;
    \item[(b)] $z(t,x)\in K_{(t,x)}$ for almost every  $(t,x)\in ((0,T)\times\mathcal D)\setminus \mathscr U$.
\end{itemize}
\end{definition}
We call the set $\mathscr U$ the mixing zone associated with $z$. Of course the definition extends to other initial conditions than \eqref{eq:initial_data} compatible with $\abs{\rho_0}=1$ a.e. and $\divv v_0=0$ in $\cD$, $v_0\cdot\nu=0$ on $\partial\cD$. 

Note that condition \eqref{eq:technical_conditions_for_convex_integration_1} is a slight generalization of the one from \cite{GKBou} where it was assumed that $e_0,e_1$ are essentially bounded. This is done in view of the subsolutions we will obtain in this article for which it is not clear that even $e_0,e_1\in L^1((0,T)\times\cD)$ holds true. However, condition \eqref{eq:technical_conditions_for_convex_integration_2} allows such less integrable subsolutions to still induce infinitely many weak solutions whose energy can be controlled. More precisely, there holds the following convex integration theorem, whose proof, i.e. the necesseray modifications due to \eqref{eq:technical_conditions_for_convex_integration_1} instead of $e_0,e_1\in L^\infty$, can be found in Appendix \ref{sec_convex_integration}.

\begin{theorem}\label{thm:CI} Given an arbitrary error function $\delta:[0,T]\rightarrow \R$ with $\delta(0)=0$ and $\delta(t)>0$ for $t>0$, and a subsolution $z_{sub}$, there exist infinitely many solutions $(\rho_{sol},v_{sol})$ of \eqref{eq:bou}, \eqref{eq:initial_data}, \eqref{eq:boundary_condition} coinciding with $(\rho_{sub},v_{sub})$ outside $\mathscr U$, and on $\mathscr{U}$ having the local energy density
\begin{align}\label{eq:solen}
\begin{split}
    &\mathcal E_{sol}(t,x)=\frac{n}{2}(e_0(t,x)+\rho_{sol}(t,x) e_1(t,x))+\rho_{sol}(t,x) gA x_n\\
    &\phantom{h}=\frac{n}{2}(e_0(t,x)+\rho_{sub}(t,x) e_1(t,x))+\rho_{sub}(t,x) gA x_n+\mathcal E^1_\delta(t,x)+\cE^2_\delta(t,x),
\end{split}
\end{align}
with $\cE^1_\delta:=\frac{n}{2}e_1(\rho_{sol}-\rho_{sub})$, $\cE^2_\delta:=gAx_n(\rho_{sol}-\rho_{sub})$ satisfying 
\begin{align}\label{eq:error_in_energies}
\abs{\int_{\cD}\cE^i_\delta(t,x)\:dx}\leq \delta(t)\quad \text{ for a.e. }t\in(0,T),~i=1,2.
\end{align}
Moreover, the solutions can be found arbitrarily close to $(\rho_{sub},v_{sub})$ in the weak $L^2(\mathscr U)$ topology.
\end{theorem}
\begin{remark}\label{rem:convex_integration_theorem1} Condition \eqref{eq:solen}, \eqref{eq:error_in_energies} allows to conclude that the induced solutions are weakly admissible provided that the subsolution satisfies
\begin{align}\label{eq:subsolution_admissibility}
    \int_{\cD}\frac{n}{2}(e_0(t,x)+\rho_{sub}(t,x) e_1(t,x))+\rho_{sub}(t,x) gA x_n\:dx< \int_{\cD}\rho_0(x)gAx_n\:dx 
\end{align}
for a.e. $t\in(0,T)$. Here the right hand side is precisely the total initial energy associated with \eqref{eq:initial_data}.
\end{remark}
\begin{remark}\label{rem:convex_integration_theorem2}
The solutions $(\rho_{sol},v_{sol})$ given by Theorem \ref{thm:CI} satisfy $\abs{\rho_{sol}(t,x)}=1$ for a.e. $(t,x)\in(0,T)\times\cD$. This is a consequence of \eqref{eq:definition_of_K}, where this condition has been imposed as part of the differential inclusion in consistency with the transport equation in \eqref{eq:bou} and the initial data $\rho_0$. Note however that the conclusion of Theorem \ref{thm:CI} remains valid if $(\rho_{sub},v_{sub})$ satisfies 
\begin{itemize}
    \item[(b$'$)] $z(t,x)\in K_{(t,x)}$ or $\rho(t,x)\in(-1,1)$, $(v,m,\sigma,e_0,e_1)(t,x)=(0,0,0,0,0)$ for almost every  $(t,x)\in ((0,T)\times\mathcal D)\setminus \mathscr U$ 
\end{itemize}
instead of Definition \ref{def:subs} (b). This is indeed true because the convex integration, cf. Appendix \ref{sec_convex_integration}, is carried out on $\mathscr{U}$ and the specified $0$-state is already a solution to the Boussinesq system. The difference is that the induced solutions $(\rho_{sol},v_{sol})$ may have regions of positive Lebesgue measure where the fluid is at rest, $v_{sol}=0$, but $\abs{\rho_{sol}}<1$. We refer to these slightly relaxed notion of subsolutions as ``subsolutions with mixed resting regions''.
\end{remark}

\subsection{One-dimensional subsolutions}\label{sec_1D_subsolutions}

Since the initial data \eqref{eq:initial_data} we consider only depends on $x_n$, one can consider subsolutions which are obtained by averaging solutions in all other spatial directions, i.e.
\begin{align}\label{eq:averaging_of_solutions}
    z(t,x_n):=\int_{[0,1]^{n-1}}z_{sol}(t,x)d(x_1,\ldots,x_{n-1}).
\end{align}
This averaging procedure suggests to consider subsolutions that only depend on $t$ and $x_n$. 
In this case one easily sees that $\divv v=0$ and $v\cdot \nu=0$ imply $v= 0$ almost everywhere, as well as $\partial_{x_j} m=0$ for $j=1,\ldots,n-1$ and $m\cdot \nu =0$ imply $m_j=0$ for $j=1,\ldots,n-1$.

\begin{definition}\label{def:1D_subs} A subsolution $z=(\rho,v,m,\sigma,p)$ is called a one-dimensional subsolution provided $z(t,x)$ depends only on $(t,x_n)\in (0,T)\times (-L,L)$ and $v\equiv 0$, $m_j\equiv 0$, $j=1,\ldots,n-1$.
\end{definition}

\begin{lemma}\label{lem:rise_of_1D_subsols}
Let $z$ be a one-dimensional subsolution w.r.t. $e_0,e_1$. Then $\rho(t,x)=\rho(t,x_n)$ and $m(t,x)=m_n(t,x_n)e_n$ enjoy the following properties
\begin{enumerate}[(i)]
    \item $\partial_t\rho+\partial_{x_n}m_n=0$ weakly, \label{eq:1D_pde}
    \item $m_n(\cdot,\pm L)=0$, $\rho(0,\cdot)=\sign$ weakly,\label{eq:1D_boundary_initial_data}
    \item there exists $\mathscr U'\subset (0,T)\times(-L,L)$ open s.t. $\rho,m_n\in\cC^0(\mathscr U')$, $\abs{\rho}<1$ on $\mathscr U'$,\label{eq:1D_mixing_zone}
    \item $\abs{\rho}=1$, $m_n=0$ a.e. outside $\mathscr U'$\label{eq:1D_outside_mixing_zone}
    \item $\frac{m_n^2}{1-\rho^2}\in L^1(\mathscr U')$,\label{eq:1D_L1}
    \item $n(e_0+\rho e_1)>\frac{m_n^2}{1-\rho^2}$ on $\mathscr U'$.\label{eq:1D_kinetic_energy}
\end{enumerate}
Conversely, let $\rho,m_n:(0,T)\times(-L,L)\rightarrow \R$ be measurable functions satisfying properties \eqref{eq:1D_pde}-\eqref{eq:1D_L1} and $\tilde{e}\in\cC^0(\mathscr U')\cap L^1(\mathscr U')$, $\tilde{e}>0$. Then for suitable $e_0,e_1$ the pair $(\rho,m_n)$ induces a one-dimensional subsolution $z$ with $\rho(t,x)=\rho(t,x_n)$, $m(t,x)=m_n(t,x_n)e_n$, $\mathscr U:=\set{(t,x):(t,x_n)\in \mathscr U'}$ and kinetic energy density given by 
\begin{align}\label{eq:kinetic_energy_density_1D_subs}
\frac{n}{2}(e_0(t,x)+\rho(t,x)e_1(t,x))=\frac{\abs{m_n(t,x_n)}^2}{2(1-\rho(t,x_n)^2)}+\tilde{e}(t,x_n)
\end{align}
for $(t,x)\in \mathscr U$.
\end{lemma}
\begin{remark}\label{rem:notion_of_weak_solution}
The weak notion of solution in Lemma \ref{lem:rise_of_1D_subsols} \eqref{eq:1D_pde},\eqref{eq:1D_boundary_initial_data} is understood in the sense that 
\begin{align}\label{eq:notion_of_weak_solution}
    \int_0^T\int_{-L}^L\rho \partial_t\varphi+m_n\partial_{x_n}\varphi\:dx_n\:dt+\int_{-L}^L\sign(x_n)\varphi(0,x_n)\:dx_n=0
\end{align}
for all $\varphi\in\cC^\infty_c([0,T)\times [-L,L])$.
\end{remark}

\begin{remark}\label{rem:one_dim_subs_with_resting_regions}
If $(\rho,m_n)$ satisfies \eqref{eq:1D_pde}-\eqref{eq:1D_mixing_zone}, \eqref{eq:1D_L1}, but instead of \eqref{eq:1D_outside_mixing_zone} only 
\[
m_n=0,~\abs{\rho}\leq 1\text{ a.e. outside }\mathscr{U}',
\]
then $(\rho,m_n)$ induces a one-dimensional subsolution with mixed resting regions as described in Remark \ref{rem:convex_integration_theorem2}.
\end{remark}

\begin{proof}[Proof of Lemma \ref{lem:rise_of_1D_subsols}]
If $z$ is a one-dimensional subsolution w.r.t. $e_0,e_1$ and mixing zone $\mathscr U$, then properties \eqref{eq:1D_pde}-\eqref{eq:1D_outside_mixing_zone} clearly hold true for $\mathscr U'$ being the projection of $\mathscr U$. Moreover, \eqref{eq:1D_L1} follows from \eqref{eq:1D_kinetic_energy}. Thus it remains to prove \eqref{eq:1D_kinetic_energy}. 

By definition of a subsolution, cf. \eqref{eq:Uconstr}, there holds
\begin{align}\label{eq:rm2}
     \frac{|m|^2}{n(\rho+1)^2}<e_0+e_1,\quad
      \frac{|m|^2}{n(\rho-1)^2}<e_0-e_1.
\end{align}
inside the mixing zone $\mathscr U$. It follows that 
\begin{align}
\begin{split}\label{eq:rmmmm2}
    n\left(e_0+\rho e_1\right)&=n\left(\frac{1+\rho}{2}(e_0+e_1)+\frac{1-\rho}{2} (e_0-e_1)\right)\\&>  n\left(\frac{1+\rho}{2} \frac{|m|^2}{n(\rho+1)^2}+\frac{1-\rho}{2} \frac{|m|^2}{n(\rho-1)^2}\right)=\frac{m_n^2}{1-\rho^2}.
\end{split}
\end{align}

Let now $\rho,m_n$ and $\tilde{e}$ be given as stated. We will define $z$ and suitable $e_0,e_1$ in terms of these three functions. We set $\mathscr U:=\set{(t,x)\in(0,T)\times \cD:(t,x_n)\in \mathscr U'}$.

Since $v$ has to be $0$ throughout $(0,T)\times \cD$ for a one-dimensional subsolution, we have, have to set resp., $m=\rho v=0$, $\sigma=(v\otimes v)^\circ=0$ and $n(e_0+\rho e_1)=\abs{v}^2=0$, hence $e_0=-\rho e_1$, outside the mixing zone $\mathscr U$. Without loss of generality we set $e_0=e_1=0$ on $((0,T)\times\cD)\setminus \mathscr U$. In consequence for a.e. $(t,x)\notin \mathscr U$ there holds $z(t,x)\in K_{(t,x)}$ with $K_{(t,x)}$ defined in \eqref{eq:definition_of_K} for $e_0=e_1=0$. 

On the other hand, inside the mixing zone we of course set $\rho(t,x)=\rho(t,x_n)$, $m(t,x)=m_n(t,x)e_n$ and observe that for $z(t,x)\in U_{(t,x)}$ it remains to satisfy
\begin{gather*}
      \frac{|m|^2}{n(\rho+1)^2}<e_0(t,x)+e_1(t,x),\quad
      \frac{|m|^2}{n(\rho-1)^2}<e_0(t,x)-e_1(t,x),\\ 
 \lamax\left( \frac{|m|^2}{1-\rho^2}e_n\otimes e_n-\sigma\right)<e_0(t,x)+\rho e_1(t,x).
\end{gather*}

First of all we claim that the third inequality automatically holds true provided the first two inequalities are valid and we define $\sigma\in\cS^{n\times n}_0$ such that  
$$
\frac{|m|^2}{1-\rho^2}e_n\otimes e_n-\sigma=\frac{|m|^2}{n(1-\rho^2)}\id.
$$
Indeed, in that case, reversing the calculations in \eqref{eq:rmmmm2}, there holds
\begin{align*}
    \lamax\left( \frac{|m|^2}{1-\rho^2}e_n\otimes e_n-\sigma\right)&=\frac{|m|^2}{n(1-\rho^2)}<e_0+\rho e_1.
\end{align*}

Moreover, turning to the set of linear equations \eqref{eq:lin_syst} one has
\begin{align*}
    \partial_t v +\divv \sigma +\rho g A e_n =( \partial_{x_n}\sigma_{nn}+\rho g A )e_n,
\end{align*}
which can always be written as $-\nabla p$ by setting 
$$
p(t,x):=-\sigma_{nn}(t,x)-\int_0^{x_n} \rho(t,x_n') \, dx_n' g A.
$$
The other two equations in \eqref{eq:lin_syst} with correct initial and boundary data hold true by assumption. 

In order to have a subsolution, it therefore only remains to find $e_0,e_1$ such that the two inequalities \eqref{eq:rm2} are valid. It is easy to check that these conditions and \eqref{eq:kinetic_energy_density_1D_subs} hold true for
\begin{align*}
    e_0:=\frac{ |m|^2(1+\rho^2)}{n(1-\rho^2)^2}+\tilde{e},\quad e_1:=-\frac{2\rho |m|^2}{n(1-\rho^2)^2}.
\end{align*}
Observing that conditions \eqref{eq:technical_conditions_for_convex_integration_1}, \eqref{eq:technical_conditions_for_convex_integration_2} indeed hold true finishes the construction of the subsolution induced by $\rho$ and $m_n$.
\end{proof}

\subsection{Applying the least action principle}\label{sec:applying_least_action}

In the classical case the action functional consists of the difference between kinetic and potential energy, and when minimized over suitable paths yields  the equations of motion for the described mechanical system. Going back to Arnold it is well-known that this principle can formally also be used to derive the Euler equations, see Appendix \ref{sec_brenier} and the references therein for more detail. We will now state it on the level of subsolutions.

The total potential energy of a one-dimensional subsolution $z=z_{sub}$ at time $t\in[0,T]$ is given by 
\[
E_{pot}(t):=\int_{\cD}\rho(t,x)gAx_n\:dx=\int_{-L}^L\rho(t,x_n)gAx_n\:dx_n.
\]
Note that in view of Theorem \ref{thm:CI} for any given error function $\delta(t)$ there exist solutions $z_{sol}$ whose total potential energy $\int_{\cD} \rho_{sol}gAx_n\:dx$ at time $t$ is $\delta(t)$-close to $E_{pot}(t)$.

In a similar way it follows from Theorem \ref{thm:CI} that there exist solutions having at time $t$ a total kinetic energy arbitrarily close to
\[
\int_{\cD} \frac{n}{2}\left(e_0(t,x)+\rho(t,x)e_1(t,x)\right)\:dx.
\]
Moreover,  given $\rho$ and $m$ of a one-dimensional subsolution $z=z_{sub}$, Lemma \ref{lem:rise_of_1D_subsols} shows that pointwise on $\mathscr U$ there holds
\begin{align*}
    \inf\set{\frac{n}{2}(\tilde{e}_0+\rho\tilde{e}_1):\tilde{z}\text{ subsolution w.r.t. }\tilde e_0,\tilde e_1,~\tilde{\rho}=\rho,~\tilde{m}=m}=\frac{\abs{m}^2}{2(1-\rho^2)}.
\end{align*}
Including this pointwise minimization over all possible subsolutions coinciding with $z$ in the $\rho$ and $m$ component we therefore define the total (least) kinetic energy of a one-dimensional subsolution $z$ by 
\[
E_{kin}(t):=\int_{-L}^L\frac{\abs{m(t,x_n)}^2}{2(1-\rho(t,x_n)^2)}\:dx_n.
\]
Here the integrand is understood to be $0$ when $(t,x)$ is outside the mixing zone, i.e. where $\abs{\rho}=1$, $m=0$. Indeed outside the mixing zone there holds $e_0+\rho e_1=0$ for any one-dimensional subsolution.

Note that the pointwise optimization above does not affect the potential energy $E_{pot}(t)$. It is therefore compatible with the following least action principle: 
\begin{align}\label{eq:our_least_action_principle}
    \text{minimize } \cA_0(z)\text{ over one-dimensional subsolutions }z,
\end{align}
where the action is defined as
\begin{align}
\begin{split}\label{eq:action_after_derivation}
    \cA_0(z)&:=\int_0^TE_{kin}(t)-E_{pot}(t)\:dt\\
    &=\int_0^T\int_{-L}^L\frac{m_n(t,x_n)^2}{2(1-\rho(t,x_n)^2)}-\rho(t,x_n)gAx_n\:dx_n\:dt.
\end{split}
\end{align}

As subsolutions relax the notion of solution for the Euler, Boussinesq resp., equations, the least action principle \eqref{eq:our_least_action_principle} can be seen as a generalization of the classical least action principle giving rise to the equations itself. We will show in Appendix \ref{sec_brenier} that this generalization is formally equivalent to the generalization of the least action principle given by Brenier in \cite{Brenier89,Brenier18}. 

Note that at this point we have not yet specified the final configuration for the subsolutions at the end time $T$ which is usually done in applications of the least action principle. This is postponed to Section \ref{sec_final_configuration}. 

\subsection{Reformulation}\label{sec:reformulation} Before continuing let us simplify our notation. First of all we keep from a one-dimensional subsolution $z$ only the information relevant for the action, that is the density $\rho$ and  the last component of the momentum $m$, which we again denote by $m$. 

Furthermore, let $\Omega:=(0,T)\times (-L,L)$ and identify a point $(t,x_n)\in \Omega$ simply by $x=(x_1,x_2)$. That is time is denoted now by $x_1$ and the last coordinate in the box $\cD$ by $x_2$. Thus the action functional \eqref{eq:action_after_derivation} can be written as
\begin{align*}
    \cA_0(\rho,m)=\int_\Omega \frac{m^2}{2(1-\rho^2)}-\rho gAx_2\:dx =\int_\Omega F(m,\rho)-\rho gAx_2\:dx,
\end{align*}
where $F$ is defined in \eqref{eq:definition_of_F_in_introduction}. Recall here that the kinetic energy density satisfies $\frac{m^2}{2(1-\rho^2)}=0$ whenever $m=0$.  

As seen in Lemma \ref{lem:rise_of_1D_subsols} the pair $(m,\rho):\Omega\rightarrow \R^2$ has to satisfy properties \eqref{eq:1D_pde}--\eqref{eq:1D_L1} of said lemma in order to correspond to a one-dimensional subsolution in the sense of Definitions \ref{def:subs}, \ref{def:1D_subs}. Some of these properties will be implemented directly into the variational formulation, while others will be shown a posteriori for an existing minimizer.

First of all observe that property \eqref{eq:1D_L1}, i.e. the $L^1$-integrability of $\frac{m^2}{1-\rho^2}$, as well as $\abs{\rho}\leq 1$ a.e. follows for $(m,\rho)$ measurable with finite action $\cA_0(\rho,m)$. 

Next, property \eqref{eq:1D_pde}, i.e. the equation $\partial_{x_1}\rho+\partial_{x_2}m=0$, will be encoded by introducing a stream function for the divergence free vector field $(\rho,m)$. If the action is finite we have $(\rho,m)\in L^2(\Omega;\R^2)$ and therefore find $u\in H^1(\Omega)$ with $m=-\partial_{x_1}u$, $\rho=\partial_{x_2}u$.

Moreover, in view of the needed initial and boundary data, i.e. property \eqref{eq:1D_boundary_initial_data}, we require the stream function $u\in H^1(\Omega)$ to satisfy \begin{align}\label{eq:reformulation_of_boundary_initital_data}
    u(0,x_2)=\abs{x_2}-L,~x_2\in(-L,L),\quad u(x_1,\pm L)=0,~x_1\in(0,T)
\end{align}  
in the sense of traces. Indeed one can easily check that for $\varphi\in \cC^\infty_c([0,T)\times[-L,L])$ there holds
\begin{align*}
    \int_\Omega \partial_{x_2}u\partial_{x_1}\varphi-\partial_{x_1}u\partial_{x_2}\varphi\:dx&=\int_{-L}^L(\abs{x_2}-L)\partial_{x_2}\varphi(0,x_2)\:dx_2\\
    &=-\int_{-L}^L\sign(x_2)\varphi(0,x_2)\:dx_2.
\end{align*}
Hence \eqref{eq:notion_of_weak_solution} and therefore Lemma \eqref{eq:1D_pde},\eqref{eq:1D_boundary_initial_data} are satisfied.
In Section \ref{sec_initial_data} we will in fact show that the boundary and initial data are attained in a stronger sense. Moreover, our investigation will in addition show that $m(0,\cdot)=-\partial_{x_1}u(0,\cdot)=0$ in this sense, see Lemma \ref{lem:txlim}.

The remaining properties Lemma \ref{lem:rise_of_1D_subsols} \eqref{eq:1D_mixing_zone},\eqref{eq:1D_outside_mixing_zone}, as well as the admissibility of the total energy, cf. \eqref{eq:subsolution_admissibility}, will be part of our investigation.

At this point the action functional in terms of a stream function $u$ satisfying \eqref{eq:reformulation_of_boundary_initital_data} can be written after a simple integration by parts as
\begin{align}\label{eq:action_after_reformulation}
    \cA_0(u)=\int_\Omega F(\nabla u)+gAu\:dx,
\end{align}
i.e. we have arrived at \eqref{eq:functional_in_introduction} with $V(x_2,u)=-gAu$.

\subsection{Final configuration}\label{sec_final_configuration} As mentioned above the least action principle is formulated with respect to variations over a class of trajectories connecting a given initial and target configuration. While our initial configuration is clear, there are plenty of admissible target configurations possible. In the present article we simply chose the stable interface configuration $-\rho_0$. This configuration has the overall least potential energy, thus also the overall least total energy provided the fluid is at rest, and therefore is a canonical candidate for the long-time limit of the system. 

In terms of the introduced stream function $u$ we therefore add
\begin{equation}\label{eq:final_configuration}
    u(T,x_2)=L-\abs{x_2}, ~x_2\in(-L,L)
\end{equation}
to the list of requirements \eqref{eq:reformulation_of_boundary_initital_data}. Note that this implies that in Remark \ref{rem:notion_of_weak_solution} the equation can be tested against $\varphi\in\cC^\infty(\overline{\Omega})$ while adding $\int_{-L}^L\sign(x_n)\varphi(T,x_n)\:dx_n$ to the left-hand-side of \eqref{eq:notion_of_weak_solution}.

\subsection{Energy dissipation}\label{sec_energy_dissipation}
It is a well-known built in feature of the least action principle that solutions conserve the total energy, which in our case at time $x_1$ reads
\begin{equation}\label{eq:total_energy_at_time_t}
    E_{tot}(x_1):=\int_{-L}^LF(\nabla u(x))-gAu(x)\:dx_2.
\end{equation}
Indeed, up to formally admitting the Euler-Lagrange equations 
$
\divv (\nabla F(\nabla u))=gA
$
of \eqref{eq:action_after_reformulation}, one also here has
\begin{align*}
    \frac{d}{dx_1}E_{tot} &=\int_{-L}^L \nabla F (\nabla u) \cdot \nabla (\partial_{x_1}u)-gA\partial_{x_1}u \, dx_2\\
    &=\int_{-L}^L \partial_{p_1}F(\nabla u)\partial^2_{x_1^2} u - \partial_{x_2}(\partial_{p_2}F(\nabla u))\partial_{x_1} u
-gA\partial_{x_1}u \, dx_2
    \\
    &=\int_{-L}^L \frac{d}{dx_1}(\partial_{p_1}F(\nabla u)\partial_{x_1} u) -\divv(F (\nabla u)) \partial_{x_1}u-gA\partial_{x_1}u \, dx_2\\
   &=   \int_{-L}^L \frac{d}{dx_1}(\partial_{p_1}F(\nabla u)\partial_{x_1} u) -2gA\partial_{x_1}u \, dx_2\\
   &=
   \int_{-L}^L \frac{d}{dx_1}(2F(\nabla u)) -2gA\partial_{x_1}u \, dx_2=2\frac{d}{dx_1}E_{tot}.
\end{align*}
Hence the total energy is constant in time. This formal computation will be made rigorous in Section \ref{sec_energy}.

This is of course undesirable in the context of turbulent fluid dynamics where energy is anomalously dissipated. Note also that, as it can be seen in \eqref{eq:solen}, the energy of the associated solutions obtained via convex integration differs from the energy of the subsolution with a small margin of error. However, if the energy of the subsolution is conserved, it can happen a priori that due to this margin of error, the energy of the solution will increase.

There are plenty of modifications and extensions of the least action principle in order to include energy dissipation, see for instance also the discussion in \cite{Brenier18}.

In the present article we overcome the issue of energy conservation by introducing an additional nonlinear potential energy. That is instead of $\cA_0$ we consider
\eqref{eq:functional_in_introduction},
% \begin{equation*}
% \cA(u)=\int_\Omega F(\nabla u)-V(x_2,u)\:dx,
% \end{equation*}
where now $V:[-L,L]\times \R\rightarrow \R$ has the form
\[
V(x_2,z)=-gAz+f(x_2,z).
\]
Through similar formal calculations as previously, one obtains that the total energy \eqref{eq:total_energy_at_time_t} now changes according to
\begin{align}\label{eq:new_energy_balance}
   \frac{d}{dx_1}\int_{-L}^L F (\nabla u)-gA 
    u \,dx_2 =-\int_{-L}^L \partial_z f(x_2,u) \partial_{x_1}
    u\,dx_2.
\end{align}
We will show that $\partial_{x_1} u\geq 0$, see Corollary \ref{cor:bounds_for_minimizer}, which means that the average momentum $m=-\partial_{x_1}u$ is negative. Thus if $\partial_z f> 0$, then strict energy dissipation for the subsolution, and hence also for the associated solutions is possible.

Integrating \eqref{eq:new_energy_balance} in time one obtains that 
\[
E_{kin}(x_1)+E_{pot}(x_1)+\int_{-L}^Lf(x_2,u(x))\:dx_2
\]
is constant for minimizers of \eqref{eq:functional_in_introduction}. 
Hence the dissipated kinetic and potential energy is absorbed in the new energy
\[
E_f(x_1):=\int_{-L}^Lf(x_2,u(x))\:dx_2.
\]
Furthermore, note that in \eqref{eq:new_energy_balance} only points where $\partial_{x_1}u>0$ contribute to a dissipation, i.e. no points outside the mixing zone $\mathscr U'$. In terms of the induced solutions this means that the energy is only dissipated inside the mixing zone where they are wildly oscillating.  

Of course at this point there are plenty of possible choices for $f$. A specific example is given in Section \ref{sec:example_f} below, after we have introduced one more condition in the next Section \ref{sec:initial_final_energy}. We refer also to the discussion in Section \ref{sec_summary_questions}.

\subsection{Initial and final energies}\label{sec:initial_final_energy}

Formally taking the limits $x_1\rightarrow 0$, $x_1\rightarrow T$ we deduce that
\begin{align*}
    E_f(T)-E_f(0)=-(E_{kin}(T)-E_{kin}(0)+E_{pot}(T)-E_{pot}(0)).
\end{align*}
The initial and final potential energy can easily be computed for $u$ satisfying \eqref{eq:reformulation_of_boundary_initital_data}, \eqref{eq:final_configuration}. There holds
\[
-E_{pot}(T)=E_{pot}(0)=gAL^2.
\]
Thus requiring 
\begin{align}\label{eq:difference_internal_energy}
    E_f(T)-E_f(0)=2gAL^2
\end{align}
renders the solutions to start and end with the same kinetic energy $E_{kin}(0)=E_{kin}(T)$, which in view of \eqref{eq:initial_data}, \eqref{eq:subsolution_admissibility} should be $0$. Note that then the fluid is at rest also at the final time.

In order to achieve $E_{kin}(0)=E_{kin}(T)=0$ as a consequence of minimizing the action functional \eqref{eq:functional_in_introduction} we choose $f$, such that for
\begin{align*}
    s_V:=\sup\set{\int_{-L}^LV(x_2, \varphi)\:dx_2:\varphi\in \cC^0([-L,L]),~\abs{\varphi(x_2)}\leq L-\abs{x_2}}
\end{align*}
there holds
\begin{align}\label{eq:condition_on_supremum_of_V}
    s_V=\int_{-L}^LV(x_2,\varphi)\:dx_2\quad \text{if and only if}\quad  \varphi=\pm(L-\abs{x_2}).
\end{align}

In fact we will show in Section \ref{sec_global_admissibility} that this condition allows to conclude
that the subsolution starts and ends with $0$ kinetic energy at least as $T\rightarrow+\infty$. This can be seen similarly to the existence of heteroclinic orbits for instance in pendulum equations.

\subsection{Example}\label{sec:example_f} One of the simplest examples for $f$ satisfying the requested properties is
\begin{align}\label{eq:example_for_f}
    f(x_2,z)=\frac{3gA}{4L}(z-(\abs{x_2}-L))^2,
\end{align}
which is also stated in \eqref{eq:example_for_V_in_introduction}.

Indeed, the monotonicity w.r.t. $z$ holds true for $z\geq \abs{x_2}-L$ which in view of Corollary \ref{cor:bounds_for_minimizer} turns out to be enough. Regarding \eqref{eq:condition_on_supremum_of_V} we first of all observe that 
\[
\int_{-L}^LV(x_2,L-\abs{x_2})\:dx_2=\int_{-L}^LV(x_2,\abs{x_2}-L)=gAL^2
\]
by the choice of the constant $\frac{3gA}{4L}$ in \eqref{eq:example_for_f}. It remains to show that 
\[
\cV(\varphi):=\int_{-L}^LV(x_2,\varphi(x_2))\:dx_2<gAL^2
\]
for any $\varphi:[-L,L]\rightarrow\R$ continuous with $\abs{\varphi(x_2)}\leq L-\abs{x_2}$, but $\varphi(x_2)$ not identical $L-\abs{x_2}$ or $-(L-\abs{x_2})$. Let $\varphi$ be such a function, i.e. there exists an open interval $I\subset(-L,L)$ on which $\abs{\varphi(x_2)}<L-\abs{x_2}$. Considering perturbations $\varphi+\varepsilon\psi$ with $\supp \psi\subset I$ and $\abs{\varepsilon}$ small enough, one concludes that $\varepsilon\mapsto \cV(\varphi+\varepsilon\psi)$ is a uniformly convex $\cC^2$ function, and thus can not have its supremum achieved in $\varepsilon=0$.  

We remark that for this specific example the new energy term absorbing kinetic and potential energy can be expressed in terms of the actual variables $(\rho,m)$ (instead of the potential $u$) as 
\[
E_f(x_1)=\frac{3gA}{4L}\norm{u(x_1,\cdot)-u(0,\cdot)}_{L^2(-L,L)}^2=\frac{3gA}{4L}\norm{\rho(x_1,\cdot)-\rho_0}^2_{H^{-1}(-L,L)},
\]
where $H^{-1}(-L,L)$ denotes the dual of $H^1_0(-L,L)$ with respect to the topology induced by the norm $\norm{\partial_{x_2}( \cdot)}_{L^2(-L,L)}$.

\section{A degenerate variational problem}\label{sec_main_result}

We now turn to the investigation of problem \eqref{eq:variational_problem_in_introduction}. More precisely we seek to minimize
\begin{align}\label{eq:functional_in_sec_3}
    \cA(u)=\int_\Omega F(\nabla u)-V(x,u)\:dx
\end{align}
over the class of functions $u\in X$ with $F$ defined in \eqref{eq:definition_of_F_in_introduction} and $X$ given by \eqref{eq:definition_of_X_in_introduction}.

The nonlinear potential $V:\overline{\Omega}\times \R\rightarrow \R$, $(x,z)\mapsto V(x,z)$ is supposed to satisfy the following regularity condition:
\begin{gather}\tag{$\text{V}_{\text{reg}}$}\label{eq:condition_V1} \begin{gathered}
V \text{ is $3$ times differentiable with respect to $z$,}\\
   \partial_z^k V:\overline{\Omega}\times\R\rightarrow\R,~k=0,\ldots,3 \text{ are Lipschitz continuous and bounded.}
   \end{gathered}
\end{gather}
This condition will be assumed throughout the remaining article. Note that the example given in \eqref{eq:example_for_V_in_introduction}, Section \ref{sec:example_f} resp., satisfies \eqref{eq:condition_V1} when the stated $V$ is extended smoothly outside $\overline{\Omega}\times [-L,L]$, such that $V$ and its $z$-derivatives are globally bounded. We remark that the precise extension turns out to be irrelevant in view of Corollary \ref{cor:bounds_for_minimizer}.

Moreover, we will frequently also assume that the potential is autonomous with respect to $x_1$, i.e.
\begin{equation}\label{eq:V1prime}\tag{$\text{V}_\text{aut}$}
    V(x,z)=V(x_2,z),~x\in\overline{\Omega},~z\in\R,
\end{equation}
such that then \eqref{eq:functional_in_sec_3} reduces to \eqref{eq:functional_in_introduction}.

Two other conditions on $V$, as indicated in Sections \ref{sec_energy_dissipation}, \ref{sec:initial_final_energy}, will be introduced when needed, which is only in the very last part of our investigation in Section \ref{sec_global_admissibility} when it comes to the interpretation of our minimizer as a subsolution for the Boussinesq system.

For now our main goal is to show the following existence and partial regularity result for the  variational problem 
\begin{align}\label{eq:var_prop_sec_3}
\text{find }u\in X\text{ such that }  \cA (u)=\inf_{u\in X} \cA (u),
\end{align}
as well as the associated energy balance in the autonomous case.
\begin{theorem}\label{thm:existence_and_part_regularity}
Suppose that \eqref{eq:condition_V1} holds. Then problem \eqref{eq:var_prop_sec_3} with $\cA$ defined in \eqref{eq:functional_in_sec_3} has a solution $u$ and there exists $\Omega^\prime\subset\Omega$ open, nonempty such that $u|_{\Omega^\prime}$ is of class $\cC^2$ with $\partial_{x_1}u\neq 0$, $\abs{\partial_{x_2}u}<1$ on $\Omega'$ and $\partial_{x_1}u(x)=0$, $\abs{\partial_{x_2}u(x)}\leq 1$ for a.e. $x\notin \Omega^\prime$. Moreover, if in addition \eqref{eq:V1prime} holds true, then $u\in \cC^0(\overline{\Omega})$, $\partial_{x_1}u>0$ on $\Omega'$ and there holds
\begin{equation}\label{eq:energy_balance_sec3}
    \frac{d}{dx_1}\int_{-L}^LF(\nabla u(x))+V(x_2,u(x))\:dx_2=0
\end{equation}
weakly on $(0,T)$.
\end{theorem}

The main difficulty lies in the degeneracy of the convex, lower semi-continuous integrand $F$. Indeed, denoting by $\Lambda(p)$, $\lambda(p)$ the maximal and minimal eigenvalue of $D^2F(p)$, cf. \eqref{eq:Fepsder} with $\varepsilon=0$, for $p_1\neq 0$, $\abs{p_2}<1$ there holds
\[
\Lambda(p)\rightarrow +\infty \text{ as }\abs{p_2}\rightarrow 1\text{ and }\lambda(p)\rightarrow 0 \text{ as }p_1\rightarrow 0.
\]
Thus the problem degenerates on the non-convex set $E:=\set{0}\times[-1,1]\cup \R\times\set{\pm 1}$ with indefinite behaviour for $\det D^2F(p)$ as $p\rightarrow (0,\pm 1)$.

This is in contrast to the prototype of degenerate problems, i.e. the $p$-Laplace problem with $F(\xi)=\abs{\xi}^p$, where the ellipticity constants degenerate only at the single point $E=\{0\}$ and with definite behaviour for all eigenvalues as $\xi\rightarrow 0$. In that case $\cC^{1,\alpha}$-regularity for minimizers is known, see \cite{evans_p,lewis,tolksdorf,uhlenbeck,uraltseva}.

Another proof for the regularity of $p$-harmonic maps is given by Wang \cite{wang} relying on a separation between degenerate and non-degenerate points, and the fact that near the degenerate set $\nabla u$ is  anyway small. Beyond the $p$-Laplacian Colombo and Figalli \cite{colombo_figalli} also applied a separation strategy based on ideas of \cite{wang} to problems that degenerate on a bounded convex set $E$ having $0$ in its interior. One of these ``very degenerate'' integrants is for example
$
F(\xi)=\left(\abs{\xi}-1\right)^p_+
$ which arises in problems related to traffic congestion. The main theorem of \cite{colombo_figalli} states that for these type of problems with $V(x,z)=V(x)$ the composition $\nabla F(\nabla u)$ is continuous. This extends the work \cite{santambrogio_vespri} by Santambrogio and Vespri relying on two-dimensional methods for $E=B_1(0)$ to any dimension and a general convex bounded $E$ with $0\in\text{int}(E)$.

Near nondegenerate points the proof of \cite{colombo_figalli} is based on a compactness result for small solutions to elliptic equations in the spirit of Savin \cite{savin} which renders also in the present work one of the main ingredients in the proof of Theorem \ref{thm:existence_and_part_regularity}. However, the nature of the degeneracy of our $F$ did so far not allow us to conclude a global regularity result like the one in \cite{colombo_figalli}.

In \cite{desilva_savin} De Silva and Savin considered a different type of degenerate variational problem arising for instance in questions related to limits of random surfaces. There the integrand $F$ is a bounded function defined on the closure of a bounded two-dimensional polygon $N$ and the set of degeneracy is given by the union of $\partial N$ and finitely many points inside $N$. Apart from this union $F$ is smooth and strictly convex. There is no potential term, i.e. $V=0$. In that setting \cite[Theorem 1.3]{desilva_savin} provides a partial regularity result for the unique minimizer $u$ and characterizes the behaviour at points where $\nabla u$ is not continuous. More precisely, every point of discontinuity is connected to the boundary of $\Omega$ along a straight segment perpendicular to one of the sides of $N$ and on that segment $u$ is affine linear. Besides the unboundedness of our $F$ and our degeneracy set $E$, the absence of the maximum principle due to the nonlinear potential $V(x,z)$ prevents us from applying the methods of \cite{desilva_savin}. The maximum principle together with a 2-dimensional topological argument seems to be the crucial tool on which their approach is based on.

Still, in the following sense points of discontinuity of $\nabla u$ are, for certain $V$, also in our case connected to $\partial\Omega$: 
\begin{lemma}\label{lem:good_set_property}
Suppose that $V$ satisfies \eqref{eq:condition_V1}, \eqref{eq:V1prime} and $\partial_z^2V(x_2,z)\geq 0$ for all $x_2,z\in[-L,L]$. Whenever $\Omega''\subset \Omega$ is open with $\partial\Omega''\subset \Omega'$, where $\Omega'$ is the set from Theorem 
\ref{thm:existence_and_part_regularity}, then there holds $\Omega''\subset \Omega'$.
\end{lemma}

For a further more general overview on degenerate variational problems we refer to the survey \cite{mooney}.

The proof of Theorem \ref{thm:existence_and_part_regularity} is carried out in Sections \ref{sec_regular_approximation}-\ref{sec_further_properties}. We begin in Section \ref{sec:global_extension} with the construction of regular approximations $\hat{F}_\varepsilon$ for the degenerate integrant $F$. Here some extra attention has to be paid, see Lemma \ref{lem:first_extension_of_F_epsilon} \eqref{eq:estimate_of_extension_for_derivative}, in order to later conclude the energy balance \eqref{eq:energy_balance_sec3}. Having a family of regularized variational problems at hand we deduce in Sections \ref{sec:regularized_var_problem}-\ref{sec_autonomous_potentials_for_approximation} the existence of regular minimizers $u_\varepsilon$ enjoying corresponding $\varepsilon$-versions of the energy balance \eqref{eq:energy_balance_sec3} or Lemma \ref{lem:good_set_property} for instance. After that Section \ref{sec_gamma_convergence} deals with the limit $\varepsilon\rightarrow 0$. We prove $\Gamma$-convergence with respect to the weak $H^1$-topology and characterize in terms of the corresponding Young measure how strong convergence might fail. In particular we deduce the existence of a minimizer $u$ to the degenerate problem \eqref{eq:var_prop_sec_3}. Section \ref{sec_partial_regularity} contains the proof of the partial regularity property which, as mentioned earlier, uses the compactness result of Savin \cite{savin}. Section \ref{sec_further_properties} collects various additional properties of $u$, for example \eqref{eq:energy_balance_sec3}, Lemma \ref{lem:good_set_property}, but also further properties that rely on the mentioned additional conditions on $V$. Finally Section \ref{sec_summary_questions} contains a summary of all conditions on $V$ together with the corresponding conclusions, as well as a discussion of open questions in the variational problem itself, but also regarding our application to the Boussinesq system.

\section{A regular approximation}\label{sec_regular_approximation}

In order to deal with the possible singularity in the denominator of $F$ and the degeneracy of $D^2F(p)$ when $p_1=0$, we introduce in this section the following regular approximation.
For $\varepsilon>0$, let $F_\varepsilon:\set{p\in\R^2:\abs{p_2}<1+\varepsilon}\rightarrow\R$,
\begin{align}\label{eq:feps}
    F_\varepsilon(p)=\frac{p_1^2+\varepsilon^\theta}{2((1+\varepsilon)^2-p_2^2)},
\end{align}
where $\theta\in(1,2)$ is a fixed constant.
A quick calculation yields
\begin{align}\label{eq:Fepsder}
\begin{split}
    \nabla F_\varepsilon(p)&=\begin{pmatrix}\frac{p_1}{(1+\varepsilon)^2-p_2^2}&\frac{(p_1^2+\varepsilon^\theta)p_2}{((1+\varepsilon)^2-p_2^2)^2}\end{pmatrix}^T,\\
    D^2 F_\varepsilon(p)&=\begin{pmatrix}
   \frac{1}{(1+\varepsilon)^2-p_2^2} & \frac{2p_1p_2}{((1+\varepsilon)^2-p_2^2)^2}  \\
   \frac{2p_1p_2}{((1+\varepsilon)^2-p_2^2)^2} & \frac{(p_1^2+\varepsilon^\theta)((1+\varepsilon)^2+3p_2^2) }{((1+\varepsilon)^2-p_2^2)^3}   
 \end{pmatrix},\\
 \det(D^2F_\varepsilon(p))&=\frac{p_1^2}{((1+\varepsilon)^2-p_2^2)^3}+\frac{\varepsilon^\theta((1+\varepsilon)^2+3p_2^2)}{((1+\varepsilon)^2-p_2^2)^4}.
 \end{split}
\end{align}
Hence $F_\varepsilon$ is uniformly convex, e.g. via Sylvestre's criterion.

\subsection{Global extension}\label{sec:global_extension}

Next we define the compact sets 
\begin{align}\label{eq:definition_of_K_epsilon}
K^\varepsilon:=\set{p\in\R^2:\abs{p_1}\leq \varepsilon^{-4\theta},~\abs{p_2}\leq 1+\varepsilon-\varepsilon^{4\theta}}
\end{align}
and extend $F_{\varepsilon|K^\varepsilon}$ in a uniformly elliptic way onto all of $\R^2$, with some additional properties. The stated uniform bound in  \eqref{eq:uniform_lambda_min_for_first_extension} for instance will be used to achieve $\Gamma$-convergence in Section \ref{sec_gamma_convergence}, while property \eqref{eq:estimate_of_extension_for_derivative} will be needed in order to conclude the energy balance in Section \ref{sec_further_properties}.
\begin{lemma}\label{lem:first_extension_of_F_epsilon}
For every $\varepsilon\in(0,1)$ there exists a smooth extension $\hat{F}_\varepsilon:\R^2\rightarrow [0,\infty)$ of $F_{\varepsilon|K^\varepsilon}$ satisfying
\begin{enumerate}[(i)]
    \item \label{eq:uniform_elliptic_extension} $\lambda_\varepsilon\id\leq D^2\hat{F}_\varepsilon(p)\leq \Lambda_\varepsilon\id$ for all $p\in\R^2$ with some constants $0<\lambda_\varepsilon,\Lambda_\varepsilon<\infty$,
    \item \label{eq:uniform_lambda_min_for_first_extension}
    $\lambda_0\id\leq D^2\hat{F}_\varepsilon(p)$ for $\abs{p_2}\geq 1+\varepsilon-\varepsilon^{4\theta}$ or $\abs{p_1}\geq 1$ with a constant $\lambda_0>0$ independent of $\varepsilon\in(0,1)$,
    \item \label{eq:estimate_of_extension_for_derivative}
    $-\varepsilon\leq \partial_{p_1}\hat{F}_\varepsilon(p)p_1\leq 3 \hat{F}_\varepsilon(p)$ for all $p\in\R^2$.
\end{enumerate}
\end{lemma}
The proof of Lemma \ref{lem:first_extension_of_F_epsilon} relies on Lemmas \ref{lem:smallest_convex_extension}, \ref{lem:convolution_of_smallest_extension}.
\begin{lemma}\label{lem:smallest_convex_extension}
Let $N_\varepsilon:=\max F_{\varepsilon|K^\varepsilon}+1$, $K_{N_\varepsilon}:=F_\varepsilon^{-1}((-\infty,N_\varepsilon])$.
There exists a convex and globally Lipschitz extension $\tilde{F}_\varepsilon:\R^2\rightarrow [0,\infty)$ of $F_{\varepsilon|K_{N_\varepsilon}}$ satisfying
\begin{align}\label{eq:estimate_for_smallest_extension}
0\leq \partial_{p_1}\tilde{F}_\varepsilon(p)p_1\leq 2\tilde{F}_\varepsilon(p)
\end{align}
for almost every $p\in \R^2$. Moreover, on $\partial K_{N_\varepsilon}$ the extension $\tilde{F}_\varepsilon$ is differentiable  with $\nabla \tilde{F}_\varepsilon=\nabla F_\varepsilon$.
\end{lemma}
\begin{proof} We abbreviate $N=N_\varepsilon$ and consider the smallest convex extension of $F_{\varepsilon|K_N}$ defined by $\tilde F_\varepsilon:\R^2\rightarrow\R$, 
$$\tilde F_\varepsilon (p)=\sup_{\tilde p\in\partial K_N}\{F_\varepsilon (\tilde p)+\nabla F_\varepsilon (\tilde p)\cdot(p-\tilde p)\}.$$
A priori $\tilde{F}_\varepsilon$ as a convex function is only locally Lipschitz, but since the subdifferential of $\tilde{F}_\varepsilon$ is given by
\begin{align}\label{eq:characterization_of_subdifferential}
\partial \tilde{F}_\varepsilon(p)=\set{\nabla F_\varepsilon(\bar{p}):\bar{p}\in\partial K_N,~\tilde{F}_\varepsilon(p)=F_\varepsilon(\bar{p})+\nabla F_\varepsilon(\bar{p})\cdot(p-\bar{p})}^{co},
\end{align} cf. \cite[Theorem 2.4.18]{zalinescu}, and $K_N$ is compact, it follows that  
$\|\nabla \tilde{F}_\varepsilon\|_{L^\infty(\R^2)}$ is finite.

Moreover, the strict convexity of $F_\varepsilon$ and \eqref{eq:characterization_of_subdifferential} imply 
\begin{align*}
    \partial \tilde{F}_\varepsilon(p)=\set{\nabla F_\varepsilon(p)}\text{ for all }p\in\partial K_{N_{\varepsilon}}.
\end{align*}

Next, knowing that $0\leq \partial_{p_1}F_\varepsilon(p)p_1\leq 2 F_\varepsilon(p)$, $p\in \R^2$, $\abs{p_2}<1+\varepsilon$ holds true, let us prove \eqref{eq:estimate_for_smallest_extension} whenever $\tilde{F}_\varepsilon$ is differentiable at $p\in\R^2$. 

In order to do this let us suppose that the supremum is achieved at some $\bar p=\bar p(p)\in\partial K_N$ and let us write $F_\varepsilon=\frac{1}{2}e^g$, with 
\[
g(p)=g_\varepsilon (p)= \log (p_1^2+\varepsilon^\theta)-\log((1+\varepsilon)^2-p_2^2).
\]
We obtain that
\begin{align}\label{eq:achieved}
    \tilde F_\varepsilon (p)=N(1+\nabla g(\bar p)\cdot (p-\bar p))
\end{align}
and under the assumption of differentiability that $\nabla \tilde{F}_\varepsilon(p)=N\nabla g(\bar{p})$.

Since $\bar p\in \partial K_N=\{ F_\varepsilon=N\}$, we have \begin{align}\label{eq:cstr}
\bar p_1^2+\varepsilon^\theta - 2N((1+\varepsilon)^2-\bar p_2^2)=0,\end{align} and hence we may rewrite
\begin{align}\label{eq:maxble}
    \nabla g(\bar p)\cdot (p-\bar p)=\frac{\bar p_1 p_1+2N\bar p_2 p_2-(2N(1+\varepsilon)^2-\varepsilon^\theta)}{N((1+\varepsilon)^2-\bar p_2^2)}.
\end{align}
We first of all observe that 
$\partial_{p_1}\tilde{F}_\varepsilon(p)p_1=N\partial_{p_1}g(\bar{p})p_1\geq 0$ as otherwise $(-\bar{p}_1,\bar{p}_2)$ would be a better choice for the supremum.

For the upper bound we distinguish two cases.

\noindent \underline{Case (a):} $p_2-\bar p_2< 0$. Since $\bar p$ maximizes \eqref{eq:maxble} under the constraint \eqref{eq:cstr}, there exists a Lagrange-multiplier $\lambda\in \R$ such that there holds
\begin{gather}\label{eq:lagr}
\begin{gathered}
    \frac{p_1}{N((1+\varepsilon)^2-\bar p_2^2)}-2\lambda \bar p_1=0,\\ \frac{2N p_2}{N((1+\varepsilon)^2-\bar p_2^2)}+\nabla g(\bar p)\cdot (p-\bar p)\frac{2N\bar p_2}{N((1+\varepsilon)^2-\bar p_2^2)}-4N\lambda \bar p_2=0.
    \end{gathered}
\end{gather}

We may assume without loss of generality that neither $\bar p_1$ or $\bar p_2$ is zero. 
Indeed, if $\bar p_1=0$, from \eqref{eq:lagr} it follows that $p_1=0$, and hence the inequality that we want to prove follows trivially. If $\bar p_2=0$, from \eqref{eq:lagr} once more it follows that $p_2=0$, which contradicts $p_2-\bar p_2< 0$.

Expressing $\lambda$ from both equations and simplifying leads to
\begin{align}\label{eq:szep}
    \nabla g(\bar p)\cdot (p-\bar p)=\frac{p_1}{\bar p_1}-\frac{p_2}{\bar p_2}.
\end{align}
As described we need to estimate $\partial_{p_1}g(\bar p)p_1$, which is given by 
\begin{multline*}
    \partial_{p_1}g(\bar p)p_1=2\frac{\bar p_1 p_1}{\bar p_1^2+\varepsilon^\theta}\leq 2\frac{p_1}{\bar p_1} = 2\nabla g(\bar p)\cdot (p-\bar p)+ 2\frac{p_2}{\bar p_2}
    \leq 2\nabla g(\bar p)\cdot (p-\bar p)+2,
\end{multline*}
where we have used \eqref{eq:szep} and the fact that we are  in the case $\frac{p_2}{\bar p_2}< 1$.
Hence, using \eqref{eq:achieved}, it follows that
\begin{align*}
  \partial_{p_1}\tilde{F}_\varepsilon(p)p_1 = N\partial_{p_1}g(\bar p)p_1 \leq 2N (\nabla g(\bar p)\cdot (p-\bar p)+1)=2\tilde{F}_\varepsilon(p),
\end{align*}
which is the desired inequality.

\noindent \underline{Case (b):} $p_2-\bar p_2\geq 0$. Observe that since $p_2\mapsto F_\varepsilon(p)$ is even, we may assume without loss of generality that $p_2\geq 0$.
Replacing in \eqref{eq:maxble} the denominator by $\frac{1}{2}(\bar{p}_1^2+\varepsilon^\theta)$ via \eqref{eq:cstr} it follows that $\bar p_2\geq 0$ as otherwise $(\bar{p}_1,-\bar{p}_2)$ would be better.
This however implies that 
$$\partial_{p_2}g(\bar p)=\frac{2\bar p_2}{(1+\varepsilon)^2-\bar p_2^2}\geq 0.$$
We may then write
\begin{align*}
    \partial_{p_1}g(\bar p)p_1\leq \nabla g(\bar p)\cdot (p-\bar p)+ \partial_{p_1}g(\bar p)\bar p_1=\nabla g(\bar p)\cdot (p-\bar p)+2\frac{\bar p_1^2}{\bar p_1^2+\varepsilon^\theta}\\ \leq \nabla g(\bar p)\cdot (p-\bar p)+2 \leq 2\nabla g(\bar p)\cdot (p-\bar p)+2,
\end{align*}
from where we conclude as in the previous case. This finishes the proof of the lemma.
\end{proof}

\begin{lemma}\label{lem:convolution_of_smallest_extension}
Let $\varphi_\eta\in\cC^\infty_c(B_\eta(0))$, $\eta>0$ be a standard symmetric mollifier and $\tilde{F}_\varepsilon$ from Lemma \ref{lem:smallest_convex_extension}. The convolution $\tilde{F}_\varepsilon^\eta:=\varphi_\eta *\tilde{F}_\varepsilon$ satisfies $\tilde{F}_\varepsilon^\eta\geq \tilde{F}_\varepsilon$ and
\begin{align}\label{eq:estimate_for_convolution_of_smallest_extension}
    -c_\varepsilon\eta\leq \partial_{p_1}\tilde{F}_\varepsilon^\eta(p)p_1\leq 2\tilde{F}_\varepsilon^\eta(p)+c_\varepsilon \eta,
\end{align}
where $c_\varepsilon>0$ depends only on $\varepsilon$. Moreover, there exists $\Lambda_{\varepsilon,\eta}>0$ with
\begin{align}\label{eq:bound_on_2nd_derivative_of_convoluted_smallest_ext}
    0\leq D^2\tilde{F}_\varepsilon^\eta (p)\leq \Lambda_{\varepsilon,\eta}\id,~p\in\R^2.
\end{align}
\end{lemma}
\begin{proof} Denoting by $c_\varepsilon>0$ the global Lipschitz constant of $\tilde{F}_\varepsilon$ and using \eqref{eq:estimate_for_smallest_extension} one concludes
\begin{align*}
    \partial_{p_1}\tilde{F}^\eta_\varepsilon(p)p_1&=\int_{\R^2}\varphi_\eta (q)\partial_{p_1}\tilde{F}_\varepsilon(p-q)(p_1-q_1)\:dq+\int_{\R^2}\varphi_\eta (q)\partial_{p_1}\tilde{F}_\varepsilon(p-q)q_1\:dq\\
    &\leq 2\tilde{F}_\varepsilon^\eta(p)+c_\varepsilon\eta
\end{align*}
and $\partial_{p_1}\tilde{F}^\eta_\varepsilon(p)p_1\geq -c_\varepsilon \eta$. 

Moreover, since $\tilde{F}^\eta_\varepsilon$ is a convex smooth function there holds
\begin{align*}
    0\leq v^TD^2\tilde{F}_\varepsilon^\eta(p)v=\int_{\R^2}(\nabla\varphi(p-q)\cdot v) (\nabla \tilde{F}_\varepsilon(q)\cdot v)\:dq\leq \tilde{c}_\varepsilon\int_{\R^2}\abs{\nabla \varphi_\eta(q)}\:dq\abs{v}^2
\end{align*}
for any $v\in\R^2$.
\end{proof}
\begin{proof}[Proof of Lemma \ref{lem:first_extension_of_F_epsilon}] We will now construct the extension $\hat{F}_\varepsilon$.
Let $N_\varepsilon$, $K_{N_\varepsilon}$, $\tilde{F}_\varepsilon$ and $\tilde{F}_\varepsilon^\eta$, $\eta>0$ be as in Lemma \ref{lem:smallest_convex_extension}, Lemma \ref{lem:convolution_of_smallest_extension} respectively. 
Furthermore, consider $Q_\varepsilon,\check F_\varepsilon,\check F^\eta_\varepsilon:\R^2\rightarrow \R$ defined by
\begin{gather*}
Q_\varepsilon(p):=\frac{p_1^2+2N_\varepsilon p_2^2}{2N_\varepsilon(1+\varepsilon)^2-\varepsilon^\theta},\\
\check F_\varepsilon(p):=\tilde{F}_\varepsilon(p)+\frac{N_\varepsilon}{4}\left(Q_\varepsilon(p)-1\right),\quad 
\check F^\eta_\varepsilon(p):=\tilde{F}^\eta_\varepsilon(p)+\frac{N_\varepsilon}{4}\left(Q_\varepsilon(p)-1\right).
\end{gather*}
Note that $Q_\varepsilon=1$ on $\partial K_{N_\varepsilon}$ by the definition of $K_{N_\varepsilon}$. 

We claim that there exist constants $C=C_\varepsilon>0$ and $\delta'=\delta'_\varepsilon>0$, such that for any $\delta\in (0,\delta')$ and $q\in \R^2$ with $\dist (q,\partial K_{N_\varepsilon})=\delta$ there holds
\begin{align}\label{eq:difference_of_extensions_estimate}
    F_\varepsilon(q)-\check{F}_\varepsilon(q)\begin{cases}
    \leq -C\delta,&q\notin K_{N_{\varepsilon}},\\
    \geq C\delta,& q\in K_{N_{\varepsilon}}.
    \end{cases}
\end{align}
For $q\in K_{N_\varepsilon}$, where also $\check F_\varepsilon$ is smooth, this is a straightforward consequence of the fact that on $\partial K_{N_\varepsilon}$ the functions $F_\varepsilon$, $\check{F}_\varepsilon$ coincide, while their gradients are related via $\nabla \check{F}_\varepsilon=\nabla F_\varepsilon+\frac{N_\varepsilon}{4}\nabla Q_\varepsilon$, cf. Lemma \ref{lem:smallest_convex_extension}.
For $q\notin K_{N_\varepsilon}$ we can argue by convexity instead. Indeed, let $p\in\partial K_{N_\varepsilon}$ and $q=p+\delta\frac{\nabla Q_\varepsilon(p)}{\abs{\nabla Q_\varepsilon(p)}}$. Note that any $q\notin K_{N_\varepsilon}$ with $\dist(q,\partial K_{N_\varepsilon})=\delta$ can be written like this. There holds
\[
\check F_\varepsilon(q)-F_\varepsilon(p)\geq (\nabla \check F_\varepsilon(p)-\nabla F_\varepsilon(p))\cdot (q-p)+o(\delta)=\frac{N_\varepsilon}{4}\abs{\nabla Q_\varepsilon(p)}\delta +o(\delta)
\]
with an error uniform in $p\in \partial K_{N_\varepsilon}$. Thus \eqref{eq:difference_of_extensions_estimate} follows.

Next we fix $\delta\in(0,\delta')$ small, such that $B_\delta (K_{N_\varepsilon})\subset\subset \set{p\in\R^2:\abs{p_2}<1+\varepsilon}$, $\dist(K^\varepsilon,B_\delta(\partial K_{N_\varepsilon}))>0$ and such that 
\begin{align}\label{eq:choice_of_deltaaaa}
    \frac{3}{4}N_\varepsilon\leq \check F_\varepsilon(p)\text{ for all }p\notin K_{N_\varepsilon}\text{ and all }p\in K_{N_\varepsilon}\text{ with }\dist(p,\partial K_{N_\varepsilon})\leq \delta.
\end{align}

We then choose the convolution scale $\eta>0$ such that $\check F_\varepsilon^\eta$ satisfies
\begin{align}\label{eq:differences_in_extensions_2}
    F_\varepsilon(q)-\check F_\varepsilon^\eta (q)\begin{cases}
    \leq -C\delta/2,&q\notin K_{N_{\varepsilon}},~\dist(q,\partial K_{N_\varepsilon})=\delta,\\
    \geq C\delta/2,& q\in K_{N_{\varepsilon}},~\dist(q,\partial K_{N_\varepsilon})=\delta.
    \end{cases}
\end{align}

Finally we define the extension $\hat{F}_\varepsilon:\R^2\rightarrow\R$ by
\[
\hat{F}_\varepsilon(p):=\begin{cases}
F_\varepsilon(p),&p\in K_{N_\varepsilon},~\dist(p,\partial K_{N_\varepsilon})>\delta,\\
\max_{\tilde\eta}\big(F_\varepsilon(p),\check F_\varepsilon^\eta(p)\big),&\dist(p,\partial K_{N_\varepsilon})\leq \delta,\\
\check F_\varepsilon^\eta(p),&p\notin K_{N_\varepsilon},~\dist(p,\partial K_{N_\varepsilon})>\delta.
\end{cases}
\]
Here $\max_{\tilde\eta}$, $\tilde\eta>0$ is a sufficiently sharp convolution of the maximum of two numbers, that is
\[
\text{max}_{\tilde \eta}(y_1,y_2):=\int_{\R^2}\varphi_{\tilde \eta}(y_1-a_1,y_2-a_2)\max\set{a_1,a_2}\:da.
\]
By this definition and \eqref{eq:differences_in_extensions_2} we see that $\hat F_\varepsilon$ is indeed a smooth function for any $\tilde \eta>0$ chosen smaller than $C\delta/4$. Also $\hat{F}_{\varepsilon|K^\varepsilon}=F_{\varepsilon|K^\varepsilon}$ by the choice of $\delta$.
It thus remains to verify properties \eqref{eq:uniform_elliptic_extension}-\eqref{eq:estimate_of_extension_for_derivative}.

We begin with \eqref{eq:estimate_of_extension_for_derivative}. $F_\varepsilon$ clearly satisfies \eqref{eq:estimate_of_extension_for_derivative} with factor $2$ on the right-hand side. Moreover, the definition of $Q_\varepsilon$ and \eqref{eq:estimate_for_convolution_of_smallest_extension}, \eqref{eq:choice_of_deltaaaa} imply
\begin{align*}
    -c_\varepsilon \eta\leq \partial_{p_1}\check{F}^\eta_\varepsilon(p)p_1&\leq 2\check{F}_\varepsilon^\eta(p)+\frac{N_\varepsilon}{2}+c_\varepsilon\eta\leq \left(2+\frac{2}{3}\right)\check{F}_\varepsilon^\eta(p)+c_\varepsilon\eta
\end{align*}
for all $p\notin K_{N_\varepsilon}$ and all $p\in K_{N_\varepsilon}$ with $\dist(p,\partial K_{N_\varepsilon})\leq \delta$. Hence by shrinking $\eta$ further we obtain that $\check F_\varepsilon^\eta(p)$ satisfies \eqref{eq:estimate_of_extension_for_derivative} with factor $2+\frac{3}{4}$ on the right-hand side for said points $p$. 

It remains to observe that in the transition zone $\dist(p,\partial K_{N_\varepsilon})\leq \delta$ the gradient of $\hat F_\varepsilon$ is given by the convex combination
\begin{align*}
    \nabla \hat F_\varepsilon(p)&=\nabla F_\varepsilon(p)\int_{\set{a_1>a_2}}\varphi_{\tilde\eta}(F_\varepsilon(p)-a_1,\check F_\varepsilon^\eta(p)-a_2)\:da\\
    &\hspace{45pt}+\nabla \check F^\eta_\varepsilon(p)\int_{\set{a_1<a_2}}\varphi_{\tilde \eta}(F_\varepsilon(p)-a_1,\check F_\varepsilon^\eta(p)-a_2)\:da\\
    &=:\lambda(p)\nabla F_\varepsilon(p)+(1-\lambda(p))\nabla \check F_\varepsilon^\eta(p),
\end{align*}
while the values satisfy
\begin{align*}
    \abs{\hat F_\varepsilon(p)-\left(\lambda(p)F_\varepsilon(p)+(1-\lambda(p))\check F_\varepsilon^\eta(p)\right)}\leq \tilde \eta.
\end{align*}
Thus by shrinking $\tilde \eta$ we deduce property \eqref{eq:estimate_of_extension_for_derivative} for $\hat{F}_\varepsilon$.

Next we turn to \eqref{eq:uniform_lambda_min_for_first_extension}. Regarding $\check F^\eta_\varepsilon(p)$ we have
\begin{align*}
    \lamin\left(D^2\check F^\eta_\varepsilon(p)\right)\geq \frac{N_\varepsilon}{4}\lamin\left(D^2Q_\varepsilon(p)\right)=\frac{N_\varepsilon}{4N_\varepsilon (1+\varepsilon)^2-2\varepsilon^\theta}\geq \frac{1}{16}.
\end{align*}

Regarding the original $F_\varepsilon$, observe that $1+\varepsilon>\abs{p_2}\geq 1+\varepsilon-\varepsilon^{4\theta}$ implies 
\begin{align}\label{eq:estimate_on_F_varepsilon}
F_\varepsilon(p)\geq \frac{\varepsilon^\theta}{2\cdot 4\varepsilon^{4\theta}}\geq \frac{1}{8}.
\end{align}
The same estimate also holds true for $p\in\R^2$, $\abs{p_2}<1+\varepsilon$ with $\abs{p_1}\geq 1$. 

For $p\in\R^2$, $\abs{p_2}<1+\varepsilon$ such that \eqref{eq:estimate_on_F_varepsilon} holds true we  abbreviate the quantity $\sigma:=(1+\varepsilon)^2-p_2^2\leq 4$ and estimate by means of \eqref{eq:Fepsder} the minimal eigenvalue as follows
\begin{align*}
    \lamin(D^2F_\varepsilon(p))&\geq \frac{\det(D^2F_\varepsilon(p))}{\tr(D^2F_\varepsilon(p))}=\frac{p_1^2+\sigma^{-1}\varepsilon^\theta ((1+\varepsilon)^2+3p_2^2)}{\sigma^2+(p_1^2+\varepsilon^\theta)((1+\varepsilon)^2+3p_2^2)}\\
    &\geq \frac{4^{-1}p_1^2+4^{-1}\varepsilon^\theta}{4\sigma +16(p_1^2+\varepsilon^\theta)}=\frac{1}{8F_\varepsilon(p)^{-1}+64}\geq \frac{1}{128}.
\end{align*}

Thus we set $\lambda_0=\frac{1}{128}$ and it remains to check the behaviour of the minimal eigenvalue in the transition zone $\dist(p,\partial K_{N_\varepsilon})\leq \delta$. From what we have seen it follows that for both functions $f_1:=F_\varepsilon$, $f_2:=\check F^\eta_\varepsilon$  there holds
\[
f_i(p)-a_i\geq f_i(p_0)-a_i+\nabla f_i(p_0)\cdot(p-p_0)+\frac{\lambda_0}{2}\abs{p-p_0}^2
\]
for all points $p,p_0\in\R^2$ which are $\delta$-close to $\partial K_{N_\varepsilon}$ and $a_i\in\R$. Therefore 
\begin{align*}
    \hat F_\varepsilon(p)&=\int_{A:=\set{f_1(p_0)-a_1>f_2(p_0)-a_2}} \varphi_{\tilde{\eta}}(a)\max\set{f_1(p)-a_1,f_2(p)-a_2}\:da\\
    &\hspace{15pt}+\int_{B:=\set{f_1(p_0)-a_1<f_2(p_0)-a_2}} \varphi_{\tilde{\eta}}(a)\max\set{f_1(p)-a_1,f_2(p)-a_2}\:da\\
    &\geq \int_A \varphi_{\tilde{\eta}}(a)\left(f_1(p_0)-a_1+\nabla f_1(p_0)\cdot (p-p_0)+\frac{\lambda_0}{2}\abs{p-p_0}^2\right)\:da\\
    &\hspace{15pt}+\int_{B} \varphi_{\tilde{\eta}}(a)\left(f_2(p_0)-a_2+\nabla f_2(p_0)\cdot (p-p_0)+\frac{\lambda_0}{2}\abs{p-p_0}^2\right)\:da\\
    &=\hat{F}_\varepsilon(p_0)+\nabla \hat{F}_\varepsilon(p_0)\cdot (p-p_0)+\frac{\lambda_0}{2}\abs{p-p_0}^2,
\end{align*}
which shows \eqref{eq:uniform_lambda_min_for_first_extension}.

Finally property \eqref{eq:uniform_elliptic_extension} follows in a similar way by observing that $F_\varepsilon$ and $\check{F}_\varepsilon^\eta$ are uniformly elliptic with $\varepsilon$-dependent bounds, cf. also \eqref{eq:bound_on_2nd_derivative_of_convoluted_smallest_ext}.
\end{proof}

\subsection{The regularized variational problem}\label{sec:regularized_var_problem}

Our regular approximation of the action functional $\cA(u)$ defined in \eqref{eq:functional_in_sec_3} then reads
\begin{equation}\label{eq:def_action_functional_approximation}
\cA_\varepsilon(u):=\int_\Omega \hat{F}_\varepsilon(\nabla u)-V(x,u)\:dx
\end{equation}
with $\hat F_\varepsilon$ from Lemma \ref{lem:first_extension_of_F_epsilon}.

Moreover, we will not only use an approximation of the integrand, but also introduce $\varepsilon$-dependent boundary data, which is used for the a priori bounds in Section \ref{sec_apriori_bounds}.  We set
\begin{align*}
    X_\varepsilon&:=\set{u\in H^1(\Omega):u(\cdot,\pm L)=0,~u(0,\cdot)=-U_\varepsilon,~u(T,\cdot)=U_\varepsilon},
\end{align*}
where $U_\varepsilon:[-L,L]\rightarrow\R$,
\begin{align}\label{eq:pertid}
U_\varepsilon(x_2):=L-\abs{x_2}+\frac{\varepsilon^\beta}{2L}(L^2-x^2_2)
\end{align}
for some fixed constant $\beta\in(1,3-\theta)$. The boundary data is again attained in the trace sense. Also observe that $\cA_\varepsilon$ is well-defined on all of $H^1(\Omega)$ due to the uniform ellipticity of $\hat{F}_\varepsilon$.
We then consider the corresponding regularized minimization problem
\begin{align}\label{eq:Dmin}
\text{find }u_\varepsilon\in X_\varepsilon\text{ such that }  \cA_\varepsilon (u_\varepsilon)=\inf_{u\in X_\varepsilon} \cA_\varepsilon (u).
\end{align}

 The uniform ellipticity of the approximations allow us to conclude the existence of sufficiently smooth minimizers in a standard way.

\begin{lemma}\label{lem:exist}
Problem \eqref{eq:Dmin} admits a  solution. Every solution belongs to $\cC^0(\overline{\Omega})$, as well as $\cC^{2,\alpha}(K)$ for any compact $K$ contained in $\overline{\Omega}$ and having a positive distance to $\set{(0,\pm L),(T,\pm L),(0,0),(T,0)}$.
\end{lemma}
\begin{proof}
Lemma \ref{lem:uniform_energy_bound} below in particular shows that $\inf_{X_\varepsilon}\cA_\varepsilon$ is finite. The existence of a minimizer $u_\varepsilon\in X_\varepsilon$ of $\cA_\varepsilon$ then follows by the boundedness of $V$ and the uniform convexity of $\hat{F}_\varepsilon$ (e.g. \cite[Chapter 8.2]{Evans_book}).
Due to the ellipticity condition Lemma \ref{lem:first_extension_of_F_epsilon} \eqref{eq:uniform_elliptic_extension} and \eqref{eq:condition_V1}
the regularity follows in the classical manner, e.g. from $u_\varepsilon\in H^1$ to $u_\varepsilon \in H^2$ to $\nabla u_\varepsilon\in \cC^{0,\alpha}$ to $u_\varepsilon\in\cC^{2,\alpha}$, see \cite{Evans_book,Gilbarg_Trudinger}. The points excluded are the points where either the boundary or the boundary data lacks the necessary smoothness.
\end{proof}

\begin{lemma}\label{lem:uniform_energy_bound}
There holds
\[
\sup_{\varepsilon\in(0,1)}\inf_{u\in X_\varepsilon} \cA_\varepsilon(u)<\infty.
\]
\end{lemma}
\begin{proof}
Define $w_\varepsilon:\Omega\rightarrow\R$, $w_\varepsilon(x)=-U_\varepsilon(x_2)\cos\left(\pi x_1/T\right)$. Then $w_\varepsilon\in X_\varepsilon$ and  
\[
\abs{\partial_{x_1} w_\varepsilon}\leq \frac{2L\pi}{T},\quad \abs{\partial_{x_2}w_\varepsilon}\leq 1+\varepsilon^\beta.
\]
Hence for sufficiently small $\varepsilon>0$ we have $\nabla w^\varepsilon(x)\in K^\varepsilon$ and therefore
\begin{align*}
    \int_\Omega \hat{F}_\varepsilon(\nabla w_\varepsilon(x))\:dx&=\int_\Omega F_\varepsilon(\nabla w_\varepsilon(x))\:dx\\
    &=\int_\Omega \frac{U_\varepsilon(x_2)^2\sin^2(\pi x_1/T)^2\pi^2/T^2+\varepsilon^\theta}{(1+\varepsilon)^2-U_\varepsilon'(x_2)^2\cos^2(\pi x_1/T)}\:dx\\
    &\leq \frac{2L\pi^2U_\varepsilon(0)^2}{T(1+\varepsilon)^2}+\frac{2LT\varepsilon^\theta}{(1+\varepsilon)^2-(1+\varepsilon^\beta)^2}.
\end{align*}
The right-hand side is uniformly bounded as $\varepsilon\rightarrow 0$, since the exponents $\theta,\beta$ are both bigger than $1$.

The uniform boundedness of $\cA_\varepsilon(w_\varepsilon)$ follows since $V$ is uniformly bounded due to condition \eqref{eq:condition_V1}.
\end{proof}

\subsection{A priori bounds}\label{sec_apriori_bounds}

Next we establish some first a priori bounds for solutions of the regularized problem \eqref{eq:Dmin}. Let $u_\varepsilon$ be such a solution. In view of Lemma \ref{lem:exist}, $u_\varepsilon$ is a classical, and therefore in particular also a viscosity solution of the associated Euler-Lagrange equation
\begin{align}\label{eq:euler_lagrange_equ_epsilon}
   \divv(\nabla \hat{F}_\varepsilon(\nabla u))+\partial_z V(x,u)= D^2 \hat F_\varepsilon(\nabla u):D^2 u+\partial_zV(x,u)=0.
\end{align}
We quickly recall the notion of being a solution in viscosity sense.
\begin{definition}\label{def:viscosity_solution}
A viscosity subsolution of \eqref{eq:euler_lagrange_equ_epsilon} is a continuous function $u:\Omega\rightarrow \R$, such that whenever $u$ is touched in a point $x_0\in \Omega$ from above by a function $\varphi\in \cC^2(B_\delta(x_0))$, then 
\[
D^2\hat{F}_\varepsilon(\nabla\varphi(x_0)):D^2\varphi(x_0)+\partial_zV(x_0,\varphi(x_0))\geq 0.
\]
If the above inequality is strict in every such situation we say that $u$ is a strict viscosity subsolution.
The notion of a (strict) supersolution is defined analogously, and a viscosity solution is both a  viscosity sub- und supersolution.
\end{definition}
\begin{lemma}\label{lem:viscosity_sub_supersolution}
For $\varepsilon>0$ sufficiently small the functions $\Omega\ni x\mapsto U_\varepsilon(x_2)+c\in \R$, $c\in\R$ are strict viscosity supersolutions. The corresponding functions induced by $-U_\varepsilon$ are strict viscosity subsolutions. 
\end{lemma}
\begin{proof} Consider $U_\varepsilon$ as a $x_1$-independent function defined on $\Omega$. For $x_0\in \Omega$ with $x_{0,2}\neq 0$ there holds
\begin{align*}
    D^2\hat{F}_\varepsilon(\nabla &U_\varepsilon(x_0)):\nabla^2 U_\varepsilon(x_0)=-\frac{\varepsilon^\beta}{L}\partial_{p_2}^2F_\varepsilon(\nabla U_\varepsilon(x_0))\\
&\leq \frac{-\varepsilon^{\beta+\theta}}{L\big((1+\varepsilon)^2-(\sign(x_{0,2})+\frac{\varepsilon^\beta}{L}x_{0,2})^2\big)^3}\leq -C \varepsilon^{\beta+\theta-3}
\end{align*}
for a suitable constant $C>0$ independent of $x_0$. By the choice of the exponents $\theta\in (1,2)$, $\beta \in (1,3-\theta)$ and the boundedness of $\partial_z V$, cf. \eqref{eq:condition_V1}, we deduce that 
\[
    D^2\hat{F}_\varepsilon(\nabla U_\varepsilon(x_0)):\nabla^2 U_\varepsilon(x_0)+\partial_z V(x_0,U_\varepsilon(x_0)+c)< 0
\]
provided $\varepsilon\in(0,\varepsilon_0)$ for some $\varepsilon_0>0$ small enough independent of $x_0$ and $c$. The same inequality holds true for any $\cC^2$ function $\varphi$ touching $U_\varepsilon+c$ from below in $x_0$. Note that $U_\varepsilon+c$ can only be touched by a $\cC^2$ function from below in points with $x_{0,2}\neq 0$. Thus $U_\varepsilon+c$ is a strict viscosity supersolution. Similarly one concludes that $-U_\varepsilon+c$ is a strict viscosity subsolution.
\end{proof}

\begin{corollary}\label{cor:maximum_principle}
For $\varepsilon>0$ sufficiently small any minimizer $u_\varepsilon\in X_\varepsilon$ of $\cA_\varepsilon$ satisfies $\abs{u_\varepsilon(x)}\leq U_\varepsilon(x_2)$ for all $x\in \overline{\Omega}$. Moreover, the inequality is strict for $x\in \Omega$.
\end{corollary}
\begin{proof} Consider again $U_\varepsilon$ as a function defined on all of $\overline{\Omega}$. Assuming to the contrary that $u_\varepsilon$ touches $U_\varepsilon+c$ for some $c\geq 0$ from below in a point $x_0\in \Omega$ with $x_{0,2}\neq 0$ (touching in a point with $x_{0,2}=0$ again is not possible) directly gives a contradiction to the fact that $U_\varepsilon+c$ is a strict viscosity supersolution and $u_\varepsilon\in\cC^2(\Omega)$ is a solution. Hence $u_\varepsilon<U_\varepsilon$ on $\Omega$.
The analogue statement regarding the comparison of $u_\varepsilon$ with $-U_\varepsilon$ is obtained similarly.
\end{proof}

\subsection{Autonomous potentials}\label{sec_autonomous_potentials_for_approximation}
In this section we conclude under condition \eqref{eq:V1prime} the positiveness of $\partial_{x_1} u_\varepsilon$ and the existence of a first integral, which is the total energy. Moreover, we lay the basis for Lemma \ref{lem:good_set_property}.   

\begin{lemma}\label{lem:partial_1_u_positive}
Assume that $V$ in addition to \eqref{eq:condition_V1} satisfies \eqref{eq:V1prime} and let $u_\varepsilon$ be a solution of \eqref{eq:Dmin}. Then $\partial_{x_1}u_\varepsilon \geq 0$ on $\Omega$ for $\varepsilon>0$ sufficiently small.
\end{lemma}
\begin{proof} Consider the function $w:=u_\varepsilon+U_\varepsilon$, which is non-negative by Corollary \ref{cor:maximum_principle}. As before $U_\varepsilon$ is considered here as a $x_1$-independent function defined on $\overline{\Omega}$. For every $x\in\Omega$, $x_2\neq 0$ Lemma \ref{lem:viscosity_sub_supersolution} implies
\[
\divv (A(x)\nabla w(x))+c(x)w(x)<0,
\]
where
\[
A(x):=\int_0^1D^2\hat{F}_\varepsilon(-\nabla U_\varepsilon+s\nabla w)\:ds,\quad c(x):=\int_0^1\partial_z^2V(x,-U_\varepsilon+sw)\:ds.
\]
Splitting the $0$th order term into $c=c^+-c^-$ with $c^+,c^-\geq 0$ and neglecting $c^+w\geq 0$ the Hopf maximum principle implies that $\partial_{x_1}u_\varepsilon(0,x_2)=\partial_{x_1}w(0,x_2)>0$ for $x_2\in(-L,L)$, $x_2\neq 0$. Similarly one sees that $\partial_{x_1}u_\varepsilon(T,x_2)>0$ for $x_2\in(-L,L)$, $x_2\neq 0$.

Next we claim that $(\partial_{x_1}u_\varepsilon)^-\in H^1_0(\Omega)$. Indeed, the difference quotients 
\[
v_h(x):=\frac{u_\varepsilon(x_1+h,x_2)-u_\varepsilon(x)}{h}
\]
satisfy $v_h^-\in H^1_0(\Omega_h)$, $\Omega_h:=(0,T-h)\times(-L,L)$ by Corollary \ref{cor:maximum_principle}. We therefore can use $v_h^-$ as a test function for the equation
\begin{align}\label{eq:d1EL}
    \divv(B(x)\nabla v_h(x))+d(x)v_h(x)=0,\quad x\in\Omega_h,
\end{align}
where 
\begin{align*}
B(x)&:=\int_0^1D^2\hat{F}_\varepsilon(\nabla u_\varepsilon(x)+sh\nabla v_h(x))\:ds,\\ d(x)&:=\int_0^1\partial_z^2V(x,u_\varepsilon(x)+shv_h(x))\:ds.
\end{align*}
Note that the potential $V$ is autonomous with respect to $x_1$ by assumption \eqref{eq:V1prime}. Testing now \eqref{eq:d1EL} with $v_h^-$ we obtain
\begin{align*}
   0&= \int_{\set{v_h< 0}} B(x) \nabla v_h \cdot\nabla v_h^- - d(x) v_hv_h^- \, dx \\
   &= \int_{\Omega_h} B(x) \nabla v_h^- \cdot\nabla v_h^- - d(x)(v_h^-)^2 \, dx.
\end{align*}
Hence the uniform ellipticity of $D^2\hat{F}_\varepsilon$ (for fixed $\varepsilon$), the boundedness of $\partial_z^2 V$  and $\partial_{x_1}u_\varepsilon\in L^2(\Omega)$ imply a bound on $\norm{\nabla v_h^-}_{L^2(\Omega_h)}$, or $\norm{\nabla v_h^-}_{L^2(\Omega)}$ when extending $\nabla v_h^-$ by $0$ outside $\Omega_h$. Now by the regularity of $u_\varepsilon$ in $\Omega$ there holds $\nabla v_h^-(x)\rightarrow \nabla (\partial_{x_1}u_\varepsilon)^-(x)$ for all $x\in \Omega$ as $h\rightarrow 0$, which together with the $L^2$ bound for $\nabla v_h$ implies $\nabla (\partial_{x_1}u_\varepsilon)^-\in L^2(\Omega)$. Therefore $(\partial_{x_1}u_\varepsilon)^-\in H^1_0(\Omega)$.

Using now $\psi:=(\partial_{x_1}u_\varepsilon)^-$ as a test function for the differentiated equation
\[
\divv (C(x)\nabla\partial_{x_1}u_\varepsilon(x))+e(x)\partial_{x_1}u_\varepsilon(x)=0,\quad x\in \Omega,
\]
where this time $C(x):=D^2\hat{F}_\varepsilon(\nabla u_\varepsilon(x))$, $e(x)=\partial_z^2V(x,u_\varepsilon(x))$, we deduce
\[
0=\int_\Omega C(x)\nabla\psi\cdot\nabla\psi-e(x)\psi^2\:dx.
\]

On the other hand, since $u_\varepsilon$ is a minimizer of $\mathcal A_\varepsilon$ we also have for any $\phi\in H^1_0(\Omega)$, that
\begin{align*}
    0\leq \frac{d^2}{ds^2}|_{s=0}\mathcal{A}_\varepsilon (u_\varepsilon+s\phi)=\int_{\Omega} C(x) \nabla \phi \cdot\nabla \phi -e(x)\phi^2 \, dx.
\end{align*}
Hence if we assume $\psi\neq 0$, then the first eigenvalue of the self-adjoint operator $\mathcal L\phi:=-\divv( C(x) \nabla \phi)-e(x)\phi$ is $0$ and $\psi$ is in the associated eigenspace. However, the eigenspace associated with the first eigenvalue is one-dimensional and spanned by a function which is positive a.e. on $\Omega$. This contradicts the fact that $\psi$ is vanishing in a neighborhood of $(0,L/2)$, due to $\partial_{x_1}u_\varepsilon(0,L/2)>0$ and the continuity of $\partial_{x_1}u_\varepsilon$ at that point. In consequence $(\partial_{x_1}u_\varepsilon)^-=\psi=0$.
\end{proof}

\begin{lemma}\label{lem:energy_conservation_for_approximation}
Under the additional assumption \eqref{eq:V1prime} the quantity
\begin{equation}\label{eq:first_integral}
\int_{-L}^L\partial_{p_1}\hat{F}_\varepsilon(\nabla u_\varepsilon(x))\partial_{x_1}u_\varepsilon(x)-\hat{F}_\varepsilon(\nabla u_\varepsilon(x))+V(x_2,u_\varepsilon(x))\:dx_2
\end{equation}
is independent of $x_1\in (0,T)$ for any solution $u_\varepsilon$ of \eqref{eq:Dmin}.
\end{lemma}
\begin{proof} Let us denote the function in the integral \eqref{eq:first_integral} by $\hat{H}_\varepsilon(x)$.
Using the Euler Lagrange equation \eqref{eq:euler_lagrange_equ_epsilon} and \eqref{eq:V1prime} one easily checks that inside $\Omega$ there holds
\begin{align*}
    \partial_{x_1}&\hat{H}_\varepsilon+\partial_{x_2}\left(\partial_{p_2}\hat{F}_\varepsilon(\nabla u_\varepsilon)\partial_{x_1}u_\varepsilon\right)\\
    &=\divv\left(\nabla_p \hat{F}_\varepsilon(\nabla u_\varepsilon)\partial_{x_1}u_\varepsilon\right)+\partial_zV(\cdot,u_\varepsilon)\partial_{x_1}u_\varepsilon-\partial_{x_1}\left(\hat{F}_\varepsilon(\nabla u_\varepsilon)\right)=0.
\end{align*}
Integration and $\partial_{x_1}u_\varepsilon=0$ on $(0,T)\times\set{\pm L}$ imply the stated conservation. 
\end{proof}

\begin{lemma}\label{lem:good_set_property_epsilon}
Suppose \eqref{eq:V1prime} and $\partial^2_zV(x_2,z)\geq 0$ for $\abs{x_2}\leq L$, $\abs{z}\leq L+1$, then there holds the following one-sided maximum principle for $\partial_{x_1}u_\varepsilon$:
\[
\inf_{\Omega'}\partial_{x_1}u_\varepsilon=\inf_{\partial\Omega'}\partial_{x_1}u_\varepsilon
\] for all $\Omega'\subset\subset \Omega$.
\end{lemma}
\begin{proof}
Differentiation, the imposed convexity of $V$ w.r.t. $z$, Corollary \ref{cor:maximum_principle} and Lemma \ref{lem:partial_1_u_positive} show that $w_\varepsilon:=\partial_{x_1}u_\varepsilon$ indeed satisfies
\[
-\divv\left(\nabla^2\hat{F}_\varepsilon(\nabla u_\varepsilon)\nabla w_\varepsilon\right)=\partial_z^2V(x_2,u_\varepsilon)w_\varepsilon\geq 0.
\]
\end{proof}

\section{$\Gamma$-convergence and Young measure representation}\label{sec_gamma_convergence}
In this Section we will show that the found regular minimizers converge to a minimizer of the corresponding unperturbed variational problem. 

\begin{proposition}\label{prop:gamma_convergence_for_minimizers}
Let $u_\varepsilon$, $\varepsilon>0$ be a solution of \eqref{eq:Dmin} then there ex. a solution $u$ of the variational problem  \eqref{eq:var_prop_sec_3} resp., such that $u_\varepsilon\rightharpoonup u$ in $H^1(\Omega)$ along a subsequence.
\end{proposition}
The proof of Proposition \ref{prop:gamma_convergence_for_minimizers} is contained in the proof of Proposition \ref{prop:extended_gamma_convergence} below, which by means of the Young measure representation will also characterize where and how strong convergence can fail. 

\subsection{The recovery sequence}\label{sec_recovery_sequence}

Recall that $X$ consists of all $u\in H^1(\Omega)$ which satisfy $\norm{\partial_{x_2}u}_{L^\infty(\Omega)}\leq 1$ and agree in trace sense with 
\begin{equation}\label{eq:definition_of_U_0_hat}
    \hat{U}_0(x):=\left(\frac{2x_1}{T}-1\right)U_0(x_2)
\end{equation}
on $\partial\Omega$, where $U_0(x_2)=L-\abs{x_2}$.

\begin{lemma}\label{lem:recovery_lemma}
For every $u\in X$ with $\cA(u)<\infty$ there exists a sequence $u^\varepsilon\in X_\varepsilon$ with $u^\varepsilon\rightarrow u$ in $H^1(\Omega)$ and 
\[
\limsup_{\varepsilon\rightarrow 0}\cA_\varepsilon(u^\varepsilon)\leq \cA(u).
\]
\end{lemma}

% The proof is based on the following in between step.
% \begin{lemma}
% For any $u\in X$ with $\cA(u)<\infty$ there exists a sequence $(u^{\delta})_{\delta>0}$ contained in $X\cap W^{1,\infty}(\Omega)$, with $u^\delta \rightarrow u$ in $H^1(\Omega)$ and 
% \[
% \limsup_{\delta\rightarrow 0}\cA(u^\delta)\leq \cA(u).
% \]
% \end{lemma}
\begin{proof}
We first define a $H^1(\Omega)$ function $\tilde{u}^\delta$ by setting
\[
\tilde{u}^\delta(x):=\begin{cases}
-\cos(x_1)U_0(x_2),&x_1\in(0,\delta),\\
\cos(\delta)u\left(\frac{T}{T-2\delta}(x_1-\delta),x_2\right),&x_1\in(\delta,T-\delta)\\
\cos(T-x_1)U_0(x_2),&x_1\in (T-\delta,T).
\end{cases}
\]
It is easy to see that $\tilde{u}^\delta$ actually belongs to $X$ and that $\tilde{u}^\delta\rightarrow u$ in $H^1(\Omega)$ as $\delta \rightarrow 0$. 

% Moreover, direct computation and \eqref{eq:V1}, \eqref{eq:V2} yield 
% \[
% \int_{\set{x\in \Omega:x_1\notin (\delta,T-\delta)}}F(\nabla \tilde{u}^\delta)-V(x_2,\tilde{u}^\delta)\:dx=O(\delta)
% \]
% as $\delta\rightarrow 0$, while transformation in $x_1$ onto $(0,T)$ and dominated convergence, since $F(\nabla u)\in L^1(\Omega)$, show
% \[
% \int_\delta^{T-\delta}\int_{-L}^L F(\nabla \tilde{u}^\delta)-V(x_2,\tilde{u}^\delta)\:dx\rightarrow \cA(u).
% \]

Next we improve the integrability of $\partial_{x_1}\tilde{u}^\delta$ from $L^2$ to $L^\infty$ in order to be able to evaluate the extension $\hat{F}_\varepsilon$ on the set $K^\varepsilon$ where it agrees with the old $F_\varepsilon$.
% In order to do this we fix $\delta>0$ and carefully convolute $\tilde{u}^\delta$ in $x_1$-direction while keeping the boundary data and the bound  $\norm{\partial_{x_2}\tilde{u}^\delta}_{L^\infty(\Omega)}\leq 1$. More precisely, let $\varphi\in \cC^\infty_c(-1,1)$ be a standard one-dimensional symmetric mollifying kernel and for $\sigma \in (0,\delta/2)$ let $\eta_\sigma:[0,T]\rightarrow \R$ be the function
% \[
% \eta_\sigma(x_1):=\begin{cases}
% \sigma x_1,&x_1\in(0,\sigma),\\
% \sigma^2,&x_1\in (\sigma,T-\sigma),\\
% \sigma (T-x_1),&x_1\in (T-\sigma,T).
% \end{cases}
% \]
% Then we set 
% \begin{align*}
% \tilde{u}^{\sigma,\delta}(x)&:=\int_\R \varphi(s)\tilde{u}^\delta\big(x_1-\eta_\sigma(x_1)s,x_2\big)\:ds\\
% &=\frac{1}{\eta_\sigma(x_1)}\int_\R\varphi\left(\frac{x_1-s}{\eta_\sigma(x_1)}\right)\tilde{u}^\delta(s,x_2)\:ds
% \end{align*}
In order to do this fix $\delta>0$ and extend $\tilde{u}^\delta$ onto $(-\delta,T+\delta)\times (-L,L)$ by setting 
\begin{equation}\label{eq:extension_of_tilde_u_in_recovery}
\tilde{u}^\delta(x)=\begin{cases}
-(2-\cos(x_1))U_0(x_2),&x_1\in(-\delta,0),\\
(2-\cos(T-x_1))U_0(x_2),&x_1\in(T,T+\delta).
\end{cases}
\end{equation}
Due to this symmetric extension we can now convolute $\tilde{u}^\delta$ in $x_1$-direction and conserve the boundary data, i.e. for $x\in\Omega$ we set
\[
\tilde{u}^{\eta,\delta}(x):=\big(\varphi_\eta*_1\tilde{u}^\delta\big)(x)=\int_\R\varphi_\eta(x_1-s)\tilde{u}^\delta(s,x_2)\:ds
\]
with a symmetric one-dimensional kernel $\varphi_\eta$ of scale $\eta\in (0,\delta/2)$. 

Observe that $\tilde{u}^{\eta,\delta}$ indeed agrees with $\hat{U}_0$ on $\partial\Omega$ and that for a.e. $x\in\Omega$ there holds
\[
\abs{\partial_{x_1} \tilde{u}^{\eta,\delta}(x)}\leq c\eta^{-1}\norm{\tilde{u}^{\delta}}_{L^\infty},\quad
\abs{\partial_{x_2}\tilde{u}^{\eta,\delta}(x)}\leq 2-\cos(\delta)=1+\frac{\delta^2}{2}+o(\delta^2)
\]
with a constant $c>0$ depending only on the kernel $\varphi_1$. Since $\norm{\tilde{u}^\delta}_{L^\infty}$ is bounded (as a bounded extension of a function in $X$), we will pick $\eta(\varepsilon)=\varepsilon^{\theta}$  and $\delta(\varepsilon)=\varepsilon$ in order to satisfy $\nabla \tilde{u}^{\eta(\varepsilon),\delta(\varepsilon)}\in K^\varepsilon$ a.e. for sufficiently small $\varepsilon$.

It remains to adapt the boundary data. Therefore we finally set 
\[
u^\varepsilon(x):=\tilde{u}^{\varepsilon^{\theta},\varepsilon}(x)+\hat{U}_\varepsilon(x)-\hat{U}_0(x),
\]
where in analogy with \eqref{eq:definition_of_U_0_hat} the function $\hat{U}_\varepsilon:\Omega\rightarrow\R$ is defined as
\begin{equation}\label{eq:sec4_definition_of_U_hat_epsilon}
\hat{U}_\varepsilon(x)=\left(\frac{2x_1}{T}-1\right)U_\varepsilon(x_2)=\hat{U}_0(x)+\left(\frac{2x_1}{T}-1\right)\frac{\varepsilon^\beta}{2L}(L^2-x_2^2).
\end{equation}

By our construction it is clear that $u^\varepsilon\in X_\varepsilon$ and that still $\nabla u^\varepsilon(x)\in K^\varepsilon$ for a.e. $x\in\Omega$ and $\varepsilon>0$ small enough. Hence $\hat{F}_\varepsilon\circ\nabla u^\varepsilon=F_\varepsilon\circ \nabla u^\varepsilon$ almost everywhere.

Next we will show that $u^\varepsilon\rightarrow u$ in $H^1(\Omega)$. Clearly $u^\varepsilon-\tilde{u}^{\varepsilon^{\theta},\varepsilon}\rightarrow 0$ in $H^1(\Omega)$ as $\varepsilon\rightarrow 0$. In order to see that $\tilde{u}^{\varepsilon^{\theta},\varepsilon}\rightarrow u$ let $\tilde{u}\in H^1((-1,T+1)\times(-L,L))$ be the extension of $u$ with the values given in \eqref{eq:extension_of_tilde_u_in_recovery} for $\tilde{u}^\delta$. Then
\begin{align*}
    \norm{\tilde{u}^{\varepsilon^{\theta},\varepsilon}-u}_{H^1(\Omega)}&\leq \norm{\varphi_{\varepsilon^{\theta}}*_1(\tilde{u}^\varepsilon-\tilde{u})}_{H^1(\Omega)} +\norm{\varphi_{\varepsilon^{\theta}}*_1\tilde{u}-\tilde{u}}_{H^1(\Omega)}\\
    &\leq \norm{\tilde{u}^\varepsilon-\tilde{u}}_{H^1(\Omega)}+\norm{\varphi_{\varepsilon^{\theta}}*_1\tilde{u}-\tilde{u}}_{H^1(\Omega)}\rightarrow 0.
\end{align*}

Finally it remains to look at the values of the action functionals. The just shown convergence $u^\varepsilon\rightarrow u$ in $H^1(\Omega)$ and condition \eqref{eq:condition_V1} easily imply
\[
\int_\Omega V(x,\tilde{u}^\varepsilon(x))\:dx\rightarrow \int_\Omega V(x,u(x))\:dx.
\]
Concerning $\hat{F}_\varepsilon$ we first of all observe that
\begin{align*}
    \limsup_{\varepsilon\rightarrow 0}\int_\Omega \hat{F}_\varepsilon(\nabla u^\varepsilon)\:dx&=\limsup_{\varepsilon\rightarrow 0}\int_\Omega F_\varepsilon(\nabla u^\varepsilon)\:dx\\
    &\leq \limsup_{\varepsilon\rightarrow 0}\int_\Omega \frac{\big(\partial_{x_1}\tilde{u}^{\varepsilon^{\theta},\varepsilon}\big)^2}{2\left(1+\varepsilon-\big(\partial_{x_2}\tilde{u}^{\varepsilon^{\theta},\varepsilon}\big)^2\right)}\:dx,
\end{align*}
where we have used that $u^\varepsilon-\tilde{u}^{\varepsilon^{\theta},\varepsilon}=O(\varepsilon^\beta)$ in $W^{1,\infty}(\Omega)$ as $\varepsilon\rightarrow 0$ and $\theta,\beta>1$. The function $\tilde{F}_\varepsilon:\set{p\in\R^2:\abs{p_2}< 1+\varepsilon}\rightarrow \R$,
\[
\tilde{F}_\varepsilon(p)=\frac{p_1^2}{2(1+\varepsilon-p_2^2)}
\]
is also convex. Hence Jensen's inequality implies 
\begin{align*}
    \int_\Omega \tilde{F}_\varepsilon(\nabla \tilde{u}^{\varepsilon^{\theta},\varepsilon})\:dx&=\int_\Omega \tilde{F}_\varepsilon (\varphi_{\varepsilon^{\theta}}*_1\nabla \tilde{u}^\varepsilon)\:dx\leq \int_\Omega \varphi_{\varepsilon^{\theta}}*_1\left(\tilde{F}_\varepsilon\circ \nabla \tilde{u}^\varepsilon\right)\:dx\\
    &=\int_D\varphi_1(s)\tilde{F}_\varepsilon(\nabla\tilde{u}^\varepsilon(x_1+s\varepsilon^\theta,x_2))\:d(s,x),
\end{align*}
where $D:=(-1,1)\times\Omega$. We now split the domain of integration into the following sets
\begin{gather*}
D_1=\set{(s,x)\in D: x_1+s\varepsilon^\theta <0},\quad D_2=\set{(s,x)\in D: 0<x_1+s\varepsilon^\theta <\varepsilon},\\
D_3=\set{(s,x)\in D: \varepsilon<x_1+s\varepsilon^\theta <T-\varepsilon},\\
D_4=\set{(s,x)\in D: T-\varepsilon<x_1+s\varepsilon^\theta <T},\quad D_5=\set{(s,x)\in D: x_1+s\varepsilon^\theta >T}
\end{gather*}
and estimate the corresponding integrals $I_1,\ldots,I_5$. There holds
\begin{align*}
    I_1&=\int_{D_1}\varphi_1(s)\frac{\sin^2(x_1+s\varepsilon^\theta)U_0(x_2)^2}{2(1+\varepsilon-(2-\cos(x_1+s\varepsilon^\theta))^2)}\:d(s,x)\\
    &\leq L^3\norm{\varphi_1}_{L^\infty}\int_{-1}^1\int_0^{\varepsilon^\theta}\frac{1}{1+\frac{\varepsilon-4+4\cos(x_1+s\varepsilon^\theta)}{\sin^2(x_1+s\varepsilon^\theta)}}\:dx_1\:ds\leq 2\varepsilon^\theta L^3\norm{\varphi_1}_{L^\infty}\frac{1}{1+\frac{\varepsilon}{2}}.
\end{align*}
Hence $I_1\rightarrow 0$ as $\varepsilon\rightarrow 0$. A similar reasoning also yields $I_2+I_4+I_5\rightarrow 0$ as $\varepsilon\rightarrow 0$.

For $I_3$ we use the transformation $x_1\mapsto y_1(s,x_1)=\frac{T}{T-2\varepsilon}(x_1+s\varepsilon^\theta-\varepsilon)$ in order to see that
\begin{align*}
    I_3&=\int_{-1}^1\varphi_1(s)\int_{\varepsilon-s\varepsilon^\theta}^{T-\varepsilon-s\varepsilon^\theta}\int_{-L}^L\frac{\left(\frac{T}{T-2\varepsilon}\right)^2\big(\partial_{x_1}u(y_1(s,x_1),x_2)\big)^2}{2\left(\frac{1+\varepsilon}{\cos^2(\varepsilon)}-\big(\partial_{x_2}u(y_1(s,x_1),x_2)\big)^2\right)}\:dx_2\:dx_1\:ds\\
    &\leq \frac{T}{T-2\varepsilon}\int_0^T\int_{-L}^L\tilde{F}_\varepsilon (\nabla u(y_1,x_2))\:dx_2\:dy_1\rightarrow \int_\Omega F(\nabla u)\:dx
\end{align*}
by means of monotone convergence, since $F(\nabla u)\in L^1(\Omega)$ by assumption.

Altogether we therefore have shown 
\[
\limsup_{\varepsilon\rightarrow 0}\int_\Omega \hat{F}_\varepsilon(\nabla u^\varepsilon)-V(x,u^\varepsilon)\:dx\leq \int_\Omega F(\nabla u)-V(x,u)\:dx.
\]
\end{proof}

\subsection{Young measures}\label{sec_young_measures}
Before stating a general convergence result, which contains the weak lower semi-continuity for the $\Gamma$-convergence, we quickly recall the notion of generalized Young measures, see e.g. \cite{kristensen_raita_lecture_notes}.

Let $S$ denote the unit sphere of $\mathbb V=\R^m$ and $\Omega$ be a bounded domain in $\R^l$ with $\partial\Omega$ having $0$ Lebesgue measure. Every weakly converging sequence $U_k\rightharpoonup U$ in $L^q(\Omega,\mathbb V)$ induces a $q$-Young measure $\boldsymbol{\nu}=\left((\nu_x)_{x\in\Omega},\lambda,(\nu_x^\infty)_{x\in\overline\Omega}\right)$. I.e., $(\nu_x)_{x\in\Omega}$ is a Lebesgue measurable family of probability measures on $\mathbb V$ (oscillation measure), $\lambda$ is a positive measure on $\overline{\Omega}$ (concentration measure) and $(\nu_x^\infty)_{x\in\overline{\Omega}}$ is a $\lambda$-measurable family of probability measures on $S$ (concentration angle measure), such that for all $q$-admissible integrands $\Phi$ there holds
\begin{align*}
    \int_\Omega\Phi(x,U_k(x))\:dx\rightarrow \int_\Omega\int_\mathbb V\Phi(x,v)\:d\nu_x(v)\:dx+\int_{\overline\Omega}\int_S\Phi^{q,\infty}(x,v)\:d\nu_x^\infty(v)\:d\lambda(x).
\end{align*}
A continuous function $\Phi:\overline{\Omega}\times \mathbb V\rightarrow \R $ is $q$-admissible provided the $q$-recession function
\[
\Phi^{q,\infty}(x,v):=\lim_{t\rightarrow\infty}\frac{\Phi(x,tv)}{t^q}
\]
exists, is finite and the convergence is locally uniform w.r.t. $(x,v)\in\overline{\Omega}\times (\mathbb V\setminus \{0\})$. 

In terms of the Young measure, the weak limit $U$ of $U_k$ is represented as the barycenter of $\nu_x$, i.e.
\begin{align*}
    U(x)=\int_{\mathbb V}v\:d\nu_x(v) \text{ for a.e. }x\in \Omega.
\end{align*}
Strong convergence $U_k\rightarrow U$ in $L^q(\Omega,\mathbb V)$ can equivalently be characterized by the absence of concentration and oscillation, i.e. $\lambda=0$ and $\nu_x=\delta_{U(x)}$ for a.e. $x\in \Omega$.

\subsection{A general convergence result}\label{sec_general_convergence_result}
While the existence of a recovery sequence has been a construction specific to our particular problem, we now investigate a more general class of variational problems and their approximations. The more general setting in this subsection is the following.

Let $0\in C\subset \R^m$ be open, convex and $f:\overline{C}\rightarrow [0,\infty]$ be a convex lower semi-continuous function. We suppose further that there exist convex functions $f_k:\R^m\rightarrow [0,\infty)$ with 
\begin{equation}\label{eq:sec4_condition_on_fk1}
    f_k(0)=0,\quad f_k(p)\geq c\abs{p}^q-d\text{ for all }p\in \R^m,
\end{equation}
for some constants $c,d>0$ and $q\in(1,\infty)$ independent of $k$, and 
\begin{gather}\label{eq:sec4_condition_on_fk2}
    f_k\rightarrow f\text{ uniformly on any compact subset of }C\text{ as }k\rightarrow\infty,\\\label{eq:sec4_condition_on_fk3}
    f_k(p)\rightarrow +\infty\text{ uniformly on any set with positive distance to } \overline{C}. 
\end{gather}
Observe that these conditions imply $f(0)=0$ and that $f$ is continuous and finite on $C$. 

We then consider the functionals 
\begin{align*}
    \cF_k(U):=\int_\Omega f_k(U(x))\:dx,\quad \cF(U):=\int_\Omega f(U(x))\:dx,
\end{align*}
where $U\in L^q(\Omega;\R^m)$, $\Omega\subset \R^l$ open, bounded with $\partial \Omega$ being a Lebesgue null set.

\begin{lemma}\label{lem:genconv}
Suppose $U_k\rightharpoonup U$ in $L^q(\Omega;\R^m)$ and let  $\boldsymbol{\nu}=\left((\nu_x)_{x\in\Omega},\lambda,(\nu_x^\infty)_{x\in\overline\Omega}\right)$ be the associated $q$-Young measure. Assume that $\sup_k\cF_k(U_k)<\infty$. Then
\begin{itemize}
    %\item[(i)] there holds $U(x)\in \overline{C}$ for a.e. $x\in\Omega$,
    \item[(i)] there holds $\supp\nu_x\subset \overline{C}$ for a.e. $x\in\Omega$,
    \item[(ii)] there holds $\cF(U)\leq \displaystyle\liminf_{k\to+\infty} \cF_k(U_k)$,
    % \item[(iii)] and if there holds equality in (ii), then $\lambda=0$, and for a.e. $x\in\Omega$ with $U(x)\in C$, there holds
    % $$\text{supp } \nu_x \cap \{U(x)+\mathbb R w\}=\{U(x)\},$$
    % whenever $w\in\R^m$, $2f(U(x))<f(U(x)+s w)+f(U(x)-s w)$ for all $s\neq 0$, $\abs{s}$ sufficiently small.
    \item[(iii)] and if there holds equality in (ii), then $\lambda=0$, and for a.e. $x\in\Omega$ there holds
    $$\text{supp } \nu_x \cap \{U(x)+\mathbb R w\}=\{U(x)\},$$
    whenever $w\in\R^m$ satisfies
    \begin{equation}\label{eq:sec4_strict_convex_condition}
        f(U(x)+sw)-f(U(x))>sz_0\cdot w
    \end{equation}
    for a subgradient $z_0\in\R^m$ of $f$ at $U(x)$ and any $s\neq 0$ with $U(x)+sw\in\overline{C}$.
\end{itemize}
\end{lemma}
Part (i) implies $U(x)\in \overline{C}$ for a.e. $x\in\Omega$. Part (ii) states the weak lower semi-continuity, except that the functionals are changing along the sequence. Part (iii) shows that, if the values of the functional actually converge, then there is no concentration ($\lambda=0$). Furthermore, strict convexity of the integrand $f$ in direction $w$ excludes oscillations in that direction. In consequence if $f$ is strictly convex (in all directions) on $C$, $U(x)\in C$ for a.e. $x\in\Omega$, then $U_k\rightarrow U$ strong in $L^q(\Omega;\R^m)$. Note here that a subgradient $z_0$ of $f$ at $U(x)$, i.e. a vector $z_0$ for which \eqref{eq:sec4_strict_convex_condition} holds with $sw$ replaced by any $\tilde{w}\in \overline{C}-U(x)$ and $>$ replaced by $\geq$, always exists when $U(x)\in C$. In Part (iii) the existence of a corresponding subgradient $z_0$ is part of the condition. 

\begin{proof}[Proof of Lemma \ref{lem:genconv}]
\emph{Part (i).} For $\delta>0$ define $K_\delta:=\set{p\in\R^m:\dist(p,\overline{C})\geq \delta}$.
The uniform boundedness of $\cF_k(U_k)$ and condition \eqref{eq:sec4_condition_on_fk3} imply that
\[
\abs{\set{x\in\Omega:U_k(x)\in K_\delta}}\rightarrow 0
\]
as $k\rightarrow \infty$, where $\abs{\cdot}$ denotes the $l$-dimensional Lebesgue measure. Now for any $\Phi\in\cC^0(\R^m)$ with $\Phi=0$ on $\R^m\setminus K_\delta$ and $0\leq \Phi\leq 1$ on $K_\delta$ there holds
\begin{align*}
    0=\lim_{k\rightarrow\infty}\abs{\set{x\in\Omega:U_k(x)\in K_\delta}}\geq \lim_{k\rightarrow\infty}\int_\Omega \Phi(U_k(x))\:dx=\int_\Omega\int_{\R^m}\Phi(v)\:d\nu_x(v)\:dx
\end{align*}
by the definition of the Young measure. Hence $\supp \nu_x\subset \R^m
\setminus K_\delta$ for a.e. $x\in \Omega$, for all $\delta>0$.

% Given $\varepsilon>0$ we can in consequence find a subsequence $(U_{k_j})_j\subset (U_k)_k$ satisfying 
% \[
% \abs{\set{x\in \Omega:U_{k_j}(x)\in K_\delta}}\leq 2^{-j}\varepsilon.
% \]
% Since $U_{k_j}\rightharpoonup U$ in $L^q(\Omega;\R^m)$ there ex. a sequence of convex combinations converging strong to $U$, i.e. there ex. $\lambda_{ij}\in[0,1]$, $j=1,\ldots,N_i$ with
% \[
% \sum_{j=1}^{N_i}\lambda_{ij}=1,\quad W_i:=\sum_{j=1}^{N_i}\lambda_{ij}U_{k_j}\rightarrow U\text{ in }L^q(\Omega;\R^m)\text{ as }i\rightarrow\infty.
% \]
% Since the complement of $K_\delta$ is convex, we conclude
% \begin{align*}
%     \abs{\set{x\in \Omega:U(x)\in K_\delta}}&=\lim_{i\rightarrow\infty}\abs{\set{x\in\Omega:W_i(x)\in K_\delta}}\\
%     &\leq \sum_{j\in \N}\abs{\set{x\in\Omega:U_{k_j}(x)\in K_\delta}}\leq 2\varepsilon.
% \end{align*}
% Hence the set $\set{x\in\Omega:U(x)\notin \overline{C}}$ has measure $0$.

\emph{Part (ii).} Let $C_j\subset \R^m$, $j\in \N$ be compact convex sets with 
\begin{align*}
0\in C_j\subset C_{j+1}\subset\subset C,\quad \bigcup_{j\in\N}C_j=C
\end{align*}
and define $T_j:\R^m\rightarrow C_j$ as the radial retraction onto $C_j$, that is $T_j(p)=r_j(p)p$, where $r_j(p):=\inf\set{r>0:r^{-1}p\in C_j}$. Note that $r_j$ is convex. Thus $T_j$ is continuous.

The convexity of $f_k$ and $f_k(0)=0$, cf. \eqref{eq:sec4_condition_on_fk1}, imply that
\begin{align}\label{eq:sec4_estimates_for_fk_circ_Tj}
    f_k(T_j(p))&=f_k\big(r_j(p) p+(1-r_j(p))0\big)\leq r_j(p)f_k(p)+0\leq f_k(p) 
\end{align}
for every $k,j\in\N$, $p\in\R^m$.

In consequence also the functions $g_{k,j},g_j:\R^m\rightarrow [0,\infty)$,
\[
g_{k,j}(p):=\max\set{f_k(T_j(p)),\frac{c}{2}\abs{p}^q-d},\quad g_{j}(p):=\max\set{f(T_j(p)),\frac{c}{2}\abs{p}^q-d} 
\]
with the constants $c,d,q$ taken from \eqref{eq:sec4_condition_on_fk1}, are continuous and there holds 
\begin{align}\label{eq:sec4_estimate_g_kj}
    g_{k,j}(p)\leq f_k(p)\text{ for any }k,j\in\N,~p\in \R^m.
\end{align}
By condition \eqref{eq:sec4_condition_on_fk2} we also see that $g_{k,j}\rightarrow g_j$ uniformly on all of $\R^m$ as $k\rightarrow\infty$.

Furthermore, the functions $g_{j}$ are $q$-admissible (autonomous) integrands. Indeed $f|_{C_j}$ is bounded by the compactness of $C_j$ and therefore the $q$-regression function reads
\[
g^{q,\infty}_{j}(p)=\lim_{t\rightarrow\infty}t^{-q}g_{j}(tp)=\frac{c}{2}\abs{p}^q,
\]
while the convergence is locally uniform w.r.t. $p\in\R^m\setminus\{0\}$.

By \eqref{eq:sec4_estimate_g_kj}, the uniform convergence $g_{k,j}\rightarrow g_j$ and the definition of the Young measure associated with $(U_k)_{k\in\N}$ we therefore obtain for any $j\in\N$
\begin{align*}
    \liminf_{k\rightarrow\infty}\cF_k(U_k)&\geq \liminf_{k\rightarrow\infty}\int_\Omega g_{k,j}(U_k(x))\:dx\\
    &\geq \liminf_{k\rightarrow\infty}\int_\Omega g_{j}(U_k(x))\:dx+ \liminf_{k\rightarrow\infty}\int_\Omega g_{k,j}(U_k(x))-g_j(U_k(x))\:dx\\
    &=\int_\Omega \int_{\overline{C}}g_{j}(v)\:d\nu_x(v)\:dx+\frac{c}{2}\lambda(\overline{\Omega})+0.
\end{align*}
In the last step we have also used (i).

Now one can check similarly to \eqref{eq:sec4_estimates_for_fk_circ_Tj} that the sequence $g_j$ is monotone increasing. Thus by monotone convergence, assumption \eqref{eq:sec4_condition_on_fk1} and the lower semicontinuity of $f$ we conclude
\begin{align*}
    \liminf_{k\rightarrow\infty}\cF_k(U_k)&\geq \int_\Omega \int_{\overline C }\lim_{j\rightarrow\infty}g_j(v)\:d\nu_x(v)\:dx+\frac{c}{2}\lambda(\overline{C})\\
    &\geq \int_\Omega\int_{\overline C}f(v)\:d\nu_x(v)\:dx+\frac{c}{2}\lambda(\overline C).
\end{align*}
Therefore Jensen's inequality finally shows that 
\begin{align}\label{eq:sec4_final_(i)}
    \liminf_{k\rightarrow\infty}\cF_k(U_k)&\geq\int_\Omega f(U(x))\:dx+\frac{c}{2}\lambda(\overline{\Omega})\geq \cF(U). 
\end{align}
This finishes the proof of part (ii).

\emph{Part (iii).} We immediately see that equality in \eqref{eq:sec4_final_(i)} implies $\lambda=0$. Going one step back there also has to hold 
\begin{align}\label{eq:sec4_barycenter_U_integral}
\int_{\overline C} f(v)\:d\nu_x(v)=f(U(x))
\end{align}
for a.e. $x\in \Omega$. We fix now such a point $x$ and suppose that $f(U(x))<\infty$, $w\in \R^m\setminus\{0\}$, $z_0\in\partial f(U(x))$ (the subdifferential of $f$ at $U(x)$) satisfy \eqref{eq:sec4_strict_convex_condition} for any $s\neq 0$ with $U(x)+sw\in \overline{C}$.

Assume to the contrary of the statement that $\supp \nu_x\cap\set{U(x)+\R w}\neq \set{U(x)}$, which means that there exists $s_0\neq 0$ with 
\begin{equation}\label{eq:sec4_measure_of_balls}
\nu_x(B_r(U(x)+s_0w))>0
\end{equation} for all $r>0$. By the properties of $f$ and \eqref{eq:sec4_strict_convex_condition} with $s=s_0$ we can pick a radius $r_0>0$, such that
\begin{equation}\label{eq:sec4_strict_convexity_2}
f(v)-f(U(x))>z_0\cdot(v-U(x))
\end{equation}
for all $v\in B:=\overline{C}\cap B_{r_0}(U(x)+s_0 w)$. Combining \eqref{eq:sec4_barycenter_U_integral}, \eqref{eq:sec4_measure_of_balls}, \eqref{eq:sec4_strict_convexity_2} and the fact that $z_0$ is a subgradient yields the contradiction
\begin{align*}
    f(U(x))&=\int_{\overline{C}\cap B}f(v)\:d\nu_x(v)+\int_{\overline{C}\setminus B}f(v)\:d\nu_x(v)\\
    &>\int_{\overline{C}\cap B}f(U(x))+z_0\cdot(v-U(x))\:d\nu_x(v)+\int_{\overline{C}\setminus B}f(v)\:d\nu_x(v)\\
    &\geq \int_{\overline{C}}f(U(x))+z_0\cdot(v-U(x))\:d\nu_x(v)=f(U(x)).
\end{align*}
Hence $\supp \nu_x\cap\set{U(x)+\R w}=\set{U(x)}$.
\end{proof}

\subsection{Our case}\label{sec_convergence_result_applied_to_our_case}

Applying Lemmas \ref{lem:recovery_lemma}, \ref{lem:genconv} to our case 
% The following analysis does not depend on the boundary conditions, hence we will not treat the Dirichlet and Neumann cases separately, rather just consider a sequence of minimizers
% $(u_\varepsilon)_{\varepsilon>0}$ of $\mathcal A_\varepsilon$, and their weak limit $u$ which is a minimizer of $\mathcal A$.
we obtain the following extended version of Proposition \ref{prop:gamma_convergence_for_minimizers}.

\begin{proposition}\label{prop:extended_gamma_convergence}
Let $u_\varepsilon$, $\varepsilon>0$ be a solution of \eqref{eq:Dmin}. Then there ex. a solution $u$ of the variational problem  \eqref{eq:var_prop_sec_3}, such that $u_\varepsilon\rightharpoonup u$ in $H^1(\Omega)$, $\cA_\varepsilon(u_\varepsilon)\rightarrow \cA(u)$ and
\begin{equation}\label{eq:convergence_of_kinetic_energy_alone}
    \int_\Omega \hat{F}_\varepsilon(\nabla u_\varepsilon(x))\:dx\rightarrow\int_\Omega F(\nabla u(x))\:dx
\end{equation}
along a subsequence. Moreover, in terms of the Young measure representation of the weak limit of $\nabla u_\varepsilon$ there holds $\boldsymbol{\nu}=((\nu_x)_{x\in\Omega},0,0)$ with 
\[
\nu_x(v)=\begin{cases}\delta_{\nabla u(x)}(v),&\text{for a.e. }x\in \Omega\text{ with }\partial_{x_1}u(x)\neq 0,\\
\delta_{\partial_{x_1}u(x)}(v_1)\otimes \hat\nu_x(v_2)&\text{for a.e. }x\in \Omega\text{ with }\partial_{x_1}u(x)=0,
\end{cases}
\]
where $v\in\R^2$ and $\hat{\nu}_x$ is a probability measure on $\R$ with support in $[-1,1]$. 
\end{proposition}
\begin{proof}
Let $u_\varepsilon\in X_\varepsilon$ be a minimizer of $\cA_\varepsilon(u)$. By Lemma \ref{lem:uniform_energy_bound} the minimal values $\cA_\varepsilon(u_\varepsilon)$, $\varepsilon\in(0,1)$ are bounded. Using the uniform convexity of the extension $\hat{F}_\varepsilon$ outside the compact set $\set{p\in\R^2:\abs{p_1}\leq 1,~\abs{p_2}\leq 1+\varepsilon-\varepsilon
^{4\theta}}$, cf. Lemma \ref{lem:first_extension_of_F_epsilon}, we conclude
\begin{equation}\label{eq:coercivity_of_F_varepsilon}
    \hat{F}_\varepsilon(p)\geq c_1\abs{p}^2-c_2
\end{equation}
for all $p\in\R^2$ and some constants $c_1,c_2>0$ independent of $\varepsilon\in (0,1)$. Condition \eqref{eq:condition_V1} therefore yields
\begin{align*}
    \norm{\nabla u_\varepsilon}_{L^2(\Omega)}^2\leq c_3\cA_\varepsilon(u_\varepsilon)+c_4
\end{align*}
with $\varepsilon$-independent constants $c_3,c_4>0$. In consequence there exists a subsequence $(u_{\varepsilon_k})_k$,  as well as $u\in H^1(\Omega)$ with $u_{\varepsilon_k}\rightharpoonup u$ in $H^1(\Omega)$. Since $u_\varepsilon-\hat{U}_\varepsilon\in H^1_0(\Omega)$ with $\hat{U}_\varepsilon$ defined in \eqref{eq:sec4_definition_of_U_hat_epsilon} we conclude that $u$ coincides with $\hat{U}_0$ on $\partial \Omega$ in the trace sense. In order to see that $u\in X$ it therefore remains to show $\norm{\partial_{x_2}u}_{L^\infty(\Omega)}\leq 1$. For this we will rely on Lemma \ref{lem:genconv}.

It is easily seen, c.f. \eqref{eq:coercivity_of_F_varepsilon}, that the set $C:=\set{p\in\R^2:\abs{p_2}<1}$, the approximating functions $f_k:\R^2\rightarrow \R$, $f_k(p):=\hat{F}_{1/k}(p)-\hat{F}_{1/k}(0)$, $p\in\R^2$ and  the limiting function $f:\overline{C}\rightarrow\R$, $f(p):=F(p)$ satisfy the conditions postulated in Section \ref{sec_general_convergence_result} with exponent $q=2$. We therefore can apply Lemma \ref{lem:genconv} with \begin{align*}
    \cF_k(U)=\int_\Omega f_k(U(x))\:dx,\quad \cF(U)=\int_\Omega f(U(x))\:dx,
\end{align*}
and $U_k:=\nabla u_{\varepsilon_k}$. Observe that $\cF_k(U_k)$ is indeed uniformly bounded due to the boundedness of $\cA_\varepsilon(u_\varepsilon)$ and condition \eqref{eq:condition_V1}. In consequence Lemma \ref{lem:genconv} tells us that if $\boldsymbol{\nu}=((\nu_x)_{x\in\Omega},\lambda,(\nu^\infty_x)_{x\in\overline{\Omega}})$ denotes the Young measure representation for $(U_k)_k$, then $\supp \nu_x\subset \set{p\in\R^2:\abs{p_2}\leq 1}$ and 
\[
\liminf_{k\rightarrow \infty}\int_\Omega \hat{F}_{\varepsilon_k}(\nabla u_{\varepsilon_k})-\hat{F}_{\varepsilon_k}(0)\:dx\geq \int_\Omega F(\nabla u)\:dx. 
\]

It follows that $\abs{\partial_{x_2}u}\leq 1$ a.e. and 
\[
\liminf_{k\rightarrow\infty}\cA_{\varepsilon_k}(u_{\varepsilon_k})\geq \cA (u),
\]
since w.l.o.g. $V(\cdot,u_{\varepsilon_k})\rightarrow V(\cdot,u)$ in $L^1(\Omega)$ by \eqref{eq:condition_V1}.

In particular $u\in X$ and $\cA(u)<\infty$. Hence the recovery sequence Lemma \ref{lem:recovery_lemma} on one hand shows that $\cA(u)\leq \cA(\tilde{u})$ for all $\tilde{u}\in X$, i.e. $u$ is a minimizer of \eqref{eq:var_prop_sec_3}, on the other hand it shows that 
\[
\lim_{k\rightarrow\infty}\cA_{\varepsilon_k}(u_{\varepsilon_k})=\cA (u),
\]
which enables us to utilize Part (iii) of Lemma \ref{lem:genconv}. Doing this we first of all see that $(\nabla u_{\varepsilon_k})_k$ does not concentrate, i.e. $\lambda =0$, $\nu_x^\infty=0$, $x\in\overline{\Omega}$. Next, since
\[
F(p+sw)-F(p)-s\nabla F(p)\cdot w>0
\]
for any triple $(p,w,s)\in \set{p\in\R^2:\abs{p_2}<1}\times(\R^2\setminus \{0\})\times (\R\setminus \{0\})$ with $p_1\neq 0$ or $w_1\neq 0$, cf. \eqref{eq:Fepsder} for $\varepsilon=0$, there also holds 
$\supp \nu_x=\set{\nabla u(x)}$ for a.e. $x\in \Omega$ with $\partial_{x_1}u(x)\neq 0$, $\abs{\partial_{x_2}u(x)}<1$ and $\supp \nu_x\subset \set{(0,v_2):\abs{v_2}\leq 1}$ for a.e. $x\in\Omega $ with $\partial_{x_1} u(x)=0$.  Note here that $0$ is a subgradient of $F$, or rather $f:\overline{C}\rightarrow\R$, at $(0,\pm 1)$. Moreover, due to the fact that the set $\set{x\in \Omega:\partial_{x_1} u(x)\neq 0,\ |\partial_{x_2} u(x)|=1}$ must be of zero measure, else $\mathcal A (u)$ would be infinite, we can also in the first case simply say $\supp \nu_x=\set{\nabla u(x)}$ for a.e. $x\in \Omega$ with $\partial_{x_1}u(x)\neq 0$. This finishes the proof of Proposition \ref{prop:extended_gamma_convergence}. 
\end{proof}
Proposition \ref{prop:extended_gamma_convergence}, Corollary \ref{cor:maximum_principle} and Lemma \ref{lem:partial_1_u_positive} directly imply the following bounds.
\begin{corollary}\label{cor:bounds_for_minimizer}
The minimizer $u$ satisfies $\abs{u(x)}\leq L-\abs{x_2}$ for a.e. $x\in \Omega$. If in addition \eqref{eq:V1prime} holds true, then $\partial_{x_1}u\geq 0$ a.e..
\end{corollary}

\begin{remark}\label{rem:obstacle_problem}
In view of the bound on $\abs{u}$ problem \eqref{eq:var_prop_sec_3} can, as the degenerate variational problem arising in the study of random surfaces in \cite{desilva_savin}, be written as an obstacle problem:
\begin{align*}
    \min_{\abs{x_2}-L\leq u\leq L-\abs{x_2}} \cA(u) \quad\quad \text{ ($+$ boundary conditions)}
\end{align*}

\end{remark}

\subsection{Autonomous potentials}\label{sec_min_cont}
In the case of $\partial_{x_1}u_\varepsilon\geq 0$ we can extend the list of convergences as $\varepsilon\rightarrow 0$ as follows.
\begin{proposition}\label{prop:u_continuous} Let $u_\varepsilon$,$u$ be as in Proposition \ref{prop:extended_gamma_convergence} and suppose that in addition there holds \eqref{eq:V1prime}. Then $u_\varepsilon\rightarrow u$ uniformly on $\overline{\Omega}$. In particular the minimizer $u$ is continuous.
\end{proposition}

The proof is a direct consequence of Lemma \ref{lem:partial_1_u_positive}, as said, and Lemma \ref{lem:mincont} below.

\begin{lemma}\label{lem:mincont}
For any $u\in\cC^0(\overline{\Omega})\cap \cC^1(\Omega)$ with $\partial_{x_1}u\geq 0$ there holds
\begin{align*}
    \osc(u;B_r(z_0)) \lesssim
    \frac{\|\nabla u\|_{L^2(\Omega)}}{\sqrt{|\log{r}|}},
\end{align*}
for any $r\in(0,1)$ and $z_0\in\R^2$ such that $B_{\sqrt{r}}(z_0)\subset\Omega$, as well as
\begin{align*}
    \osc(u;B_r(z_0)\cap\overline{\Omega}) \lesssim
    \frac{\|\nabla u\|_{L^2(\Omega)}}{\sqrt{|\log{r}|}}+\osc(u;B_{\sqrt{r}}(z_0)\cap\partial\Omega),
\end{align*}
for $r\in(0,1)$ and arbitrary $z_0\in\R^2$.
\end{lemma}
The proof relies on the classical Courant-Lebesgue lemma, \cite{courant}.
\begin{lemma}[Courant-Lebesgue]\label{lem:CL}
Let $u\in \cC^1(\Omega)$. For $z\in\R^2$ denote the length of $u(\partial B_r(z)\cap\Omega)$ by $L(r):=\int_{\partial B_r(z)\cap\Omega}\abs{\partial_\tau u}$, where $\partial_\tau u$ is the tangential derivative. Then there holds
$$\int_0^\infty \frac{L(r)^2}{r}\, dr \leq 2\pi\|\nabla u\|_{L^2(\Omega)}^2,$$
and consequently
$$\min_{a<r<b}L(r)^2 \leq \frac{2\pi\|\nabla u\|_{L^2(\Omega)}^2}{\log(b/a)} .$$
\end{lemma}
\begin{proof}
Using the parametrization $\theta\mapsto u(z+re^{i\theta})$ of $\partial B_r(z)$, one has
\begin{align*}
    L(r)&=
    r \int_{\set{\theta: \ z+re^{i\theta}\in\Omega }}|\partial_\tau u(z+re^{i\theta})|\, d\theta \\
    &\leq r \left(2\pi \int_{\set{\theta: \ z+re^{i\theta}\in\Omega }}|\nabla u(z+re^{i\theta})|^2\, d\theta \right)^{1/2},
\end{align*}
and hence
\begin{align*}
    \int_0^\infty \frac{L(r)^2}{r}\, dr \leq 2\pi \int_0^\infty \int_{\set{\theta: \ z+re^{i\theta}\in\Omega }}|\nabla u(z+re^{i\theta})|^2\, rd\theta \, dr = 2\pi  \|\nabla u\|_{L^2(\Omega)}^2.
\end{align*}
\end{proof}

\begin{proof}[Proof of Lemma \ref{lem:mincont}]
In any of the two cases it follows from Lemma \ref{lem:CL} that there exists $\rho>0$ such that $r<\rho<\sqrt{r}$ and
$$
L(\rho)^2 \leq \frac{4\pi\|\nabla u\|_{L^2(\Omega)}^2}{\abs{\log(r)}}.
$$
If now $B_{\sqrt{r}}(z_0)\subset\Omega$ and $z_1,z_2\in B_r(z_0)$, we denote by $z_i^\pm$ the associated boundary points such that
$$\partial B_\rho(z_0)\cap\{z_i+\R e_1\}=\{z_i^-,z_i^+\}.$$
Hence we may write
\begin{align*}
    u(z_1)-u(z_2)\leq u(z_1^+)-u(z_2^-) \leq L(\rho) \leq \frac{\sqrt{4\pi}\|\nabla u\|_{L^2(\Omega)}}{\sqrt{|\log(r)|}},
\end{align*}
where for the first inequality we use $\partial_{x_1} u \geq 0$. This concludes the proof of the first part of the lemma.

%For the second part, since $\Omega$ is $\cC^0$ regular, there exists $r_0\in(0,1)$ such that 
%$\partial B_r(z_0)\cap\Omega$ has at most one component for any $r\in(0,\sqrt{r_0})$. 
%
%Hence, if $r<r_0$, there exists
%$\rho>0$ such that $r<\rho<\sqrt{r}<\sqrt{r_0}$ and $$L(\rho)^2 \leq \frac{4\pi\|\nabla u\|_{L^2(\Omega)}^2}{\log(r)}.$$
For the second part given $z_1,z_2\in B_{\sqrt{r}}(z_0)\cap\overline{\Omega}$, we denote by $z_i^\pm$ the boundary points now given by
$$\partial (B_\rho(z_0)\cap\Omega)\cap\{z_i+\R e_1\}=\{z_i^-,z_i^+\}.$$
We may once more write
\begin{align*}
    u(z_1)-u(z_2)&\leq u(z_1^+)-u(z_2^-) \leq L(\rho)+\osc(u;B_\rho(z_0)\cap\partial\Omega)\\ &\leq \frac{\sqrt{4\pi}\|\nabla u\|_{L^2(\Omega)}}{\sqrt{|\log(r)|}}+\osc(u;B_{\sqrt{r}}(z_0)\cap\partial\Omega),
\end{align*}
which finishes the proof of the lemma.
\end{proof}

\section{Partial regularity near good points}\label{sec_partial_regularity}

In this section we will apply the result of Savin \cite{savin} to the minimizers $(u_\varepsilon)_\varepsilon$, in order to obtain partial regularity for the limit $u$. Throughout the section we mean by $(u_\varepsilon)_\varepsilon$ the subsequence from Proposition \ref{prop:extended_gamma_convergence} converging to a minimizer $u$ of problem \eqref{eq:var_prop_sec_3}.

Let
$$G:=\{p\in\R^2:\ p_1\neq 0,\ |p_2|<1\},$$
and note that $D^2 F(p)$ is non-degenerate and positive definite for $p\in G$. Therefore, we would like to show that whenever $\nabla u$ takes values in this ``good'' set, one may deduce higher regularity of $u$ via the Euler-Lagrange equation. However, the set of points $x\in\Omega$ for which $\nabla u(x)\in G$ is a priori not open, hence we will need to adapt our argument and use the Euler-Lagrange equations associated with the approximation $u_\varepsilon$ in order to deduce the openness of this set, and hence the regularity.
The main result of this section is the following proposition, which allows to directly conclude the partial regularity stated in Theorem \ref{thm:existence_and_part_regularity}.

\begin{proposition}\label{prop:parreg}
For any $p_0\in G$ there exists $\delta>0$, $R_0>0$ such that whenever $B_r(x_0)\subset \Omega$, $r\in(0,R_0)$ and
$$\fint_{B_r(x_0)} |\nabla u - p_0|^2<\delta,$$
then
$u\in \cC^{2,\alpha}(B_{r/2}(x_0))$ for some $\alpha\in (0,1)$.
\end{proposition}

The idea behind the proof is the following. Let $x_0\in\Omega$ and $r>0$ be such that $B_r(x_0)\subset\Omega$, and let $u_0\in\R$, $p_0\in G$. 
For $\varepsilon>0$ we know that $u_\varepsilon:B_r(x_0)\to\R$ is a viscosity solution, c.f. Definition \ref{def:viscosity_solution}, to the Euler-Lagrange equation
$$D^2 \hat F_\varepsilon (\nabla u_\varepsilon):D^2 u_\varepsilon+\partial_z V(x,u_\varepsilon)=0.$$
We will show that after rescaling
\begin{align}\label{eq:uscale}
    u_\varepsilon(x)=u_0+p_0\cdot(x-x_0)+r v_\varepsilon((x-x_0)/r),
\end{align}
which implies that $v_\varepsilon:B_1(0)\to\R$ is a (viscosity) solution
of
\begin{align}\label{eq:vvisc}
    D^2\hat F_\varepsilon(p_0+\nabla v_\varepsilon):D^2v_\varepsilon+r\partial_zV(x_0+rx,u_0+rp_0\cdot x+rv_\varepsilon)=0,
\end{align}
and some further technical manipulations, we can apply the following regularity result of Savin. Recall that $\cS^{n\times n}$ denotes the set of symmetric $n\times n$ matrices. We also set $B_1:=B_1(0)\subset\R^n$.

\begin{theorem}[Savin \text{\cite[Theorem 1.3]{savin}}]\label{thm:savin}
Let $\cF:\cS^{n\times n}\times\R^n\times\R\times B_1\rightarrow \R$, $(M,p,z,x)\mapsto \cF(M,p,z,x)$ be a measurable map and $K,\bar\delta>0$, $\Lambda\geq \lambda>0$ constants satisfying
\begin{itemize}
    \item[(H1)] $\cF(M+N,p,z,x)\geq \cF(M,p,z,x)$ \\ for $M,N\in \cS^{n\times n}$, $N\geq 0$, $\abs{p}\leq \bar \delta,\abs{z}\leq \bar \delta$, $x\in B_1$,
    \item[(H2)] $\Lambda \norm{N}\geq \cF(M+N,p,z,x)-\cF(M,p,z,x)\geq \lambda\norm{N}$\\
    for $M,N\in \cS^{n\times n}$, $N\geq 0$, $\norm{M}\leq \bar{\delta}$, $\norm{N}\leq \bar{\delta}$, $\abs{p}\leq \bar\delta$, $\abs{z}\leq \bar \delta$, $x\in B_1$,
    \item[(H4)] $\cF(0,0,0,x)=0$, and  in the $\bar\delta$-neighborhood of  $\set{(0,0,0,x):x\in B_1}$ the map $\cF$ is of class $\cC^2$ there holds the uniform bound $\norm{D^2\cF}\leq K$.
\end{itemize}
Then there exists a constant $c_1>0$ depending only on $K,\bar{\delta},\Lambda,\lambda$ such that if the function $u:B_1\rightarrow\R$ is a viscosity solution of $\cF(D^2u,\nabla u,u,x)=0$ with $\norm{u}_{L^\infty(B_1)}\leq c_1$, then $u\in \cC^{2,\alpha}(B_{1/2})$ and $\norm{u}_{\cC^{2,\alpha}(B_{1/2})}\leq \bar{\delta}$.
\end{theorem}

\begin{remark}\label{rem:savin}
We will apply  Theorem \ref{thm:savin} to maps $\cF$ that are of class $\cC^2$ with respect to $M$, $p$ and $z$, whereas they are only H\"older continuous with respect to $x$. I.e. instead of (H4) we have
\begin{itemize}
    \item[(H4$^\prime$)] $\cF(0,0,0,x)=0$, and in the $\bar\delta$-neighborhood of  $\set{(0,0,0,x):x\in B_1}$ the derivatives $D^2_{(M,p,z)}\cF$ exist, are continuous and $\norm{D^2_{(M,p,z)}\cF}\leq K$. Moreover, $\norm{\cF(M,p,z,\cdot)}_{\cC^{0,\beta}(\overline{B}_1)}\leq K$ for some $\beta\in (0,1)$ and any $(M,p,z)$ in said $\bar{\delta}$-neighborhood.
\end{itemize}
However, the proof in \cite[Sections 3 and 4]{savin} shows that the conclusions of Theorem \ref{thm:savin} remain valid for any $\alpha<\beta$.
\end{remark}

%Hence we proceed as follows.
In the next two lemmas we show that the rescaled functions $v_\varepsilon:B_1\rightarrow\R$ introduced in \eqref{eq:uscale} satisfy the needed $L^\infty$-bound for a suitable choice of the constant $u_0\in\R$. 

\begin{lemma}\label{lem:ourcase}
Let $x_0\in\Omega$, $r>0$ such that $B_r(x_0)\subset\Omega$, $p_0\in G$.
There holds
 $$\lim_{\varepsilon\to 0} \fint_{B_r(x_0)} |\nabla u_\varepsilon - p_0|^2 \, dx \leq \left(1+\frac{1}{(p_{0,1})^2}\right)\fint_{B_r(x_0)} |\nabla u - p_0|^2 \, dx.$$
\end{lemma}

\begin{proof}
Using the associated Young measure $\boldsymbol{\nu}=((\nu_x)_{x\in\Omega},0,0)$ given in Proposition \ref{prop:extended_gamma_convergence}, and in particular the fact that the concentration measure $\lambda$ vanishes, which allows to approximate the non-continuous indicator function of $B_r(x_0)$ by continuous integrands, there holds 
\begin{align*}
   &\lim_{\varepsilon\to 0} \fint_{B_r(x_0)} |\nabla u_\varepsilon - p_0|^2 \, dx = \fint_{B_r(x_0)} \int_{\R^2} |v - p_0|^2 \, d\nu_x (v) \, dx \\ 
   &\phantom{=}= \fint_{B_r(x_0)\cap \set{\partial_{x_1}u=0}} \int_{-1}^1 |(0,v_2) - p_0|^2 \, d\hat\nu_x(v_2) \, dx+\fint_{B_r(x_0) \cap \set{\partial_{x_1}u\neq 0}} |\nabla u - p_0|^2\:dx \\ 
   &\phantom{=}=
   \fint_{B_r(x_0)\cap \set{\partial_{x_1}u=0}}  \int_{-1}^1 |v_2-p_{0,2}|^2-|\partial_{x_2}u(x)-p_{0,2}|^2 \, d\hat\nu_x(v_2) \, dx \\
   &\hspace{40pt}+\fint_{B_r(x_0)} |\nabla u - p_0|^2 \, dx .
\end{align*}
We further estimate
\begin{align*}
   &\fint_{B_r(x_0)\cap \set{\partial_{x_1}u=0}}  \int_{-1}^1 |v_2-p_{0,2}|^2-|\partial_{x_2}u(x)-p_{0,2}|^2 \, d\hat\nu_x(v_2) \, dx \\
   &\phantom{=}= \fint_{B_r(x_0)\cap \set{\partial_{x_1}u=0}} \int_{-1}^1 |v_2|^2-|\partial_{x_2}u(x)|^2 \, d\hat\nu_x(v_2) \, dx\leq \fint_{B_r(x_0)\cap \set{\partial_{x_1}u=0}}  \, dx  \\
   &\phantom{=}=\frac{1}{(p_{0,1})^2} \fint_{B_r(x_0)\cap \set{\partial_{x_1}u=0}} |\partial_{x_1} u(x)-p_{0,1}|^2 \, dx\leq \frac{1}{(p_{0,1})^2}\fint_{B_r(x_0)} |\nabla u - p_0|^2 \, dx ,
\end{align*}
which concludes the proof of the lemma.
\end{proof}

\begin{lemma}\label{lem:L_infty_bound_on_v_epsilon}
Let $p_0\in G$, $x_0\in\Omega$, $r>0$ with $B_r(x_0)\subset\Omega$ and 
\begin{equation}\label{eq:condition_on_r_and_p}
\abs{p_{0,1}}-4r\norm{\partial_z V}_{L^\infty(\Omega\times \R)}>0.
\end{equation}
There exists $\varepsilon_0\in(0,1)$ such that for any $\varepsilon\in(0,\varepsilon_0)$ there exists $u_0\in\R$ and $r_1\in (r/2,r)$, such that the rescaled functions $v_\varepsilon:B_1\rightarrow \R$,
\[
v_\varepsilon(x):=\frac{u_\varepsilon(x_0+r_1x)-u_0}{r_1}-p_0\cdot x
\]
satisfy 
\begin{equation}\label{eq:L_Infty_bound_for_v_epsilon_lemma}
    \norm{v_\varepsilon}_{L^\infty(B_1)}\leq C\left(\fint_{B_r(x_0)}\abs{\nabla u-p_0}^2\:dx\right)^{\frac{1}{2}}+2r\norm{\partial_z V}_{L^\infty(\Omega\times \R)}
\end{equation}
for a constant $C>0$ depending only on $(p_{0,1})^{-1}$.
\end{lemma}
\begin{proof} For $p_0$ and $B_r(x_0)$ as stated we set
$$E:=\fint_{B_r(x_0)} |\nabla u - p_0|^2\:dx.$$
We assume $E>0$, otherwise our main goal, Proposition \ref{prop:parreg}, is trivial.
By Lemma \ref{lem:ourcase} there exists $\varepsilon_0>0$ small enough such that
$$\fint_{B_r(x_0)} |\nabla u_\varepsilon - p_0|^2\:dx<C_0E,$$
for  $\varepsilon\in(0,\varepsilon_0)$, where $C_0:=2+\frac{1}{(p_{0,1})^2}$. Moreover, in view of condition \eqref{eq:condition_on_r_and_p} we can, after shrinking $\varepsilon_0>0$, also assume that for $\varepsilon\in (0,\varepsilon_0)$ the extension $\hat F_\varepsilon$ coincides with the original approximation $F_\varepsilon$ in a neighborhood of the segment \begin{equation}\label{eq:sec5_good_segment}
I:=\left[p_0-4r\norm{\partial_z V}_\infty e_1,p_0+4r\norm{\partial_z V}_\infty e_1\right],
\end{equation}
which is compactly contained in the good set $G$.

It follows that there exists $r_1\in(r/2,r)$ such that
$$\frac{r}{2}\int_{\partial B_{r_1}(x_0)} |\nabla u_\varepsilon - p_0|^2\:dS\leq \int_{B_r(x_0)} |\nabla u_\varepsilon - p_0|^2\:dx<C_0E r^2\pi,$$
hence
$$\frac{1}{2\pi r_1}\int_{\partial B_{r_1}(x_0)} |\nabla u_\varepsilon - p_0|^2\:dS<2C_0E.$$
Then, defining $v_\varepsilon$ as stated with $u_0:=\fint_{\partial B_{r_1}(x_0)}u_\varepsilon\:dS$ and $r_1$ as chosen before, one obtains that
$$\fint_{\partial B_1} |\nabla v_\varepsilon|^2\:dS<2C_0E.$$
By Morrey's inequality for instance, we get
\begin{align}\label{eq:vmorrey}
    \|v_\varepsilon\|_{L^\infty(\partial B_1)}\lesssim \sqrt{E}
\end{align}
with a proportionality constant depending only on $(p_{0,1})^{-1}$.

Now define 
\[
A_0:=\|\partial_z V\|_\infty((1+\varepsilon_0)^2-(p_{0,2})^2),
\]
and note that for any $\varepsilon\in(0,\varepsilon_0)$, since $\partial^2_{1}\hat F_\varepsilon(p_0)=\partial^2_1 F_\varepsilon(p_0)=\frac{1}{(1+\varepsilon)^2-(p_{0,2})^2}$, we have
\begin{equation}\label{eq:sec5_partial_F}
\partial^2_{1}\hat F_\varepsilon(p_0)  A_0 >\|\partial_z V\|_\infty.
\end{equation}

Finally, we want to show that
\begin{align}\label{eq:vlinfty}
    \norm{v_\varepsilon}_{L^\infty(B_1)}\leq \|v_\varepsilon\|_{L^\infty(\partial B_1)}+\frac{r_1A_0}{2}.
\end{align}
We argue by contradiction, suppose that
\[
\eta:=\max_{B_1}\set{v_\varepsilon-\|v_\varepsilon\|_{L^\infty(\partial B_1)}-\frac{r_1A_0}{2}(1-x_1^2)}>0.
\]
Then, the function $\bar{\phi}:B_1\rightarrow\R$,
\[
\bar\phi(x):=\eta+\|v_\varepsilon\|_{L^\infty(\partial B_1)}+\frac{r_1A_0}{2}(1-x_1^2)
\] 
touches $v_\varepsilon$ from above at some point $\bar x\in B_1$. Since $v_\varepsilon$ is a viscosity solution to \eqref{eq:vvisc}, there must hold
\begin{align*}
    0&\leq  D^2\hat F_\varepsilon(p_0+\nabla \bar\phi(\bar{x})):D^2\bar\phi(\bar{x})+r_1\partial_zV(x_0+r_1\bar x,u_0+r_1p_0\cdot\bar x+r\bar\phi(\bar{x}))\\
    &\leq -\partial^2_{1} \hat F_\varepsilon(p_0-r_1A_0\bar{x}_1e_1)r_1A_0+r_1\|\partial_z V\|_\infty=r_1\left(-\partial_{p_1}^2F_\varepsilon(p_0)A_0+\norm{\partial_z V}_\infty\right),
\end{align*}
which contradicts \eqref{eq:sec5_partial_F}. Note here that in the last step we have made use of the fact that $p_0-r_1A_0\bar{x}_1e_1$ lies on the segment $I$ defined in \eqref{eq:sec5_good_segment}.

A similar contradiction is obtained if one assumes that 
\[
\min_{B_1}\set{v_\varepsilon+\|v_\varepsilon\|_{L^\infty(\partial B_1)}+\frac{r_1A_0}{2}(1-x_1^2)}<0.
\]
Therefore, combining \eqref{eq:vmorrey} and \eqref{eq:vlinfty} we obtain the inequality stated in \eqref{eq:L_Infty_bound_for_v_epsilon_lemma}.
\end{proof}

The next Lemma will help us to set up a family of functionals which satisfies the conditions of Theorem \ref{thm:savin}.

\begin{lemma}\label{lem:implicit_function_argument}
Let $p_0\in G$ and $\alpha\in(0,1)$. There exists $\varepsilon_1\in(0,1)$ and $r_2>0$ such that for any $u_0\in\R$, $x_0\in\R^2$, $r\in(0,r_2)$ with $B_r(x_0)\subset\Omega$, $\varepsilon\in [0,\varepsilon_1]$ the boundary value problem
\begin{align}\label{eq:implicit_function_BVP}
    \begin{cases}
    D^2\hat F_\varepsilon(p_0+\nabla \phi):\nabla^2\phi+r\partial_z V(x_0+rx,u_0+rp_0\cdot x+r\phi)=0&in~B_1,\\
    \phi=0,&on~\partial B_1,
    \end{cases}
\end{align}
has a $\cC^{2,\alpha}$ solution $\phi_r^{\varepsilon,x_0,u_0,p_0}$ satisfying
$\norm{\phi_r^{\varepsilon,x_0,u_0,p_0}}_{\cC^{2,\alpha}(\overline{B}_1)}\rightarrow 0$ as $r\rightarrow 0$ uniformly in $(\varepsilon,x_0,u_0)\in[0,\varepsilon_1]\times K$ for any $K\subset \subset \Omega\times \R$.
\end{lemma}
\begin{proof}
Let $p_0\in G$ and $\alpha\in(0,1)$ be fixed. We first of all pick $\varepsilon_1>0$ and $\eta_0>0$ such that $\overline{B_{\eta_0}(p_0)}\subset G$ and $\hat F_\varepsilon(p_0+p)=F_\varepsilon(p_0+p)$ for all $\varepsilon\in (0,\varepsilon_1]$, $\abs{p}\leq \eta_0$.

Next let 
\[
\cB:=\set{\phi\in \cC^{2,\alpha}(\overline{B}_1):\phi|_{\partial B_1}=0,~\norm{\phi}_{\cC^{2,\alpha}(\overline{B}_1)}< \eta_0}
\]
and consider for $R>0$ the family of maps $\fF^{a}_{p_0}:[0,R)\times\cB\rightarrow \cC^{0,\alpha}(\overline{B}_1)$,
\begin{align*}
\fF_{p_0}^a(r,\phi)(x)=D^2 F_\varepsilon(p_0+\nabla \phi(x))&:\nabla^2\phi(x)+r\partial_z V(x_0+rx,u_0+rp_0\cdot x+r\phi(x)),
\end{align*}
where $a$ is an abbreviation for the tuple of parameters $a:=(\varepsilon,x_0,u_0)$ satisfying $\varepsilon\in [0,\varepsilon_1]$, $x_0\in\Omega$, $\dist(x_0,\partial\Omega)> R$, $u_0\in\R$. Observe that $\fF^a_{p_0}$ is well-defined and that for $\varepsilon>0$ the equation $\fF^a_{p_0}(r,\phi)=0$ holds true if and only if $\phi$ solves the boundary value problem \eqref{eq:implicit_function_BVP}.

Since $F_\varepsilon$ is smooth on the closure of $B_{\eta_0}(p_0)$ and by \eqref{eq:condition_V1}, one can check that 
$\fF^a_{p_0}$ is continuous and Fr\'echet-differentiable with respect to $\phi$, and that the corresponding derivative $D_\phi \fF^a_{p_0}:[0,\infty)\times\cB \rightarrow \cL\left(\set{\psi\in\cC^{2,\alpha}(\overline{B}_1):\psi|_{\partial B_1}=0};\cC^{0,\alpha}(\overline{B}_1)\right)$ is continuous. For later use we also like to point out that not only each $D_\phi \fF^a_{p_0}$ is continuous as a function of $(r,\phi)$, but that also the joint function  
\[
[0,\varepsilon_1]\times\set{\dist(x,\partial\Omega)>R}\times \R\times[0,R)\times\cB\ni (\varepsilon,x_0,u_0,r,\phi)\mapsto D_\phi \fF^{(\varepsilon,x_0,u_0)}_{p_0}(r,\phi)
\]
is continuous. The same is true for $\fF_{p_0}$ itself.

Moreover, there clearly holds $\fF^a_{p_0}(0,0)=0$ for any considered parameter triple $a$ and $\cL^a_{p_0}:=D_\phi\fF^a_{p_0}(0,0):\set{\psi\in\cC^{2,\alpha}(\overline{B}_1):\psi|_{\partial B_1}=0}\rightarrow \cC^{0,\alpha}(\overline{B}_1)$ is given by
\[
\cL^a_{p_0}[\psi](x)=D^2F_\varepsilon(p_0):\nabla^2\psi(x). 
\]
By Schauder theory, cf. \cite{Gilbarg_Trudinger}, we see that $\cL^a_{p_0}$ is an isomorphism with 
\begin{equation}\label{eq:ift_schauder_estimate}
    \norm{(\cL^a_{p_0})^{-1}[\theta]}_{\cC^{2,\alpha}(\overline{B}_1)}\leq C\norm{\theta}_{\cC^{0,\alpha}(\overline{B}_1)}
\end{equation}
for all $\theta\in \cC^{0,\alpha}(\overline{B}_1)$ and a constant $C>0$ depending only on the ellipticity constants of $D^2F_\varepsilon(p_0)$. Therefore, since $p_0\in G$ is a good point, this constant can be chosen independently of $\varepsilon\in [0,\varepsilon_1]$. Thus \eqref{eq:ift_schauder_estimate} holds with a constant $C>0$ independent of the considered parameter triple $a=(\varepsilon,x_0,u_0)$. 

Hence by using a quantitative version of the implicit function theorem we conclude the existence of $r_2\in(0,R)$ depending solely on $p_0$, as well as for every $a$ a continuous family $[0,r_2)\rightarrow \cB$, $r\mapsto \phi_r
^a$ with $\phi_0^a=0$ satisfying $\fF_{p_0}^a(r,\phi_r^a)=0$, $r\in[0,r_2)$. 

Now it only remains to show that the convergence $\phi_r^a\rightarrow 0$ in $\cC^{2,\alpha}(\overline{B}_1)$ as $r\rightarrow 0$ is in fact uniform in $a=(\varepsilon,x_0,u_0)$ for $\varepsilon\in[0,\varepsilon_1]$ and $(x_0,u_0)$ taken from a compact subset of $\Omega\times \R$. First of all note that there exists $R>0$ such that the above defined map $\fF_{p_0}^a$ is well-defined for any $(x_0,u_0)\in K$.

The uniform convergence then follows from the fact that, as observed earlier, the maps $\fF_{p_0}$, $D_\phi\fF_{p_0}$ are continuous as functions of $(\varepsilon,x_0,u_0,r,\phi)$, such that the implicit function theorem also provides us with the continuity of the joint map $[0,\varepsilon_1]\times\set{\dist(x,\partial\Omega)\geq R}\times\R\times[0,r_2)\rightarrow\cB$, 
\[
(\varepsilon,x_0,u_0,r)\mapsto \phi_r^{(\varepsilon,x_0,u_0)}.
\]
Now the stated uniform convergence is a direct consequence of the compactness of the set $[0,\varepsilon_1]\times K\times\{0\}$.
\end{proof}

\begin{proof}[Proof of Proposition \ref{prop:parreg}]
We fix $p_0\in G$. For $\varepsilon\in (0,\varepsilon_1)$, $r\in (0,r_2)$, $x_0\in\R^2$, $B_r(x_0)\subset\Omega$ and $u_0\in\R$ as in Lemma \ref{lem:implicit_function_argument} we consider the family of nonlinear maps $\cF_r^{\varepsilon,x_0,u_0}:\cS^{2\times 2}\times\R^2\times \R\times B_1\rightarrow\R$ defined by
\begin{align*}
\cF_r^{\varepsilon,x_0,u_0}(M,p,z,x):=D^2\hat{F}_\varepsilon&(p_0+p+\nabla \phi_r^a(x)):(M+D^2\phi_r^a(x))\\&+r\partial_zV(x_0+rx,u_0+rp_0\cdot x+rz+r\phi^a_r(x)).
\end{align*}
Here we have denoted by $\phi_r^a:=\phi_r^{\varepsilon,x_0,u_0,p_0}$ the $\cC^{2,\alpha}$ solution of \eqref{eq:implicit_function_BVP} provided by Lemma \ref{lem:implicit_function_argument}.

We will now show that conditions (H1), (H2) and  (H4) from Theorem \ref{thm:savin}, condition (H4$^\prime$) from Remark \ref{rem:savin} resp., are satisfied for $\cF_r^{\varepsilon,x_0,u_0}$.
Indeed, condition (H1) in fact holds even globally, while (H2) holds true with a constant $\bar{\delta}>0$ independent of $r,\varepsilon,x_0,u_0$ provided $(x_0,u_0)$ is restricted to a compact subset of $\Omega\times \R$ and $r,\varepsilon$ are chosen small enough. That this is possible is a consequence of $p_0$ being a good point and Lemma \ref{lem:implicit_function_argument}.

Furthermore, we see that $\cF_r^{\varepsilon,x_0,u_0}(0,0,0,x)=0$ by \eqref{eq:implicit_function_BVP} and the partial second derivative $D^2_{(M,p,z)}F_r^{\varepsilon,x_0,u_0}$ is bounded on a neighborhood of $\set{(0,0,0,x):x\in B_1}$. The size of the neighborhood and the bound can be chosen with the same (in)-dependencies as $\bar{\delta}>0$ above in (H2). Moreover, in view of \eqref{eq:condition_V1} and Lemma \ref{lem:implicit_function_argument} we see that any $\cF^{\varepsilon,x_0,u_0}_r$ is H\"older continuous with any exponent $\beta\in(0,1)$, where for fixed $\beta$ the corresponding H\"older norm can again assumed to be bounded uniformly in $r,\varepsilon,x_0,u_0$ for $(x_0,u_0)$ from a compact set and $r,\varepsilon$ small enough.
We therefore also have property (H4$^\prime$).

We may then apply Theorem \ref{thm:savin} to $\cF_r^{\varepsilon,x_0,u_0}$ and obtain a constant $c_1>0$ independent of $\varepsilon,r$ small and $(x_0,u_0)$ from a compact subset of $\Omega\times \R$ having the property that any viscosity solution $v:B_1\rightarrow\R$ of $\cF_r^{\varepsilon,x_0,u_0}=0$ with $\norm{v}_{L^\infty(B_1)}\leq c_1$ belongs to $\cC^{2,\alpha}(B_{1/2})$ and $\norm{v}_{\cC^{2,\alpha}(B_{1/2})}\leq \bar{\delta}$.

Let us now also fix $x_0\in \Omega$ and consider for $\varepsilon,r>0$ small, such that the above conclusions hold true, as well as the conclusions of Lemma \ref{lem:L_infty_bound_on_v_epsilon}, the function $w_\varepsilon:B_1\rightarrow\R$,
\begin{align*}
w_\varepsilon(x):=v_\varepsilon(x)-\phi_{r_1}^{\varepsilon,x_0,u_0,p_0}(x)=\frac{u_\varepsilon(x_0+r_1x)-u_0}{r_1}-p_0\cdot x-\phi_{r_1}^{\varepsilon,x_0,u_0,p_0}(x),
\end{align*}
where $r_1\in(r/2,r)$, $u_0\in\R$ are given by Lemma \ref{lem:L_infty_bound_on_v_epsilon}.

Then $w_\varepsilon$ satisfies 
$
\cF_{r_1}^{\varepsilon,x_0,u_0}(D^2w_\varepsilon,\nabla w_\varepsilon,w_\varepsilon,x)=0
$
and by Lemma \ref{lem:L_infty_bound_on_v_epsilon} there holds
\begin{align*}
    \norm{w_\varepsilon}_{L^\infty(B_1)}\leq C\left(\fint_{B_r(x_0)}\abs{\nabla u-p_0}^2\:dx\right)^{\frac{1}{2}}+2r&\norm{\partial_z V}_{L^\infty(\Omega\times \R)}\\
    &\phantom{nnnnnn}+\norm{\phi_{r_1}^{\varepsilon,x_0,u_0,p_0}}_{L^\infty(B_1)}
\end{align*}
for a constant $C>0$ depending only on $p_0$.
In view of Lemma \ref{lem:implicit_function_argument} we therefore reach
$$\|w_\varepsilon\|_{L^\infty(B_1)}\leq c_1 $$
by assuming that $r>0$ and $\fint_{B_r(x_0)}\abs{\nabla u-p_0}^2\:dx$ are small enough.
Therefore, we may conclude that $(w_\varepsilon)_{\varepsilon}$ is bounded in $\cC^{2,\alpha}(B_{1/2})$.

It then follows that $(u_\varepsilon)_{\varepsilon}$ is bounded in $\cC^{2,\alpha}(B_{r/2}(x_0))$, and hence converges to the limit $u$ in $\cC^{2,\alpha-}(B_{r/2}(x_0))$. This finishes the proof of Proposition \ref{prop:parreg}.
\end{proof}

\section{Further properties}\label{sec_further_properties}
Here we collect some additional properties for our minimizer that are important in relation to the role of $\nabla u$ as a subsolution to the Boussinesq equation. Throughout this section we consider again $(u_\varepsilon)_\varepsilon$ and $u$ as in Proposition \ref{prop:extended_gamma_convergence}.

\subsection{Topology of $\Omega'$}\label{sec:property_of_good_set}

We begin by noting that Lemma \ref{lem:good_set_property} is a direct consequence of the one-sided maximum principle in Lemma \ref{lem:good_set_property_epsilon} and the uniform convergence in Proposition \ref{prop:u_continuous}.

\subsection{Energy balance}\label{sec_energy} The Young measure representation in Proposition \ref{prop:gamma_convergence_for_minimizers} allows us to pass to the limit in the energy balance in Lemma \ref{lem:energy_conservation_for_approximation}.
\begin{lemma}\label{lem:energy_balance}
The measure $(\partial_{p_1}\hat{F}_\varepsilon(\nabla u_\varepsilon(x))\partial_{x_1} u_\varepsilon(x)-\hat{F}_\varepsilon(\nabla u_\varepsilon(x)))\:dx$ converges weakly to $F(\nabla u(x))\:dx$. In particular, if \eqref{eq:V1prime} holds true, then
\begin{equation}\label{eq:energy_balance}
    \frac{d}{dx_1}\int_{-L}^LF(\nabla u(x))+V(x_2,u(x))\:dx_2=0
\end{equation}
in the weak sense.
\end{lemma}
\begin{proof}
Let $\psi:\overline{\Omega}\rightarrow \R$ be an arbitrary continuous, bounded  function and set $G_\varepsilon(p):=\partial_{p_1}\hat{F}_\varepsilon(p)p_1-\hat{F}_\varepsilon(p)$, as well as
\begin{gather*}
    I_\varepsilon:=\int_\Omega G_\varepsilon(\nabla u_\varepsilon(x))\psi(x)\:dx,\quad I_\varepsilon^1:=\int_\Omega G_\varepsilon(\nabla u_\varepsilon(x))\psi(x)\eta(\nabla u_\varepsilon(x))\:dx,\\
    I_\varepsilon^2:=I_\varepsilon-I^1_\varepsilon=\int_\Omega G_\varepsilon(\nabla u_\varepsilon(x))\psi(x)(1-\eta(\nabla u_\varepsilon(x)))\:dx,
\end{gather*}
where $\eta:\R^2\rightarrow [0,1]$ is continuous with support compactly contained in the open strip $\set{p\in\R^2:\abs{p_2}<1}$.

By Proposition \ref{prop:extended_gamma_convergence} and the convergence of $G_\varepsilon(p)$ to $\partial_{p_1}F(p)p_1-F(p)=F(p)$, which is uniform on the support of $\eta$, it follows that 
\begin{align*}
    \lim_{\varepsilon\rightarrow 0} I_\varepsilon^1= \int_\Omega F(\nabla u)\eta(\nabla u)\psi\:dx.
\end{align*}
For $I^2_\varepsilon$ we use Lemma \ref{lem:first_extension_of_F_epsilon} \eqref{eq:estimate_of_extension_for_derivative} in order to estimate
\begin{align*}
    \abs{I_\varepsilon^2}&\leq \norm{\psi}_{L^\infty(\Omega)}\int_\Omega \left(\varepsilon+4\hat{F}_\varepsilon(\nabla u_\varepsilon)\right)(1-\eta(\nabla u_\varepsilon))\:dx\\
    &\rightarrow 4\norm{\psi}_{L^\infty(\Omega)}\int_\Omega F(\nabla u)(1-\eta(\nabla u))\:dx
\end{align*}
as $\varepsilon\rightarrow 0$. Indeed the convergence of the  integral can be seen by splitting it up, the use of the convergence of the total kinetic energies \eqref{eq:convergence_of_kinetic_energy_alone} and the same argument as above for the convergence of $I_\varepsilon^1$. We conclude
\begin{align*}
    \abs{I_\varepsilon-\int_\Omega F(\nabla u)\eta(\nabla u)\psi\:dx}\leq o(1)+ C \int_\Omega F(\nabla u)(1-\eta(\nabla u))\:dx
\end{align*}
as $\varepsilon\rightarrow 0$
with a constant $C>0$ depending on $\psi$, but not on $\eta$. 

Taking now a sequence $(\eta_j)_j$ converging pointwise to $1$ on $\set{p\in\R^2:\abs{p_2}\leq 1}$ we deduce that 
\[
\lim_{\varepsilon\rightarrow 0}I_\varepsilon=\int_\Omega F(\nabla u)\psi\:dx.
\]
Recall here that $F(p)=0$ for $p_1=0$, $\abs{p_2}=1$ and that the measure of the set $\abs{\set{x\in\Omega:\partial_{x_1} u(x)\neq 0,~\abs{\partial_{x_2}u(x)}=1}}$ is $0$ as otherwise $F(\nabla u)$ would not be integrable. Thus we have shown the stated weak convergence.

If now \eqref{eq:V1prime} holds true it easily follows from Lemma \ref{lem:energy_conservation_for_approximation} that
\begin{align*}
    \int_0^T\varphi'(x_1)\int_{-L}^LF(\nabla u(x))+V(x_2,u(x))\:dx_2\:dx_1=0
\end{align*}
for all $\varphi\in \cC_c^1(0,T)$.
\end{proof}

Note that at this point all claims of Theorem \ref{thm:existence_and_part_regularity} have been shown.

\subsection{Stronger attainment of boundary data}\label{sec_initial_data}
As already discussed in Sections \ref{sec:reformulation}, \ref{sec_final_configuration}, considering functions $u\in X$ implies that $\rho:=\partial_{x_2}u$ and $m:=-\partial_{x_1}u$ satisfy
\begin{align}\label{eq:weak_balance}
    \int_\Omega \rho\partial_{x_1}\varphi+m\partial_{x_2}\varphi\:dx+\int_{-L}^L\sign(x_2)(\varphi(0,x_2)\:dx_2+\varphi(T,x_2))\:dx_2=0
\end{align}
for all $\varphi\in H^1(\Omega)$. Concerning the boundary data for $m$ and the initial, final resp., data for $\rho$ we can from this conclude certain weak convergences, see for instance Lemma \ref{lem:weak_convergence_of_initial_data} below. The goal of this subsection is to improve these weak convergences to strong convergences. This is the statement of Lemmas \ref{lem:tlim},\ref{lem:xlim}. Moreover, the energy balance allows us to also conclude that $m$ attains $0$ initial and final data, see Lemma \ref{lem:txlim}. This information has not been encoded in the function space $X$, not even in a weaker form.

We begin with the claimed weak convergence of $\rho$ near $\set{x_1=0}$. A similar statement holds true near $\set{x_1=T}$.
\begin{lemma}\label{lem:weak_convergence_of_initial_data}
For $a>0$, let $v_a(x):=u(ax_1,x_2)-(|x_2|-L)$. Then there holds $\partial_{x_2} v_{a}\overset{\ast}{\rightharpoonup} 0$ in $L^\infty((0,1)\times(-L,L))$ as $a\rightarrow 0$.
\end{lemma}
\begin{proof}
Let $\psi\in\cC^1([0,1]\times[-L,L])$ and define for $a\in(0,T)$ the function $\varphi_a:\overline{\Omega}\rightarrow\R$, 
\[
\varphi_a(x)=-\int_{x_1/a}^1\psi(x_1',x_2)\:dx_1'
\]
for $x_1\in[0,a)$ and $\varphi_a(x)=0$ for $x_1\geq a$. Then using \eqref{eq:weak_balance} one computes
\begin{align*}
    \int_0^1\int_{-L}^L\partial_{x_2}v_a(x)\psi(x)\:dx&=\int_\Omega (\partial_{x_2}u(x)-\sign(x_2))\partial_{x_1}\varphi_a(x)\:dx\\
    &=\int_\Omega \partial_{x_1}u(x)\partial_{x_2}\varphi_a(x)\:dx\rightarrow 0
\end{align*}
as $a\rightarrow 0$. 

The general case $\psi\in L^1((0,1)\times (-L,L))$ follows by observing that $(\partial_{x_2}v_a)_a$ is bounded in $L^\infty((0,1)\times(-L,L))$, cf. Proposition \ref{prop:extended_gamma_convergence}, and approximation. 
\end{proof}

This weak convergence can easily be improved.
\begin{lemma}\label{lem:tlim} For $v_a$ as in Lemma \ref{lem:weak_convergence_of_initial_data} there holds $\partial_{x_2}v_a\rightarrow 0$ in $L^1((0,1)\times (-L,L))$ as $a\rightarrow 0$.
\end{lemma}
\begin{proof}
Recalling that $|\partial_{x_2}u|\leq 1$ a.e. one may write
\begin{align*}
    \|\partial_{x_2} v_{a}\|_{L^1((0,1)\times(-L,L))}&= \int_0^1 \int_{-L}^L |\partial_{x_2}u(ax_1,x_2)-\sign(x_2)| \, dx \\&=
    \int_0^1 \int_0^L \left(\mathbbm{1}_{\set{x_2<0}}(x)-\mathbbm{1}_{\set{x_2>0}}(x)\right)\partial_{x_2} v_a(x) \, dx  \to 0,
\end{align*}
by the previous weak convergence and the fact that the indicators are in $L^1$.
\end{proof}

% Note that this implies that the boundary data for $m$ and the initial, final resp., data for $\rho$ is attained in the sense $m(\cdot,-L+b\cdot)\rightharpoonup 0$ weakly in $L^2((0,T)\times (0,1))$ as $b\rightarrow 0$ or $\rho(a\cdot,\cdot)\rightharpoonup \sign(x_2)$ weakly in $L^2((0,1)\times(-L,L))$ as $a\rightarrow 0$. Indeed, for $\psi\in\cC_c^\infty((0,T)\times (0,1))$ we set $\varphi_b:\overline{\Omega}\rightarrow\R$,
% \begin{align*}
%     \varphi_b(x)=
%     \int_0^{\frac{x_2+L}{b}}\psi(x_1,s)\:ds-\int_0^1\psi(x_1,s)\:ds
% \end{align*}
% for $x_2\in(-L,-L+b)$ and $\varphi_b=0$ for $x_2>-L+b$ and compute
% \begin{align*}
%     \int_0^T\int_0^1m(x_1,-L+bx_2)\psi(x)\:dx= -\int_0^T\int_{-L}^{-L+b}\rho(x)\partial_{x_1}\varphi_b(x)\:dx\rightarrow 0.
% \end{align*}
% Similarly one sees the corresponding weak convergence of $m$ near $\set{x_2=L}$ and of $\rho$ near $\set{x_1=0}$ and $\set{x_1=T}$.

For a similar strong convergence of $\partial_{x_1}u$ near $\set{x_2=\pm L}$ we argue directly.
\begin{lemma}\label{lem:xlim} Let \eqref{eq:V1prime} be satisfied, such that $\partial_{x_1}u\geq 0$.
There holds the convergence $\|\partial_{x_1} u(\cdot,-L+b\cdot) \|_{L^1((0,T)\times(0,1))}\to 0$ as $b\to 0$.
\end{lemma}
\begin{proof}
Using $\partial_{x_1} u\geq 0$, $u(\cdot,x_2)\in H^1(0,T)$ for a.e. $x_2\in (-L,L)$, as well as the continuity of $u$, cf. Proposition \ref{prop:u_continuous},  we observe that for any $b>0$ there holds
\begin{align*}
    \|\partial_{x_1} u(\cdot,-L+b\cdot) \|_{L^1((0,T)\times(0,1))}&=\int_0^T\int_0^1 \partial_{x_1} u(x_1,-L+bx_2)\, dx \\
    &=\frac{1}{b}\int_0^T\int_{-L}^{-L+b} \partial_{x_1} u(x)\, dx
    \\&=\frac{1}{b} \int_{-L}^{-L+b} u(T,x_2)-u(0,x_2)\, dx_2 \\&=
    \frac{1}{b} \int_{-L}^{-L+b} 2(L+x_2)\, dx_2 \leq \frac{1}{b} 2 \int_{-L}^{-L+b} b \, dx_2 = 2b,
\end{align*}
hence the claim follows.
\end{proof}

Utilizing the energy balance from Lemma \ref{lem:energy_balance} we in addition obtain a corresponding strong convergence of $\partial_{x_1}u$ near $\set{x_1=0}$ and $\set{x_1=T}$.

\begin{lemma}\label{lem:txlim} Let \eqref{eq:V1prime} be satisfied.
For the minimizer $u$ there holds $\partial_{x_1} u(a\cdot,\cdot)\to 0$ in $L^1((0,1)\times(-L,L))$ as $a\to 0^+$.
\end{lemma}
\begin{proof}
We may write
\begin{align*}
    &\int_0^1 \int_{-L}^L |\partial_{x_1} u(a x_1,x_2)| \, dx =\frac{1}{a}
    \int_0^a \int_{-L}^L \partial_{x_1} u(x) \, dx \\
    &\hspace{35pt}=\frac{1}{a}
    \int_0^a \int_{-L}^L \sqrt{2F(\nabla u(x))}\sqrt{1-\partial_{x_2}u(x)^2} \, dx 
    \\&\hspace{35pt} \leq   \left(\frac{1}{a}
    \int_0^a \int_{-L}^L 2F(\nabla u(x)) \, dx \right)^{1/2}\left(\frac{1}{a}
    \int_0^a \int_{-L}^L 1-\partial_{x_2}u(x)^2 \, dx \right)^{1/2}.
\end{align*}
Denoting by $E_0\in\R$ the constant total energy value given by the balance \eqref{eq:energy_balance} the first factor can be estimated against $\big(2E_0+2\norm{V}_{L^\infty(\Omega\times \R)}\big)^{1/2}$, while for the second factor we have
\begin{align*}
    \frac{1}{a}
    \int_0^a \int_{-L}^L 1-\partial_{x_2}u(x)^2 \, dx&=\frac{1}{a}\int_0^a\int_{-L}^L(\sign(x_2)-\partial_{x_2}u(x))(\sign(x_2)+\partial_{x_2}u(x))\:dx\\
    &\leq 
    2\norm{\partial_{x_2}v_a}_{L^1((0,1)\times(-L,L))}\rightarrow 0
\end{align*}
as $a\rightarrow 0$ by Lemma \ref{lem:tlim}.
\end{proof}

Comparing Lemmas \ref{lem:tlim}, \ref{lem:txlim} with \cite[Theorem 5.7]{evans_m} one can say that the functions 
\begin{gather*}
    x_1\mapsto\int_{-L}^L\abs{\partial_{x_2}u(x)-\sign(x_2)}\:dx_2,\quad x_1\mapsto\int_{-L}^L\abs{\partial_{x_1}u(x)}\:dx_2,
\end{gather*}
which are a priori only in $L^\infty(0,T)$, $L^2(0,T)$ resp., have a trace at $x_1=0$ and $x_1=T$.

% Note that the above lemmas imply existence of the respective traces
% $$\partial_{x_1} u=0\text{ when }t=0\text{ or }x\in\{\pm L\},\quad \partial_{x_2}u(0,x)=\text{sign}(x),\quad \partial_{x_2}u(T,x)=-\text{sign}(x),$$
% as integrable functions on the respective boundaries, for instance as done in the proof of Theorem 1 from Chapter 5.3 of \cite{evans_m}. 

\subsection{Admissibility}\label{sec_global_admissibility}

We will now discuss the actual energy balance of the Boussinesq subsolution induced by the minimizer $u$. Recall that the one-dimensional subsolution $(\rho,m):=(\partial_{x_2}u,-\partial_{x_1}u)$ is called weakly admissible provided \eqref{eq:subsolution_admissibility}, i.e.
\begin{equation}\label{eq:weak_admissibility_sec7}
E_{tot}(x_1):=\int_{-L}^LF(\nabla u(x))-gAu(x)\:dx_2<\int_{-L}^L\rho_0(x_2)gAx_2\:dx_2
\end{equation}
holds true for a.e. $x_1\in (0,T)$. 

As indicated in Section \ref{sec_energy_dissipation} we in fact will have a monotone decay of the total energy on $(0,T)$ provided $V$ satisfies besides \eqref{eq:condition_V1} also
\begin{gather}\tag{$\text{V}_{\text{dis}}$}\label{eq:V_dissipative} \begin{gathered} V\text{ has the form }
V(x,z)=-gAz+f(x_2,z) \\\text{ with }\partial_zf(x_2,z)>0\text{ whenever }\abs{z}<L-\abs{x_2}.
   \end{gathered}
\end{gather}
Note that \eqref{eq:V_dissipative} contains \eqref{eq:V1prime}.

As a direct consequence of Lemma \ref{lem:energy_balance} and the fact that our minimizer $u$ satisfies $\abs{x_2}-L\leq u(x)\leq L-\abs{x_2}$ for all $x\in \overline{\Omega}$, cf. Corollary \ref{cor:bounds_for_minimizer} and Proposition \ref{prop:u_continuous}, we indeed deduce the energy balance \eqref{eq:new_energy_balance} on the open interval $(0,T)$.
\begin{corollary}\label{cor:energy_dissipation}
If $V$ satisfies \eqref{eq:condition_V1} and \eqref{eq:V_dissipative}, then the sum of kinetic and potential energy (given only by the gravity potential) satisfies
\begin{align*}
    \frac{d}{dx_1}\int_{-L}^LF(\nabla u)-gAu\:dx_2=-\int_{-L}^L\partial_zf(x_2,u)\partial_{x_1}u\:dx_2\leq 0
\end{align*}
weakly on $(0,T)$.
\end{corollary}

Thus we have strict dissipation of the total energy on any time interval $I\subset(0,T)$ on which $((L-\abs{x_2})^2-u^2)\partial_{x_1}u$ is not essentially vanishing, i.e. $u$ has to be different from the initial and final configuration and the momentum $-\partial_{x_1}u$ has to be strictly negative.

Let us now turn to the behaviour of the energy as $x_1\rightarrow 0$ or $x_1\rightarrow T$.
Regarding the potential energy Proposition \ref{prop:u_continuous} implies that 
\begin{align*}
    \lim_{x_1\rightarrow 0}E_{pot}(x_1)&:=\lim_{x_1\rightarrow 0}\int_{-L}^L-gAu(x)\:dx_2=\int_{-L}^L-gA(\abs{x_2}-L)\:dx_2\\
    &\phantom{:}=\int_{-L}^L\rho_0(x_2)gAx_2\:dx_2
\end{align*}
and similarly, including the dissipated energy,
\begin{align*}
    \lim_{x_1\rightarrow 0}\int_{-L}^LV(x_2,u(x))\:dx_2=\int_{-L}^LV(x_2,\abs{x_2}-L)\:dx_2.
\end{align*}
The corresponding limits also exist at $x_1=T$.

Due to the continuity of the potential energy $E_{pot}(t)$ at $t=0$ one sees that weak admissibility \eqref{eq:weak_admissibility_sec7} requires 
\begin{align}\label{eq:limit_at_0_kinetic_energy}
    \esslim_{x_1\rightarrow 0}E_{kin}(x_1):=\esslim_{x_1\rightarrow 0}\int_{-L}^LF(\nabla u(x))\:dx_2=0.
\end{align}

Although the initial momentum vanishes in the sense of Lemma \ref{lem:txlim} we can a priori not conclude \eqref{eq:limit_at_0_kinetic_energy}. However, we will show that for suitable $V$ the possible initial jump in kinetic energy becomes arbitrarily small when the variational problem is considered over a longer and longer time interval. 

For that we consider the final time $T>0$ no longer as a fixed constant and indicate the $T$-dependency in our variational problem by writing $\Omega_T$, $\cA_T$, $X_T$ instead of $\Omega$, $\cA$, $X$. 

Moreover, as in Section \ref{sec:initial_final_energy} we define
    \begin{align*}
    s_V:=\sup\set{\int_{-L}^LV(x_2, \varphi)\:dx_2:\varphi\in \cC^0([-L,L]),~\abs{\varphi(x_2)}\leq L-\abs{x_2}}
\end{align*}
and assume in addition to \eqref{eq:condition_V1}, \eqref{eq:V_dissipative}
\begin{align}\label{eq:V_sup}\tag{$\text{V}_{\text{sup}}$}
    s_V=0\text{ and }s_V=\int_{-L}^LV(x_2,\varphi)\:dx_2\text{ if and only if }   \varphi=\pm(L-\abs{x_2}).
\end{align}
Note that assuming $s_V$ to be $0$ is not a restriction as one can always shift $V$ by a fixed constant without changing the variational problem.

\begin{lemma}\label{lem:long_time_limit}
Assume \eqref{eq:condition_V1}, \eqref{eq:V_dissipative}, \eqref{eq:V_sup}. Let $u_T$, $T>0$ be a minimizer of $\cA_T$ over $X_T$ given by Proposition \ref{prop:extended_gamma_convergence}. For $T\geq 1$ there holds
\begin{align*}
    \esslim_{x_1\rightarrow T}\int_{-L}^LF(\nabla u_T(x))\:dx_2=\esslim_{x_1\rightarrow 0}\int_{-L}^LF(\nabla u_T(x))\:dx_2\leq \frac{\cA_1(u_1)}{T}.
\end{align*}
\end{lemma}
\begin{proof}
In view of Lemma \ref{lem:energy_balance} there exists a constant $c_T\in\R$ such that 
\[
\int_{-L}^LF(\nabla u_T(x))+V(x_2,u_T(x))\:dx_2=c_T
\]
for a.e. $x_1\in (0,T)$.  By \eqref{eq:V_sup} we therefore have 
\begin{align*}
    \esslim_{x_1\rightarrow 0}\int_{-L}^LF(\nabla u_T(x))\:dx_2=c_T-\lim_{x_1\rightarrow 0}\int_{-L}^LV(x_2,u_T(x))\:dx_2=c_T.
\end{align*}
Similarly we also conclude that the essential limit of the kinetic energy as $x_1\rightarrow T$ is given by $c_T$. This shows the stated equality between the two limits. Note also that $c_T\geq 0$.

Next let $T'>T$ and define $v\in X_{T'}$ by setting 
\begin{align*}
    v(x)=\begin{cases}
    u_T(x),&x_1\in(0,T),\\
    L-\abs{x_2},&x_1\in(T,T').
    \end{cases}
\end{align*}
Due to \eqref{eq:V_sup} and $F(0,\pm 1)=0$ one deduces
\[
\cA_{T'}(u_{T'})\leq \cA_{T'}(v)=\cA_T(u_T).
\]
Thus \eqref{eq:V_sup} and Corollary \ref{cor:bounds_for_minimizer} imply 
\begin{align*}
    0\leq Tc_T\leq Tc_T-2\int_{\Omega_T}V(x_2,u_T(x))\:dx=\cA_T(u_T)\leq \cA_1(u_1)
\end{align*}
and the statement follows.
\end{proof}

\section{Summary and further questions}\label{sec_summary_questions}
Let us first of all repeat the full list of used requirements regarding the nonlinear potential $V(x,z)$ and formulate an extended version of Theorem \ref{thm:existence_and_part_regularity}. That is
\begin{gather}\tag{$\text{V}_{\text{reg}}$}\begin{gathered}
   \partial_z^k V:\overline{\Omega}\times\R\rightarrow\R,~k=0,1,2,3 \text{ exist, are Lipschitz and bounded,}
   \end{gathered}\\
\begin{gathered}\tag{$\text{V}_\text{aut}$}
    V(x,z)=V(x_2,z),~x\in\overline{\Omega},~z\in\R,
\end{gathered}\\
\begin{gathered}\tag{$\text{V}_\text{con}$}
    \partial_z^2V(x,z)\geq 0,~x\in\overline{\Omega},~z\in[-L-1,L+1],
\end{gathered}\\
\begin{gathered}\tag{$\text{V}_{\text{dis}}$}  V(x,z)=-gAz+f(x_2,z) \text{ with }\partial_zf(x_2,z)>0\text{ whenever }\abs{z}<L-\abs{x_2},
   \end{gathered}\\
\begin{gathered}\tag{$\text{V}_{\text{sup}}$}
    s_V=0\text{ and }s_V=\int_{-L}^LV(x_2,\varphi)\:dx_2\text{ if and only if }   \varphi=\pm(L-\abs{x_2}),
\end{gathered}
\end{gather}
where 
\begin{align*}
    s_V:=\sup\set{\int_{-L}^LV(x_2, \varphi)\:dx_2:\varphi\in \cC^0([-L,L]),~\abs{\varphi(x_2)}\leq L-\abs{x_2}}.
\end{align*}
We have seen, cf. Section \ref{sec:example_f}, that a suitable extension of
\begin{align*}
    V(x,z)=-gAz+\frac{3gA}{4L}(z-(\abs{x_2}-L))^2+const.
\end{align*}
satisfies all of the stated conditions. Summarizing the statements of Sections \ref{sec_regular_approximation}-\ref{sec_further_properties} there holds the following extension of Theorem \ref{thm:existence_and_part_regularity}.
\begin{theorem}\label{thm:full_detail}
Suppose that $V$ satisfies all five conditions \eqref{eq:condition_V1}-\eqref{eq:V_sup}. Then
Problem \eqref{eq:var_prop_sec_3} with $\cA$ defined in \eqref{eq:functional_in_sec_3} has a solution $u$ with the following properties
\begin{enumerate}[a)]
    \item $u$ is continuous on $\overline{\Omega}$ with $\abs{u(x)}\leq L-\abs{x_2}$ and $\partial_{x_1}u\geq 0$, $\abs{\partial_{x_2}u}\leq 1$ a.e.,
    \item there exists $\Omega^\prime\subset\Omega$ open, nonempty and such that every connected component of $\Omega'$ is simply connected, on which $u$ is of class $\cC^2$ with $\partial_{x_1}u>0$, $\abs{\partial_{x_2}u}<1$, while $\partial_{x_1}u(x)=0$ for a.e. $x\notin \Omega^\prime$,
    \item $\partial_{x_1}u(\cdot,\pm L)=\partial_{x_1}u(0,\cdot)=\partial_{x_1}u(T,\cdot)=0$, $\partial_{x_2}u(0,x_2)=-\partial_{x_2}u(T,x_2)=\sign(x_2)$ in the sense specified in Lemmas \ref{lem:tlim}, \ref{lem:xlim}, \ref{lem:txlim},
    \item on $(0,T)$ the balance 
    \begin{align*}
    \frac{d}{dx_1}\int_{-L}^LF(\nabla u)-gAu\:dx_2=-\int_{-L}^L\partial_zf(x_2,u)\partial_{x_1}u\:dx_2\leq 0
\end{align*}
holds in a weak sense,
\item while at $x_1=0$, $x_1=T$ there holds
\begin{align*}
    \esslim_{x_1\rightarrow T}\int_{-L}^LF(\nabla u(x))\:dx_2=\esslim_{x_1\rightarrow 0}\int_{-L}^LF(\nabla u(x))\:dx_2\leq \frac{c_1}{T}
\end{align*}
for all $T\geq 1$ and a constant $c_1>0$ (specified in Lemma \ref{lem:long_time_limit}).
\end{enumerate}
\end{theorem}

\subsection{Use of $\nabla u$ as a Boussinesq subsolution} Let $u$ be the minimizer from Theorem \ref{thm:full_detail} with partial regularity set $\Omega'$ and set $\rho,m_n:(0,T)\times(-L,L)\rightarrow\R$
\[
\rho(t,x_n)=\partial_{x_2}u(t,x_n),\quad m_n(t,x_n)=-\partial_{x_1}u(t,x_n).
\]
We indeed see that $\rho,m_n$ and $\mathscr{U}':=\Omega'$ satisfy Lemma \ref{lem:rise_of_1D_subsols} \eqref{eq:1D_pde}-\eqref{eq:1D_mixing_zone},\eqref{eq:1D_L1} while \eqref{eq:1D_outside_mixing_zone} is relaxed to $m_n=0$, $\abs{\rho}\leq 1$ a.e. outside $\Omega'$. In consequence $\nabla u$ induces a subsolution with mixed resting regions, cf. Remark \ref{rem:one_dim_subs_with_resting_regions}, and therefore via Theorem \ref{thm:CI}, Remark \ref{rem:convex_integration_theorem2} resp., infinitely many solutions  $(\rho_{sol},v_{sol})$ to the Boussinesq system \eqref{eq:bou},\eqref{eq:initial_data},\eqref{eq:boundary_condition} that are turbulently mixing on $\mathscr{U}'$, and of which the above density $\rho$, momentum $m=m_ne_n$ and velocity $v\equiv 0$ can be seen as horizontally averaged quantities.

At this point we have proven Theorem \ref{thm:introduction}. Let us however state some consequences of Theorem \ref{thm:full_detail} for the induced subsolution, solutions resp.. 
We conclude that
\begin{itemize}
    \item the average momentum $m$ is directed downwards ($m_n\leq 0$),
    \item outside the mixing zone $\mathscr{U}'$ the fluid is at rest ($v_{sol}=0$), but our investigation does not allow us to conclude that the density is in one of the two initial phases ($\abs{\rho_{sol}}<1$ not excluded),
    \item the resting regions, with or without $\rho_{sol}\in\set{\pm 1}$, can not be surrounded by the mixing zone $\mathscr{U}'$,
    \item besides the initial and boundary conditions for \eqref{eq:bou}, also $m$ vanishes in a certain trace sense as $t\rightarrow 0$, $t\rightarrow T$ and $\rho$ approaches the stable interface configuration $-\rho_0(x)$ as $t\rightarrow T$,
    \item the total energy of the subsolution 
    \[
    E_{sub}(t):=\int_{-L}^L\frac{m_n(t,x_n)^2}{2(1-\rho(t,x_n)^2)}+\rho(t,x_n)gAx_n\:dx_n
    \]
    might jump upwards from $\int_{-L}^L\rho_0(x_n)gAx_n\:dx_n$ at $t=0$, and then monotonically decays on $(0,T]$ to $-\int_{-L}^L\rho_0(x_n)gAx_n\:dx_n$ with a reversed jump at $t=T$,
    \item the height of the initial and final energy jump vanishes as the considered time interval $(0,T)$ becomes unbounded.
\end{itemize}
We recall that induced solutions can be found with total energy $E_{tot}(t)$ in an arbitrary $\delta(t)$-neighborhood of $E_{sub}(t)$, cf. Theorem \ref{thm:CI}.

\subsection{Open questions}
Here we discuss some further questions regarding the variational problem \eqref{eq:var_prop_sec_3}, properties of the induced subsolutions and the modeling in general.

Starting with the list of the previous subsection it would be of interest to see if, under suitable conditions on $V$, the possibility of mixed resting regions can be excluded, i.e. if in Theorem \ref{thm:full_detail} b) one could have $\partial_{x_1}u=0$ \emph{and} $\abs{\partial_{x_2}u}=1$ a.e. outside $\Omega'$. An analogue property for instance holds true in the setting of De Silva, Savin \cite{desilva_savin}, cf. Section \ref{sec_main_result} and Remark \ref{rem:obstacle_problem}.

Other questions for problem \eqref{eq:var_prop_sec_3} address uniqueness of minimizers, also this property holds true in \cite{desilva_savin}, as well as global regularity for instance comparable to the result of Colombo, Figalli \cite{colombo_figalli}, and any further information regarding the partial regularity set $\Omega'$ which corresponds to the turbulent mixing zone of the induced solutions. Of particular interest in applications is for instance the growth of this zone in time.

On a larger scale of questions we recall that our investigation has been motivated by the search of global in time selection criteria for subsolutions of the Euler equations. 
Here we first of all like to point out that the derivation of the non-dissipative action functional $\cA_0(u)$ in Sections \ref{sec_relaxation_boussinesq}-\ref{sec:reformulation} relies on almost no ansatzes besides the imposition of the least action principle itself. The only a priori unjustified choice made is that the kinetic energy density of the solutions $(\rho_{sol},v_{sol})$ in the turbulent zone $\mathscr{U}$ satisfies $\abs{v_{sol}}^2\in\cC^0(\mathscr{U})+\rho_{sol}\cC^0(\mathscr{U})$, cf. \eqref{eq:local_energy_density}. This choice has already been made in \cite{GKBou} and is, at least in the here considered one-dimensional initial configuration, a posteriori backed up by the fact that the functional $\cA_0$ can also be derived from a different point of view avoiding the notion of subsolutions at all, see Appendix \ref{sec_brenier}.

However, after the derivation of $\cA_0$ we introduced in Section \ref{sec_energy_dissipation} the nonlinear potential $V(x_2,z)=-gAu+f(x_2,z)$ that allowed us to haven energy dissipation (up to the initial jump controlled by $T^{-1}$) while staying within the variational framework of the least action principle. 
Note that, in terms of the subsolution components $\rho$ and $m_n$ the associated Euler-Lagrange equation (formal on all of $(0,T)\times(-L,L)$, rigorous on $\Omega'$) is given by
\begin{align}\label{eq:euler_lagrange_rho_m}
    \partial_t\left(\frac{m_n}{1-\rho^2}\right)-\partial_{x_n}\left(\frac{m_n^2\rho}{(1-\rho^2)^2}\right)-\Psi[\rho,m_n]=-gA,
\end{align}
where $\Psi$ is the nonlocal operator
\[
\Psi[\rho,m_n](t,x_n)=\partial_zf\left(x_n,\int_{-L}^{x_n}\rho(t,s)\:ds\right).
\]
Besides the here stated properties \eqref{eq:condition_V1}-\eqref{eq:V_sup} a further investigation, justification resp., concerning suitable choices of $f$, or more general, of a different type of relation $\Psi[\rho,m_n]$ consistent with energy dissipation is open.

\appendix

\section{Relation to Brenier's generalized least action principle}\label{sec_brenier}
We quickly recall the least action principle, Brenier's generalization of it, and thereafter focus on a special one-dimensional problem leading to a functional formally equivalent to our $\cA_0$ derived in Section \ref{sec_action_functional}.

\subsection{The least action principle}
Let $\cD\subset\R^n$ be a compact domain, $T>0$, $U:(0,T)\times \cD\rightarrow \R$ be a given potential and $\rho_0:\cD\rightarrow (0,\infty)$ an initial mass distribution. It is well known, originating in the work of Arnold \cite{arnold}, that the Euler equations
\begin{align}\label{eq:eulerbr}
    \begin{split}
        \partial_t (\rho v) +\divv (\rho v\otimes v) +\nabla p&=-\rho \nabla U,\\
\divv v&=0,\\
\partial_t \rho + \divv (\rho v)&=0
    \end{split}
\end{align}
can formally be derived by minimizing the action functional
\begin{align*}
    \cA(g)=\int_0^T\int_{\cD}\rho_0(x)\left(\frac{1}{2}\abs{\partial_tg(t,x)}^2-U(t,g(t,x))\right)\:dx\:dt
\end{align*}
over trajectories $t\mapsto g(t,\cdot)$ in the manifold of volume preserving diffeomorphisms $\cD\rightarrow \cD$ connecting a given initial and end state, say $g(0,\cdot)=\id$, $g(T,\cdot)=h$. Assuming the existence of a regular minimizer $g$ and an associated Lagrange multiplier $p:(0,T)\times\cD\rightarrow\R$ one derives that the tuple $(\rho,v,p)$, where $v$ is the velocity field inducing the flow $g$, i.e. $\partial_tg(t,x)=v(t,g(t,x))$, and $\rho$ is the corresponding transported density distribution, i.e. $\rho(t,g(t,x))=\rho_0(x)$,
is a solution of \eqref{eq:eulerbr} with initial mass distribution $\rho_0$. For more detail we refer to \cite{Brenier18} and the references therein. %\todo{other works \cite{Brenier89,Brenier99,lopesf_nlopes_precioso}, have a look at the latter 2, don't need to cite them necessarily}

Rigorous existence results of minimizers for suitable target diffeomorphisms $h$ not too far away from the identity are due to Ebin and Marsden \cite{ebin_marsden}. However, for a general $h$ it also has been shown by Shnirelman \cite{shnirelman} that there does not need to be a solution in the classical sense described above.

In order to overcome this, Brenier introduced in \cite{Brenier89} the aforementioned generalization of the least action principle, which allowed him to conclude the existence of a solution given the existence of at least one competitor with finite action.

\subsection{Relaxation via generalized flows}\label{sec_Brenier}

Let us recall that
Brenier's generalized action functional associated with \eqref{eq:eulerbr} is defined as
\begin{align}\label{eq:general_action_boussinesq}
    \cA(\mu):=\int_{\Omega(\cD)}\rho_0(\omega(0))\left(\int_0^T \frac{1}{2}|\omega'(t)|^2-U(t,\omega(t))\,dt\right)\,\mu(d\omega),
\end{align}
where $\Omega(\cD):=\{[0,T]\ni t\mapsto\omega(t)\in\cD\}=\cD^{[0,T]}$ is equipped with the product topology, hence compact, and $\mu$ is a generalized flow, namely a regular Borel probability measure on $\Omega(\cD)$ satisfying the incompressibility constraint
\begin{align}\label{eq:general_incompressibility_constraint}
    \int_{\Omega(\cD)}f(\omega(t))\,\mu(d\omega)=\fint_{\cD} f(x)\, dx,\ \forall f\in \cC(\cD),\ t\in[0,T],
\end{align}
as well as the initial and final data constraint
\begin{align}\label{eq:general_initial_end_data}
   \int_{\Omega(\cD)}f(\omega(0),\omega(T))\,\mu(d\omega)=\fint_{\cD} f(x,h(x))\, dx,\ \forall f\in \cC(\cD^2).
\end{align}
Here $h$ again denotes the target configuration, which in this setting is only required to be a measure preserving map $(\cD,\:dx)\rightarrow (\cD,\:dx)$, and not necessarily a diffeomorphism. %Note that $\cA$ can be seen as the difference of kinetic and potential energy associated with the system \eqref{eq:eulerbr}. Furthermore, setting $U:=gA x_n$ allows to recover the potential we are interested in in this paper.

It was shown in \cite{Brenier89} that if $\inf_{\mu} \cA(\mu)<+\infty$, then there exists a minimizer, and furthermore, if system \eqref{eq:eulerbr} has a solution (enjoying certain properties) then the minimizer corresponds to the flow associated with the fluid velocity of the solution, see \cite[Theorem 3.2, Theorem 5.1]{Brenier89} for the precise statements.

Through slight modifications one sees that the same two properties remain valid for the generalized action functional corresponding to the here investigated Euler system in Boussinesq approximation \eqref{eq:bou}, which reads
\begin{align}\label{eq:actbou}
    \cA(\mu):=\int_{\Omega(\cD)}\int_0^T\left( \frac{1}{2}|\omega'(t)|^2-\rho(0,\omega(0))\,gA\omega_n(t)\right)\,dt\,\mu(d\omega)
\end{align}
with $\rho(0,\cdot)=\sign$ being the normalized initial density distribution.

\subsection{One-dimensional two phase flows}

For the particular Rayleigh-Taylor situation we consider $\cD=[0,1]^{n-1}\times[-L,L]$, $\rho(0,x)=\sign (x_n)$ and the target transformation $h:\cD\rightarrow \cD$, $h(x',x_n)=(x',h_n(x_n))$, where 
$$
h_n(x_n)=\begin{cases}x_n-L,& x_n\geq 0,\\x_n+L,&x_n<0.\end{cases}
$$ 
I.e., $h$ swaps the upper half $\cD_+$ of the container with the lower half $\cD_-$. Note that in fact $h$ prescribes even a particle-by-particle exchange of the two halfs. 

This situation (without the first $n-1$ dimensions and without the potential) appears in Brenier's revisitation of the least action principle \cite{Brenier18} as one of the examples for generalized incompressible flows, see \cite[Section 4.3]{Brenier18}. 

Also here, i.e. with potential term, it follows formally (ignoring the conditions on $p$ and $T$) from \cite[Theorem 5.1]{Brenier89}  that the minimizer of \eqref{eq:actbou} is given by a one-dimensional two phase flow provided the associated vectorfields satisfy the corresponding Euler-Lagrange equation.

More precisely, a one-dimensional two-phase flow is a generalized flow of the type
$$
\mu(d\omega)=(\mu_+(0,x)\delta_{G_+(\cdot,x)}(\omega)+\mu_-(0,x)\delta_{G_-(\cdot,x)}(\omega))dx,
$$
where
\begin{itemize}
    \item the $\delta_{G_\pm(\cdot,x)}$ denote Dirac measures on $\Omega(\cD)$,
    \item $G_\pm(t,x)=(x',g_\pm(t,x_n))$ denote actual flows of two one-dimensional vectorfields $V_\pm(t,x)=(0,v_\pm(t,x_n))$, i.e. 
    \begin{gather*}
        \partial_t g_\pm(t,x_n)=v_\pm(t,g_\pm(t,x_n)),\\ g_\pm(0,x_n)=x_n,\quad g_\pm(T,x_n)=x_n\mp L.
    \end{gather*}
    The maps $g_\pm(t,\cdot)$ are understood as orientation preserving diffeomorphisms $\R\rightarrow \R$ with the property that 
    \begin{equation}\label{eq:two_phase_boundary_condition}
    g_+(t,[0,L])\subset [-L,L],\quad g_-(t,[-L,0])\subset [-L,L],\quad t\in[0,T].
    \end{equation}
\item the functions $\mu_{\pm}:[0,T]\times \cD\rightarrow\R$ indicate the two phases initially given by
\[
\mu_\pm(0,x)=|\cD|^{-1}\mathbbm{1}_{\cD_\pm}(x)=(2L)^{-1}\mathbbm{1}_{\cD_\pm}(x)
\]
and obeying
\begin{align}\label{eq:eulerbrg2}
\mu_++\mu_-=(2L)^{-1},\quad 
    \partial_t \mu_\pm +\partial_{x_n}(\mu_\pm v_\pm)=0.
    \end{align}
\end{itemize} 

Note that the continuity equations imply 
\begin{align}\label{eq:solution_of_continutiy_equations}
\mu_\pm(t,g_\pm(t,x_n))\partial_{x_n}g_\pm(t,x_n)=\mu_\pm(0,x_n).
\end{align}
One can then check using \eqref{eq:two_phase_boundary_condition}, \eqref{eq:eulerbrg2}, \eqref{eq:solution_of_continutiy_equations} resp., that such a two phase flow indeed satisfies the incompressibility constraint \eqref{eq:general_incompressibility_constraint}. Condition \eqref{eq:general_initial_end_data} is directly stated. Moreover, it follows from \eqref{eq:two_phase_boundary_condition}, \eqref{eq:eulerbrg2} that the average of the velocities satisfies
\begin{align}\label{eq:avg_velo_0}
    \mu_-v_-+\mu_+v_+=0\text{ on }[0,T]\times\cD .
\end{align}

As indicated above it follows formally from \cite[Theorem 5.1]{Brenier89} that such a two-phase flow is minimizing \eqref{eq:actbou} provided $v_\pm$ satisfy
 \begin{align}\begin{split}
 \label{eq:eulerbrg}
    \partial_tv_++\partial_{x_n}\left(\frac{1}{2}v_+^2+p\right)=-gA\quad &\text{for }x_n\in g_+(t,[0,L]),\\
    \partial_t v_-+\partial_{x_n}\left(\frac{1}{2}v_-^2+p\right)=gA\quad &\text{for }x_n\in g_-(t,[-L,0])
    \end{split}
\end{align}
with a pressure function $p:[0,T]\times [-L,L]\rightarrow \R$, $(t,x_n)\mapsto p(t,x_n)$ independent of the sign $\pm$.

The generalized action of a two phase flow with initial density $\rho(0,x)=\sign(x_n)$ transformed to Eulerian coordinates using \eqref{eq:eulerbrg2}, \eqref{eq:solution_of_continutiy_equations} reads
\begin{align}\label{eq:action_two_phase}
    \cA(\mu)=\int_0^T\int_{-L}^L\frac{1}{2}\left(\mu_+v_+^2+\mu_-v_-^2\right)-gA(\mu_+-\mu_-)x_n\:dx_n\:dt.
\end{align}

As a side remark we like to mention that as in \cite{Brenier18} the action \eqref{eq:action_two_phase} for a two phase flow can solely be written in terms of the flow of one of the phases, i.e.
\begin{align}\label{eq:br_actiong}\begin{split}
   \cA(\mu)= \frac{1}{2}\int_0^T\int_{0}^L   \partial_t g_+(t,x_n)^2&\left(1+ \frac{2L}{\partial_{x_n} g_+(t,x_n)-2L}\right)\:dx_n\:dt\\&\phantom{ajajajajaj}-2gA
    \int_0^T\int_{0}^L g_+(t,x_n) \, dx_n \, dt.
    \end{split}
\end{align}
The computations for the kinetic energy are not original to us, this is precisely Brenier's example in \cite[Section 4.3]{Brenier18}. The stated form of the potential energy easily follows from the incompressibility condition \eqref{eq:general_incompressibility_constraint} applied to the odd map $f(x)=x_n$. Also here one can check that condition \eqref{eq:eulerbrg} is precisely the (formal) Euler-Lagrange equation of \eqref{eq:br_actiong}.

\subsection{Comparison}
In order to see how the two phase action \eqref{eq:action_two_phase} corresponds to the functional \eqref{eq:action_after_derivation} derived for subsolutions we define $\rho,m:[0,T]\times[-L,L]\rightarrow\R$ by
\begin{align*}
    \rho&:=4L\mu_+-1=1-4L\mu_-,\\
    m&:=4L\mu_+v_+=-4L\mu_-v_-
\end{align*}
and observe that \eqref{eq:avg_velo_0} implies
\begin{align*}
    \mu_-v_-^2+\mu_+v_+^2&=-\frac{v_-v_+}{2L}=\frac{m^2}{2L(1-\rho^2)},\\
    \mu_+-\mu_-&=2\mu_+-\frac{1}{2L}=\frac{\rho}{2L}.
\end{align*}
Thus \eqref{eq:action_two_phase} becomes 
\[
\cA(\mu)=\frac{1}{2L}\int_0^T\int_{-L}^L\frac{m^2}{2(1-\rho^2)}-\rho gAx_n\:dx_n\:dt.
\]
Finally observe that the tuple $(\rho,m)$ satisfies 
\begin{align*}
    \partial_t\rho+\partial_{x_n}m=4L(\partial_t\mu_++\partial_{x_n}(\mu_+v_+))=0
\end{align*}
by \eqref{eq:eulerbrg2}. Thus we arrive at \eqref{eq:action_after_derivation} and Lemma \ref{lem:rise_of_1D_subsols} \eqref{eq:1D_pde}.

Regarding Lemma \ref{lem:rise_of_1D_subsols} \eqref{eq:1D_boundary_initial_data} and \eqref{eq:final_configuration} we like to point out that the conditions for initial and final data for $\rho$ are build into the two-phase flow framework by specifying $\mu_\pm(0,\cdot)$ and the target diffeomorphism $h$. As mentioned earlier the specification of $h$ is even stronger than requiring $\rho(T,\cdot)=-\rho_0$ via \eqref{eq:final_configuration} as it corresponds to a particle-by-particle exchange of the two fluids. 

Next we will convince ourselves that $m(t,\pm L)=0$ holds true as well. Indeed a generalized flow $\mu(d\omega)$ has to be a measure on the path space $\Omega(\cD)$. In terms of a two-phase flow this is ensured by \eqref{eq:two_phase_boundary_condition}. Now if $g_+(t,x_n)=L$ for some $x_n\in[0,L]$, then $\partial_tg_+(t,x_n)\leq 0$, and thus $v_+(t,L)\leq 0$. On the other hand if $L\notin g_+(t,[0,L])$, then $\mu_+(t,L)=0$. We therefore conclude 
$m(t,L)=4L\mu_+(t,L)v_+(t,L)\leq 0,$
and similarly $m(t,-L)\geq 0$. Now by means of \eqref{eq:solution_of_continutiy_equations} we compute
\[
\int_\cD\rho(t,x)\:dx=4L\int_\cD\mu_+(t,x)\:dx-2L=0
\]
and therefore using \eqref{eq:eulerbrg2} it follows that
\[
0=\frac{d}{dt}\int_{\cD}\rho(t,x)\:dx=-(m(t,L)-m(t,-L))\geq 0.
\]
Thus $m(t,\pm L)=0$.

Following from now on the same reformulation as in Section \ref{sec:reformulation} we therefore have shown that the variational problem \eqref{eq:variational_problem_in_introduction} considered in this article  can also be derived from Brenier's generalization of the least action principle instead of subsolutions. The relations are summarized in Figure \ref{fig:diagram}.

\begin{figure}
    \centering

\begin{tikzpicture}[font=\small,thick]
 
\begin{scope}[node distance=20mm and 2mm]
 
% Conditions test
\node[draw,
    %diamond,
    %below=of block3,
    minimum width=2.5cm,
    minimum height=2cm,
    inner sep=0] (block4){ 
    \begin{tabular}{c}
    variational problem\\
   \phantom{1111111111} $\min_{u\in X}\cA_0(u)$ \phantom{1111111111}
\end{tabular}};
 
\node[draw,
    %diamond,
   above left=of block4,
    minimum width=3cm,
    minimum height=1cm,
    inner sep=0] (block5) { Generalized L.A.P.};
 
\node[draw,
    %diamond,
  above right=of block4,
    minimum width=3.5cm,
    minimum height=1cm,
    inner sep=0] (block6) { Subsolutions from \cite{GKBou}};
 
% Increase and Decrease duty cycle
\node[draw,
    above=of block5,
    minimum width=2.5cm,
    minimum height=1cm] (block7) { L.A.P.};

\node[draw,
    above=of block6,
    minimum width=2.5cm,
    minimum height=1.5cm] (block9) { 
    \begin{tabular}{c}
    Euler equations\\
    (in Boussinesq approx.)
\end{tabular}};

 \end{scope}

% Arrows

\draw[-latex] (block5) edge node[pos=0.4,fill=white,inner sep=2pt]{\begin{tabular}{c}
    restricted to 1D two-phase flows\\ with
    $h$ as in \cite[Section 4.3]{Brenier18}
\end{tabular} }(block4);
 
\draw[-latex] (block6) edge node[pos=0.4,fill=white,inner sep=2pt]{\begin{tabular}{c}
    applying L.A.P. to\\
    1D subsolutions
\end{tabular}  }(block4);
 
\draw[-latex] (block7) edge node[pos=0.4,fill=white,inner sep=2pt]{Brenier \cite{Brenier89}}(block5);

\draw[-latex] (block9) edge node[pos=0.4,fill=white,inner sep=0]{\begin{tabular}{c}
    relaxation as\\
    differential inclusion
\end{tabular} }(block6);
   
 %\draw[latex-] (block9) edge node[pos=0.6,fill=white,inner sep=0]{Convex Integration}(block4);
\draw[-latex] (block4) |- (block9)
    node[pos=0.25,fill=white,inner sep=0]{\begin{tabular}{c}
    adding dissipation mechanism\\
    + convex integration for minimizer
\end{tabular} };
    
  \draw[-latex] (block7) .. controls +(0,2) and +(0,2) .. node[pos=0.5,fill=white,inner sep=0]{\begin{tabular}{c}
    Arnold\\
    (+ Boussinesq approx.)
\end{tabular}}(block9);

\end{tikzpicture}
\caption{Relation between relaxations of the least action principle and Euler equations.}
    \label{fig:diagram}
\end{figure}
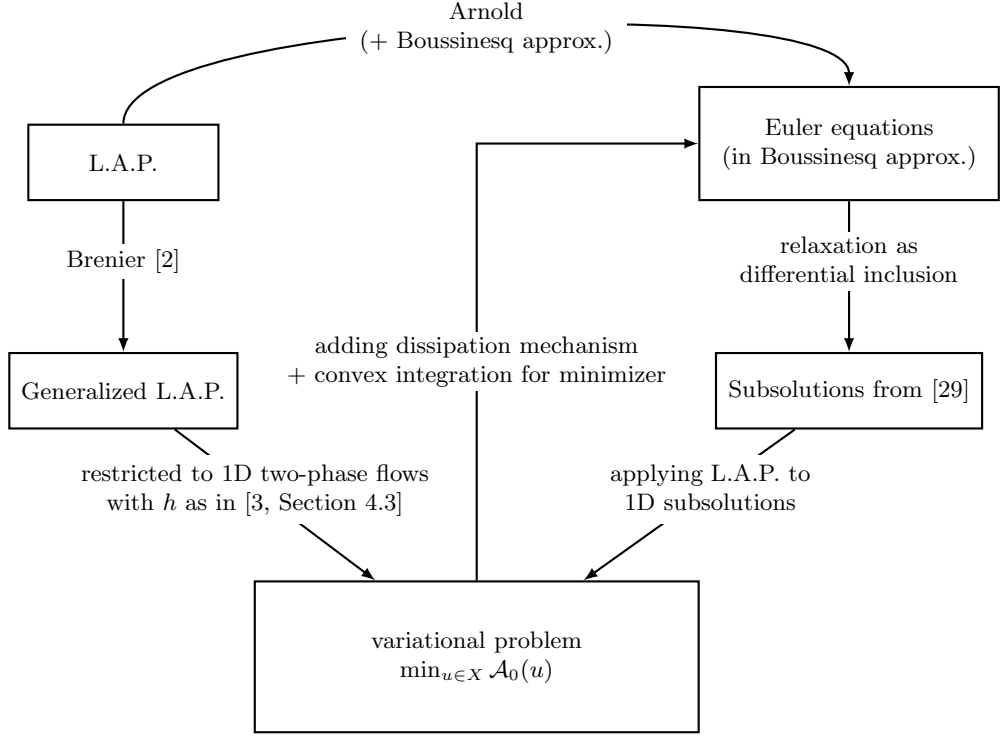

\section{Regarding the convex integration}\label{sec_convex_integration}
%\todo{convex integration ala Tartar but on growing sets}

In this section we prove Theorem \ref{thm:CI}. Let $z_{sub}$ be a subsolution with respect to $e_0,e_1$ and with mixing zone $\mathscr{U}$ and $\delta:[0,T]\rightarrow \R$ be continuous with $\delta(0)=0$, $\delta(t)>0$, for $t>0$. We also define the set of functions
\[
\cF:=\set{\frac{n}{2}e_1,(t,x)\mapsto gAx_n}
\]
and take open sets $V_j\subset\subset V_{j+1}\subset \mathscr U$ with $\bigcup_{j\geq 1}V_j=\mathscr U$, $|\partial V_j|=0$ for $j\geq 1$. We convex integrate recursively as follows.

\textbf{Step 1.} Initiation: Let $X_0^1$ be the set of tuples $z=(\rho,v,m,\sigma,p)$ satisfying
\begin{itemize}
    \item $(\rho,v,m,\sigma)\in (L^\infty\times L^2\times L^2 \times L^1)((0,T)\times\mathcal D)$,
    \item $p$ is a distribution on $(0,T)\times \cD$,
    \item $z$  satisfies the linear system \eqref{eq:lin_syst} with \eqref{eq:initial_data},\eqref{eq:lin_syst_boundary_condition},
    \item $(\rho,v,m,\sigma)$ is continuous on $V_1$ and $z(t,x)\in U_{(t,x)}$ for $(t,x)\in V_1$,
    \item $z=z_{sub}$  a.e. in $(0,T)\times\mathcal D\setminus V_1$,
    \item $\exists C(z)\in(0,1)$ with \[
    \left|\int_{\mathcal D} f(\rho-\rho_{sub})\, dx \right| \leq \frac{C(z)}{2}\delta(t)
    \] for all $t\in[0,T]$ and all $f\in\cF$.
\end{itemize}
Since $e_0$ and $e_1$ are continuous on $\overline{V_1}$, it follows that the set
$$\tilde{X}_0^1:=\set{(\rho,v|_{V_1},m|_{V_1},\sigma|_{V_1}):\ z\in X_0^1\text{ for some }p}$$
is bounded in $L^\infty((0,T);L^2(\mathcal D))\times L^2(V_1;\mathbb R^n\times \mathbb R^n\times\mathcal S_0^{n\times n})$. Moreover, as in \cite[Remark 2.4]{GKBou} it follows that for any such $z$ there holds $\rho\in \cC^0([0,T];L^2_{w}(\mathcal D))$.

Now let $B^{(1)}\subset L^2(\mathcal D)$, $B^{(2)}\subset L^2(V_1;\mathbb R^n\times \mathbb R^n\times\mathcal S_0^{n\times n})$ be bounded balls such that
$$\rho(t,\cdot)\in B^{(1)},\text{ for all }t\in[0,T],\quad (v|_{V_1},m|_{V_1},\sigma|_{V_1})\in  B^{(2)},$$
for any $(\rho,v|_{V_1},m|_{V_1},\sigma|_{V_1})\in \tilde{X}_0^1$. Furthermore let us denote by $d^{(1)}$ and $d^{(2)}$ the metrizations of the respective weak-$L^2$ topologies on these balls, and set
\begin{align*}
    d_{X^1}(y,y'):=\max\Big\{\sup_{t\in[0,T]}d^{(1)}(\rho(t,\cdot)&,\rho'(t,\cdot)),\\&
    d^{(2)}((v|_{V_1},m|_{V_1},\sigma|_{V_1}),(v'|_{V_1},m'|_{V_1},\sigma'|_{V_1}))\Big\},
\end{align*}
for $y,y'\in \tilde{X}_0^1$. We define $X^1$ as the closure of $\tilde{X}_0^1$ with respect to $d_{X^1}$, such that  $(X^1, d_{X^1})$ is a complete metric space.

Proceeding as in \cite{GKBou} one obtains that the
functional
\begin{align*}
    I_1(y):=\int_{V_1} |y(t,x)|^2 d(t,x)
\end{align*}
is Baire-1 on $X^1$ and that
\begin{align*}
    J_1(y):=\int_{V_1} d(y(t,x),\pi(K_{(t,x)})) d(t,x)
\end{align*}
is continuous with respect to the strong-$L^2$ topology on $X^1$. Here $\pi$ is the canonical projection from $\mathbb R\times\mathbb R^n\times \mathbb R^n\times\mathcal S_0^{n\times n}\times \mathbb R$ to $\mathbb R\times\mathbb R^n\times \mathbb R^n\times\mathcal S_0^{n\times n}$ eliminating the pressure component.

One may then use an argument based on a perturbation lemma (see for instance \cite[Lemma 3.13]{GKBou}) to show that $J_1^{-1}(0)$ is residual in $X^1$. Hence for every $\varepsilon_1>0$ we find $y_1\in J_1^{-1}(0)$, which after augmentation by a suitable $p^{(1)}$ gives a subsolution $z_{sub}^{(1)}$, $\varepsilon_1$-close to $(\rho_{sub},v_{sub},m_{sub},\sigma_{sub})$. More precisely, there holds
\begin{itemize}
    \item $z_{sub}^{(1)}$ solves the linear system \eqref{eq:lin_syst},\eqref{eq:initial_data},\eqref{eq:lin_syst_boundary_condition},
    \item $z_{sub}^{(1)}=z_{sub}$ outside $V_1$,
    \item $z_{sub}^{(1)}(t,x)\in K_{(t,x)}$ for a.e. $(t,x)\in V_1$,
    \item $d_{X^1}\left((\rho^{(1)}_{sub},v^{(1)}_{sub},m^{(1)}_{sub},\sigma^{(1)}_{sub}),(\rho_{sub},v_{sub},m_{sub},\sigma_{sub})\right)<\varepsilon_1$,
    \item $\displaystyle \left|\int_{\mathcal D}f \left(\rho_{sub}^{(1)}-\rho_{sub}\right)\, dx \right| \leq \frac{1}{2}\delta(t)$ for all $t\in[0,T]$, $f\in\cF$.
\end{itemize}
I.e., the mixing zone of $z_{sub}^{(1)}$ is given by $\mathscr U\setminus\overline{V_1}$.

\textbf{Step 2.} Recursion: Suppose that $j\geq 1$ and that there exists a 
subsolution $z_{sub}^{(j)}$ with mixing zone $\mathscr U\setminus \overline{V_j},$ i.e.
\begin{itemize}
    \item $z_{sub}^{(j)}$ solves the linear system \eqref{eq:lin_syst},\eqref{eq:initial_data},\eqref{eq:lin_syst_boundary_condition},
    \item $z_{sub}^{(j)}=z_{sub}$ outside $V_j$,
    \item $z_{sub}^{(j)}\in K_{(t,x)}$ for a.e. $(t,x)\in V_j$,
    \item $d_{X^j}\left((\rho^{(j)}_{sub},v^{(j)}_{sub},m^{(j)}_{sub},\sigma^{(j)}_{sub}),(\rho^{(j-1)}_{sub},v^{(j-1)}_{sub},m^{(j-1)}_{sub},\sigma^{(j-1)}_{sub})\right)<\varepsilon_j$,
    \item $\displaystyle \left|\int_{\mathcal D}f \left(\rho_{sub}^{(j)}-\rho_{sub}^{(j-1)}\right)\, dx \right| \leq \frac{1}{2^j}\delta(t)$ for all $t\in[0,T]$, $f\in\cF$.
\end{itemize}

Here the spaces $(X^j,d_{X^j})$ are defined recursively by first saying that a tuple  $z=(\rho,v,m,\sigma,p)$ belongs to $X_0^{j+1}$ if and only if
\begin{itemize}
    \item $(\rho,v,m,\sigma)\in (L^\infty\times L^2\times L^2 \times L^1)((0,T)\times\mathcal D)$,
    \item $p$ is a distribution on $(0,T)\times \cD$,
    \item $z$  satisfies the linear system \eqref{eq:lin_syst} with \eqref{eq:initial_data},\eqref{eq:lin_syst_boundary_condition},
    \item $(\rho,v,m,\sigma)$ is continuous on $W_{j+1}:=V_{j+1}\setminus \overline{V_j}$ and $z(t,x)\in U_{(t,x)}$ for $(t,x)\in W_{j+1}$,
    \item $z=z^{(j)}_{sub}$  a.e. in $(0,T)\times\mathcal D\setminus W_{j+1}$,
    \item $\exists C(z)\in(0,1)$ with \[
    \left|\int_{\mathcal D} f\left(\rho-\rho^{(j)}_{sub}\right)\, dx \right| \leq \frac{C(z)}{2^{j+1}}\delta(t)
    \] for all $t\in[0,T]$ and all $f\in\cF$,
\end{itemize}
and then constructing an appropriate completion $(X^{j+1},d_{X^{j+1}})$ as in Step 1. Note that the whole space $X^{j+1}$ depends on the previously chosen $j$th-order subsolution $z^{(j)}_{sub}$. 

As in Step 1, relying on the functionals $I_{j+1},J_{j+1}:X^{j+1}\rightarrow \R$, 
\begin{align*}
    I_{j+1}(y):=\int_{W_{j+1}} |y(t,x)|^2 d(t,x),\quad J_{j+1}(y):=\int_{W_{j+1}} d(y(t,x),\pi(K_{(t,x)})) d(t,x)
\end{align*}
we may conclude the existence of subsolutions $z_{sub}^{(j+1)}$ satisfying the properties listed in the beginning of Step 2 with $j$ replaced by $j+1$ for any given $\varepsilon_{j+1}>0$.

\textbf{Step 3.} Conclusion: In this manner we may construct (infinitely many) sequences $\{z_{sub}^{(j)}\}_{j\geq 1}$ which further satisfy 
\begin{align}\label{eq:nested_subsolutions}
    z_{sub}^{(k)}|_{V_j}=z_{sub}^{(j)}|_{V_j},\quad \text{for }k\geq j\geq 1.
\end{align}
For simplicity of notation we also set $z_{sub}^{(0)}:=z_{sub}$.
Let us show that any such a sequence converges to a solution.

First of all we claim that there exists a dimensional constant $C=C(n)>0$ such that 
\begin{align}\label{eq:bounds_in_U}
    \abs{m^{(j)}_{sub}}^2+\abs{v^{(j)}_{sub}}^2+\abs{\sigma^{(j)}_{sub}}\leq C\left(e_0+\rho_{sub}^{(j)}e_1\right)\quad \text{a.e. on }(0,T)\times\cD.
\end{align}
Indeed, outside the corresponding mixing zone this is clear by the definition of $K_{(t,x)}$, whereas inside the mixing zone the inequalities involved in the definition of $U_{(t,x)}$ imply
\begin{align*}
    \abs{m}^2+\abs{v}^2&=\frac{\abs{m+v}^2}{2}+\frac{\abs{m-v}^2}{2}<\frac{\abs{m+v}^2}{1+\rho}+\frac{\abs{m-v}^2}{1-\rho}\\
    &<n(e_0+e_1)(1+\rho)+n(e_0-e_1)(1-\rho)=2n(e_0+\rho e_1),
\end{align*}
as well as 
\begin{align*}
    -\lamin(\sigma)&=\lamax(-\sigma)<e_0+\rho e_1+\lamax\left(-\frac{v\otimes v-\rho (v\otimes m+m\otimes v)+m\otimes m}{1-\rho^2}\right)\\
    &=e_0+\rho e_1+\lamax\left(-\frac{(v+m)\otimes(v+m)}{2(1+\rho)}-\frac{(v-m)\otimes(v-m)}{2(1-\rho)}\right)\\
    &\leq e_0+\rho e_1.
\end{align*}
Note that the bound on $-\lamin(\sigma)=\abs{\lamin(\sigma)}$ is enough since $\sigma$ is trace free.

% Thus using $f=\frac{n}{2}e_1\in\cF$ we compute
% \begin{align*}
%     \int_0^T\int_{\mathcal D}|&e_0+\rho_{sub}^{(j)} e_1| \, dx \, dt=
%     \int_0^T\int_{\mathcal D}e_0+\rho_{sub}^{(j)} e_1 \, dx \, dt\\
%     &=
%      \int_0^T\int_{\mathcal D}e_0+\rho_{sub}e_1 \, dx \, dt+\frac{2}{n}\int_0^T\int_\cD f\left(\rho_{sub}^{(j)}-\rho_{sub}\right)\:dx\:dt
%      \\&\leq \norm{e_0+\rho_{sub}e_1}_{L^1((0,T)\times \cD)} + \sum_{k=0}^{j-1}\frac{2}{n}\int_0^T\left|\int_{\mathcal D}f\left(\rho_{sub}^{(k+1)}-\rho_{sub}^{(k)}\right) \, dx\right|\, dt\\
%      &\leq \norm{e_0+\rho_{sub}e_1}_{L^1((0,T)\times \cD)}+ \frac{2}{n}\sum_{k=0}^{j-1}2^{-(k+1)}\int_0^T\delta(t) \, dt,
% \end{align*}
% which is bounded since $e_0+\rho_{sub}e_1\in L^1$ by assumption and $\delta \in \cC^0([0,T])$. Hence we may conclude that $\left(\rho_{sub}^{(j)},v_{sub}^{(j)},m_{sub}^{(j)}\right)_{j\geq 1}$ converges weakly in $L^2((0,T)\times \cD)$. 

%For $\left\{\sigma_{sub}^{(j)}\right\}_{j\geq 1}$ we only have an $L^1$ bound, so we shall prove that it is also equi-integrable. We have the following lemma.
As a consequence we can conclude the following equi-integrability.
\begin{lemma} For any $\varepsilon>0$ there exists $\delta>0$ such that for any measurable set $S\subset(0,T)\times\mathcal D$ with $|S|\leq \delta$ there holds
$$\int_S\abs{\rho_{sub}^{(j)}}+\abs{m^{(j)}_{sub}}^2+\abs{v^{(j)}_{sub}}^2+\abs{\sigma_{sub}^{(j)}} \, dx \, dt\leq\varepsilon,\text{ for all }j\geq 1.$$
\end{lemma}
\begin{proof} The equi-integrability of the $\rho$-component is clear due to the uniform $L^\infty$-bound. 

Let us now show the equi-integrability of the $\sigma$-component. Let $\varepsilon>0$ and $j_0\geq 1$ such that 
\begin{align}\label{eq:epsconds}
    2^{-j_0+1}\int_0^T \delta(t)\, dt\leq \varepsilon\quad \text{ and }\quad\int_{\mathscr U \setminus V_{j_0}}e_0+\rho_{sub}e_1 \, dx \, dt\leq \varepsilon.
\end{align}
Then there exists $\delta>0$ such that for any measurable set $S\subset(0,T)\times\mathcal D$ with $|S|\leq \delta$ there holds
\begin{align*}
    \int_S|\sigma_{sub}| \, dx \, dt\leq\varepsilon\quad\text{ and }\quad\int_S\left|\sigma_{sub}^{(j)}\right| \, dx \, dt\leq\varepsilon \text{ for all }j\leq j_0,
\end{align*}
due to the equi-integrability of a finite family of integrable functions.

For $j>j_0$ using $f=\frac{n}{2}e_1\in\cF$ and \eqref{eq:bounds_in_U}, \eqref{eq:epsconds} we estimate
\begin{align*}
     \int_S\abs{\sigma_{sub}^{(j)}}& \, dx \, dt= \int_{S\setminus V_j}|\sigma_{sub}| \, dx \, dt+\int_{S\cap (V_j\setminus V_{j_0})}\left|\sigma_{sub}^{(j)}\right| \, dx \, dt+\int_{S\cap V_{j_0}}\left|\sigma_{sub}^{(j_0)}\right| \, dx \, dt \\
     &\leq 2\varepsilon + \int_{\mathscr U\setminus V_{j_0}}\left|\sigma_{sub}^{(j)}\right| \, dx \, dt \leq 2\varepsilon + C
     \int_{\mathscr U\setminus V_{j_0}}e_0+\rho_{sub}^{(j)}e_1 \, dx \, dt \\
     &=  
     2\varepsilon + C \left(
     \int_{\mathscr U\setminus V_{j_0}}e_0+\rho_{sub}e_1 \, dx \, dt 
     +
     \frac{2}{n}\int_0^T\int_\cD f\left(\rho_{sub}^{(j)}-\rho_{sub}\right) \, dx \, dt
     \right)
     \\& \leq 2\varepsilon + C\varepsilon + \frac{2C}{n} \sum_{k=j_0}^{j-1}\int_0^T\left|\int_{\mathcal D}f\left(\rho_{sub}^{(k+1)}-\rho_{sub}^{(k)}\right) \, dx\right|\, dt\\
    &\leq 2\varepsilon + C\varepsilon+ \frac{2C}{n}\sum_{k=j_0}^{+\infty}2^{-(k+1)}\int_0^T\delta(t) \, dt\leq \tilde{C}(n) \varepsilon.
\end{align*}
The proof for the $m$- and $v$-component is the same.
\end{proof}

The previous lemma and \eqref{eq:nested_subsolutions} allow to check that $\left(\rho_{sub}^{(j)},v_{sub}^{(j)},m_{sub}^{(j)},\sigma_{sub}^{(j)}\right)_j$ is a Cauchy sequence in $L^1\times L^2\times L^2\times L^1$ and thus converging to a limit $(\rho,v,m,\sigma)$ solving the linear system \eqref{eq:lin_syst}, \eqref{eq:initial_data}, \eqref{eq:lin_syst_boundary_condition} with some distributional pressure $p$. 

Furthermore, outside $\mathscr U$ we have $z(t,x)=z_{sub}(t,x)\in K_{(t,x)}$ a.e., while for every $(t,x)\in\mathscr U$ there exists $j\geq 1$ such that $(t,x)\in V_j$, and hence $z(t,x)=z_{sub}^{(j)}(t,x)\in K_{(t,x)}$ a.e. in $\mathscr U$. So $z$ is a solution to the Boussinesq system with $\abs{\rho}=1$ a.e. and its energy density is given by
$$
\mathcal E(t,x)=\frac{n}{2}(e_0(t,x)+\rho(t,x) e_1(t,x))+\rho(t,x) gA x_n.
$$

Next we will show \eqref{eq:error_in_energies}, i.e. we will show that
$$F(t):=\int_{\mathcal D} f(\rho-\rho_{sub}) \, dx$$
first of all is well-defined and moreover satisfies $|F(t)|\leq \delta(t)$ for a.e. $t\in(0,T)$ and every $f\in\cF$. In order to do this we define
$$ F_j(t):=\int_{\mathcal D} f\left(\rho_{sub}^{(j)}-\rho_{sub}\right) \, dx.$$
Since  $\rho_{sub},\rho_{sub}^{(j)}\in \cC^0([0,T];L^2_{w}(\mathcal D))$,
and we have that $\rho_{sub}^{(j)}-\rho_{sub}$ is supported in $V_j$, where $f$ is continuous, 
it follows that each $F_j$ is continuous.
On the other hand, for $j>j'\geq 1$ we may estimate 
\begin{align*}
    \left| F_j(t)- F_{j'}(t) \right| \leq \sum_{k=j'}^{j-1}\left|\int_{\mathcal D} f\left(\rho_{sub}^{(k+1)}-\rho_{sub}^{(k)}\right) \, dx \right|  \leq \sum_{k\geq j'}2^{-(k+1)}\delta(t),
\end{align*}
hence $\{ F_j\}_{j\geq 1}$ is Cauchy in $\cC^0([0,T])$, since $\delta$ was assumed to be continuous. Thus, $ F_j \to \tilde{F}$ uniformly for some $\tilde{F}$, which satisfies
$\abs{\tilde{F}(t)}\leq \delta(t)$ for all $t\in[0,T]$.

In order to conclude $\tilde{F}(t)=F(t)$ for a.e. $t\in(0,T)$ we will show that $f(\rho-\rho_{sub})\in L^1((0,T)\times\cD)$. This is of course clear in the case $f=gAx_n$, but not for $f=\frac{n}{2}e_1$, as $e_1$ might be not integrable. Nonetheless the claimed integrability for $f(\rho-\rho_{sub})$ follows from the following lemma and monotone convergence.

\begin{lemma} There holds
\[
\sup_{j}\int_0^T\int_\cD\abs{e_1\left(\rho_{sub}^{(j)}-\rho_{sub}\right)}\:dx\:dt<\infty.
\]
\end{lemma}
\begin{proof}
We abbreviate $\rho^j:=\rho^{(j)}_{sub}$ and $\bar{\rho}:=\rho_{sub}$. First of all note that if $\abs{\rho^j}=1$, i.e. a.e. on $V_j$, there holds 
\begin{align*}
    \abs{\rho^j-\bar{\rho}}\leq 2(1-\bar{\rho}^2)\quad \text{if and only if} \quad \rho^j\bar{\rho}\geq -\frac{1}{2}.
\end{align*}
We therefore estimate
\begin{align*}
    \int_0^T\int_\cD &\abs{e_1}\abs{\rho^j-\bar{\rho}}\:dx\:dt=\int_{V_j}\abs{e_1}\abs{\rho^j-\bar{\rho}}\:dx\:dt\\
    &\leq \int_{V_j\cap \set{\rho^j\bar{\rho}<-\frac{1}{2}}}\abs{e_1}\abs{\rho^j-\bar{\rho}}\:dx\:dt+2\int_{V_j\cap \set{\rho^j\bar{\rho}\geq-\frac{1}{2}}}\abs{e_1}(1-\bar{\rho}^2)\:dx\:dt.
\end{align*}
Now the second term is bounded by $2\norm{e_1(1-\bar{\rho}^2)}_{L^1((0,T)\times\cD)}$ which is finite by assumption \eqref{eq:technical_conditions_for_convex_integration_1}.

Regarding the first term observe that if $\rho^j\bar{\rho}<0$, then $e_1(\rho^j-\bar{\rho})\geq 0$ by condition \eqref{eq:technical_conditions_for_convex_integration_2}. We therefore have 
\begin{align*}
    \int_{V_j\cap \set{\rho^j\bar{\rho}<-\frac{1}{2}}}&\abs{e_1}\abs{\rho^j-\bar{\rho}}\:dx\:dt=\int_{\set{\rho^j\bar{\rho}<-\frac{1}{2}}}e_1\left(\rho^j-\bar{\rho}\right)\:dx\:dt\\
    &=\int_0^T\int_\cD e_1\left(\rho^j-\bar{\rho}\right)\:dx\:dt-\int_{\set{\rho^j\bar{\rho}\geq -\frac{1}{2}}}e_1\left(\rho^j-\bar{\rho}\right)\:dx\:dt\\
    &\leq \int_0^T\delta(t)\:dt+2\norm{e_1(1-\bar{\rho}^2)}_{L^1((0,T)\times \cD)}.
\end{align*}
This finishes the proof of the lemma.
\end{proof}

In consequence $F(t)$ is well-defined for a.e. $t\in(0,T)$ and using dominated convergence, one obtains that for any $\varphi\in L^\infty(0,T)$ there holds
\begin{align*}
    \int_0^T \varphi(t)\tilde{F}(t) \, dt&=\lim_{j\to+\infty}\int_0^T \int_{\mathcal D} \varphi(t) f\left(\rho_{sub}^{(j)}-\rho_{sub}\right) \, dx \, dt\\&=\lim_{j\to+\infty}\int_{V_j} \varphi(t) f(\rho-\rho_{sub}) \, dx \, dt=\int_0^T \varphi(t) F(t) \, dt.
\end{align*}
Thus $\tilde F=F$ a.e. and therefore $\abs{F(t)}\leq \delta(t)$ for a.e. $t\in(0,T)$.

It remains to show the existence of a sequence of such solutions $(\rho_k,v_k)_k$ converging to $(\rho_{sub},v_{sub})$ weakly in $L^2((0,T)\times \cD)$. Indeed Step 2 and what we have shown so far in Step 3 allows to have solutions $(\rho_k,v_k,m_k,\sigma_k)$, $k\geq 1$ that are generated by sequences of subsolutions $\left(\rho^{(j)}_{sub,k},v^{(j)}_{sub,k},m^{(j)}_{sub,k},\sigma^{(j)}_{sub,k}\right)_j$ such that for each fixed $j\geq 1$ there holds 
\[
\left(\rho^{(j)}_{sub,k},v^{(j)}_{sub,k},m^{(j)}_{sub,k},\sigma^{(j)}_{sub,k}\right)\rightharpoonup (\rho_{sub},v_{sub},m_{sub},\sigma_{sub})
\]
weakly in $L^2((0,T)\times\cD)$.

For the $\rho$- and $v$-components this weak convergence extends as follows.
Let $\varepsilon>0$, $\varphi\in L^2((0,T)\times\cD)$ and pick a fixed $j\geq 1$ such that $\int_{\mathscr{U}\setminus V_j}\abs{\varphi}<\varepsilon$. Then 
\begin{align*}
    \int_0^T\int_\cD \varphi(\rho_k-\rho_{sub})\:dx\:dt&=\int_{\mathscr{U}\setminus V_j}\varphi\left(\rho_k-\rho_{sub}\right)\:dx\:dt\\&
    \hspace{70pt}+\int_{0}^T\int_\cD \varphi\left(\rho^{(j)}_{sub,k}-\rho_{sub}\right)\:dx\:dt\\
    &\leq 2\varepsilon +o(1)
\end{align*}
as $k\rightarrow \infty$. Thus $\rho_k\rightharpoonup \rho_{sub}$ as $k\rightarrow\infty$.

The convergence $v_k\rightharpoonup v_{sub}$ follows similarly, since for fixed $j\geq 1$ there holds
\begin{align*}
    \int_0^T\int_\cD \varphi(v_k-&v_{sub})\:dx\:dt=\int_{\mathscr{U}\setminus V_j}\varphi\left(v_k-v_{sub}\right)\:dx\:dt+o(1)\\
    &\leq \norm{\varphi}_{L^2(\mathscr U\setminus V_j)}\left(\norm{v_k}_{L^2((0,T)\times\cD)}+\norm{v_{sub}}_{L^2((0,T)\times \cD)}\right)
    +o(1),
\end{align*}
and
\begin{align*}
    \norm{v_k}_{L^2((0,T)\times\cD)}^2&=\int_0^T\int_\cD n(e_0+\rho_k e_1)\:dx\:dt\\
    &\leq n\norm{e_0+\rho_{sub}e_1}_{L^1((0,T)\times\cD)}+2\int_0^T\delta(t)\:dt.
\end{align*}
Thus we have shown the convex integration Theorem \ref{thm:CI}.

\vspace{20pt}
\noindent Mathematisches Institut,  Universit\"at Leipzig,  Augustusplatz 10, D-04109 Leipzig \\
\texttt{bjoern.gebhard@math.uni-leipzig.de}\\
\texttt{jonas.hirsch@math.uni-leipzig.de}\\
\texttt{jozsef.kolumban@math.uni-leipzig.de}

\begin{thebibliography}{99}

\bibitem{arnold} V. Arnold, Sur la g\'{e}om\'{e}trie diff\'{e}rentielle des groupes de {L}ie de
              dimension infinie et ses applications \`a l'hydrodynamique des
              fluides parfaits, Ann. Inst. Fourier (Grenoble) 16 (1966), 319--361.

%\bibitem{santambrogio} M. Bernot, A. Figalli, F. Santambrogio, Generalized solutions for the Euler equations in one and two dimensions, J. Math. Pures Appl. 91 (2009) 137–155.

\bibitem{Brenier89} Y. Brenier, The Least Action Principle and the Related Concept of Generalized Flows for Incompressible Perfect Fluids, Journal of the American Mathematical Society 2, no. 2 (1989): 225–55.

%\bibitem{Brenier99} Y. Brenier, Minimal geodesics on groups of volume-preserving maps and generalized solutions of the {E}uler equations, Comm. Pure Appl. Math. 52.4 (1999), 411--452.

\bibitem{Brenier18} Y. Brenier, Some Concepts of Generalized and Approximate Solutions in Ideal Incompressible Fluid Mechanics Related to the Least Action Principle, Nečas Center Series, Springer International Publishing; 2018; 53--75.

\bibitem{Bronzi} A. C. Bronzi, M. C. Lopes Filho, H. J. Nussenzveig Lopes, Wild solutions for 2{D} incompressible ideal flow with passive tracer, Commun. Math. Sci. 13.5 (2015), 1333--1343.

\bibitem{Buckmaster_Vicol_survey} T. Buckmaster, V. Vicol, Convex integration constructions in hydrodynamics, Bull. Amer. Math. Soc. 58.1 (2021), 1--44.

\bibitem{Castro_Cordoba_Faraco} \'A. Castro, D. C\'ordoba, D. Faraco, Mixing solutions for the Muskat problem, Invent. Math. 226.1 (2021), 251--348. 

\bibitem{Castro_Faraco_Mengual} \'A. Castro, D. Faraco, F. Mengual, Degraded mixing solutions for the Muskat problem, Calc. Var. Partial Differential Equations 58.2 (2019).

\bibitem{Castro_Faraco_Mengual_2} \'A. Castro, D. Faraco, F. Mengual, Localized mixing zone for Muskat bubbles and turned interfaces, Ann. PDE 8.7 (2022).

\bibitem{Chae-Nam} D. Chae, H. Nam, Local existence and blow-up criterion for the {B}oussinesq equations, Proc. Roy. Soc. Edinburgh Sect. A 127.5 (1997), 935--946.

\bibitem{cheng} F. Cheng, On the dissipative solutions for the inviscid Boussinesq equations, AIMS Mathematics
2020, Volume 5, Issue 4: 2869-2876.

\bibitem{Chiodaroli_Kreml_energy_dissipation} E. Chiodaroli, O. Kreml, On the energy dissipation rate of solutions to the compressible isentropic {E}uler system, Arch. Ration. Mech. Anal. 214.3 (2014), 1019--1049.

\bibitem{Chiodaroli_Michalek} E. Chiodaroli, M. Mich\'alek, Existence and Non-uniqueness of Global Weak Solutions to Inviscid Primitive and Boussinesq Equations, Commun. Math. Phys. 353 (2017), 1201--1216.

\bibitem{colombo_figalli} M. Colombo, A. Figalli, Regularity results for very degenerate elliptic equations, J. Math. Pures Appl. (9), 101.1 (2014), 94--117.

\bibitem{Cordoba} D. C\'ordoba, D. Faraco, F. Gancedo, Lack of uniqueness for weak solutions of the incompressible porous media equation, Arch. Rat. Mech. Anal. 200.3 (2011), 725--746.

\bibitem{courant} R. Courant, Plateau’s Problem and Dirichlet’s Principle, Ann. of Math. 38.3 (1937), 679--724.

\bibitem{Dafermos_entropy_rate} C. M. Dafermos, The entropy rate admissibility criterion for solutions of hyperbolic conservation laws, J. Diff. Equ. 14 (1973), 202--212.

\bibitem{Danchin} R. Danchin, Remarks on the lifespan of the solutions to some models of incompressible fluid mechanics, Proc. Amer. Math. Soc. 141.6 (2013), 1979--1993.

\bibitem{DeL-Sz-Survey} C. De Lellis, L. Sz\'ekelyhidi Jr., Weak stability and closure in turbulence, Phil. Trans. R. Soc. A. 380.2218 (2022)

\bibitem{DeL-Sz-Annals} C. De Lellis, L. Sz\'ekelyhidi Jr., The Euler equations as a differential inclusion, Ann. Math. 170.3 (2009), 1417--1436.

\bibitem{DeL-Sz-Adm} C. De Lellis, L. Sz\'ekelyhidi Jr., On admissibility criteria for weak solutions
of the Euler equations, Arch. Rat. Mech. Anal. 195.1 (2010), 225--260.

\bibitem{desilva_savin} D. De Silva, O. Savin, Minimizers of convex functionals arising in random surfaces, Duke Math. J., 151.3 (2010), 487--532.

\bibitem{ebin_marsden} D. G. Ebin, J. Marsden, Groups of Diffeomorphisms and the Motion of an Incompressible Fluid, Ann. of Math. (2) 92 (1970), 102--163.

\bibitem{Elgindi} T. M. Elgindi, I. Jeong, Finite-time singularity formation for strong solutions to the {B}oussinesq system, Ann. PDE 6.1 (2020).

\bibitem{evans_p} L. C. Evans, A new proof of local {$C^{1,\alpha }$} regularity for solutions of certain degenerate elliptic p.d.e, J. Differential Equations 45.3 (1982), 356--373.

\bibitem{evans_m} L. C. Evans, R. F. Gariepy, Measure theory and fine properties of functions,  Chapman and Hall/CRC, 2015.

\bibitem{Evans_book} L. C. Evans, Partial Differential Equations 2nd Edition, American Mathematical Society (2010) 

\bibitem{Feireisl} E. Feireisl, Maximal dissipation and well-posedness for the compressible Euler system, J. Math. Fluid Mech. 16.3 (2014), 447--461.

\bibitem{Foerster-Sz} C. F\"orster, L. Sz\'ekelyhidi Jr., Piecewise constant subsolutions for the Muskat problem, Commun. Math. Phys. 363.3 (2018), 1051--1080.

\bibitem{GKBou} B. Gebhard, J. J. Kolumb\'an, Relaxation of the Boussinesq system and applications to the Rayleigh-Taylor instability,  Nonlinear Differential Equations and Applications 29, 7 (2022).

\bibitem{GKSz} B. Gebhard, J. J. Kolumb\'an, L. Sz\'ekelyhidi Jr., A new approach to the Rayleigh-Taylor
instability, Arch. Rat. Mech. Anal. 241 (2021), 1243--1280.

\bibitem{Gilbarg_Trudinger} D. Gilbarg, N. S. Trudinger, Elliptic Partial Differential Equations of Second Order, Springer-Verlag Berlin Heidelberg (2001)

\bibitem{Hitruhin-Lindberg} L. Hitruhin, S. Lindberg, The lamination convex hull of stationary incompressible porous media equations, SIAM J. Math. Anal. 53.1 (2021), 491--508.

\bibitem{kristensen_raita_lecture_notes} J. Kristensen, B. Rai\c{t}\u{a}: An introduction to generalized Young measures. Lecture notes 45, Max-Planck-Institut f\"ur Mathematik in den Naturwissenschaften Leipzig (2020)

\bibitem{lewis} J. L. Lewis, Regularity of the derivatives of solutions to certain    degenerate elliptic equations, Indiana Univ. Math. J. 32.6 (1983), 849--858.

\bibitem{Mengual} F. Mengual, H-principle for the 2D incompressible porous media equation with viscosity jump, Analysis \& PDE 15 (2022), 429--476.

\bibitem{Mengual_Sz_vortex_sheet} F. Mengual, L. Sz\'ekelyhidi Jr., Dissipative Euler flows for vortex sheet initial data without distinguished sign, Comm. Pure Appl. Math. (2022)

\bibitem{mooney} C. Mooney, Hilbert’s 19th problem revisited, Boll Unione Mat Ital (2022).

\bibitem{Noisette-Sz} F. Noisette, L. Sz\'ekelyhidi Jr., Mixing solutions for the Muskat problem with variable speed, J. Evol. Equ. 21 (2021), 3289--3312.

\bibitem{otto-ipm} F. Otto, Evolution of microstructure in unstable porous media flow: a relaxational approach. Comm. Pure Appl. Math 52.7 (1999), 873--915.

\bibitem{santambrogio_vespri} F. Sanatambrogio, V. Vespri, Continuity in two dimensions for a very degenerate elliptic equation, Nonlinear Anal. 73.12 (2010), 3832--3841.

\bibitem{savin} O. Savin, Small perturbation solutions for elliptic equations, Comm. Part. Diff. Equ. 32.4-6 (2007), 557--578.

\bibitem{shnirelman} A. I. Shnirelman, The geometry of the group of diffeomorphisms and the dynamics of an ideal incompressible fluid, Mat. Sb. (N.S.) 128.170 (1985), 82--109.

\bibitem{Sz-KH} L. Sz\'ekelyhidi Jr., Weak solutions to the incompressible Euler equations with vortex sheet initial data, C. R. Acad. Sci. Paris, Ser. I 349 (2011), 1063--1066.

\bibitem{Sz-Muskat} L. Sz\'ekelyhidi Jr., Relaxation of the incompressible porous media equation, Ann. Scient. \'Ec. Norm. Sup. 45.3 (2012), 491--509.

\bibitem{tolksdorf} P. Tolksdorf, Regularity for a more general class of quasilinear elliptic equations, J. Differential Equations 51.1 (1984), 126--150.

\bibitem{uhlenbeck} K. Uhlenbeck, Regularity for a class of non-linear elliptic systems, Acta Math., 138.3-4 (1977), 219--240.

\bibitem{uraltseva} N. N. Uraltseva, Degenerate quasilinear elliptic systems, Zap. Nau\v{c}n. Sem. Leningrad. Otdel. Mat. Inst. Steklov. (LOMI), 7 (1968), 184--222.

\bibitem{wang} L. Wang, Compactness methods for certain degenerate elliptic equations, J. Differential Equations 107.2 (1994), 341--350.

\bibitem{Wiedemann} E. Wiedemann, Weak-strong uniqueness in fluid dynamics, London Math. Soc. Lecture Note Ser. 452, Cambridge Univ. Press (2018), 289--326.

\bibitem{zalinescu} C. Z\u{a}linescu, Convex Analysis in General Vector Spaces, World Scientific Publishing (2002)

\end{thebibliography}
\end{document}